\documentclass[11pt]{amsart}

\newskip\stdskip
 \usepackage{amsfonts}
\usepackage{amsmath}
\usepackage{amssymb}
\usepackage{amstext}
\usepackage{amsthm}

\usepackage[all]{xy}

\usepackage[dvips]{epsfig}
\usepackage{pinlabel}
\usepackage{graphicx}
\usepackage[usenames,dvipsnames]{color}
\usepackage{enumitem}
\usepackage{mathtools}
\usepackage[margin=1in]{geometry}

\usepackage{hyperref}

\newcommand{\C}{\mathbb C}
\newcommand{\R}{\mathbb R}

\newcommand{\Z}{\mathbb Z}
\newcommand{\N}{\mathbb N}
\newcommand{\F}{\mathbb F}

\newcommand{\op}{\operatorname}
\newcommand{\bs}{\boldsymbol}
\newcommand{\fr}{\mathfrak{r}}
\newtheorem{thm}{Theorem}[section]
\newtheorem{lemma}[thm]{Lemma}
\newtheorem{cor}[thm]{Corollary}
\newtheorem{prop}[thm]{Proposition}

\theoremstyle{definition}
\newtheorem{dfn}[thm]{Definition}
\newtheorem{rem}[thm]{Remark}
\newtheorem{ex}[thm]{Example}

\begin{document}

\author[B. Chantraine]{Baptiste Chantraine}
\author[G. Dimitroglou Rizell]{Georgios Dimitroglou Rizell}
\author[P. Ghiggini]{Paolo Ghiggini}
\author[R. Golovko]{Roman Golovko}

\address{Universit\'e de Nantes, France.}
\email{baptiste.chantraine@univ-nantes.fr}

\address{Uppsala University, Sweden}
\email{georgios.dimitroglou@math.uu.se}

\address{Universit\'e de Nantes, France.}
\email{paolo.ghiggini@univ-nantes.fr}

\address{Charles University, Czech Republic.}
\email{golovko@karlin.mff.cuni.cz}

\subjclass[2010]{Primary 53D37; Secondary 53D40, 57R17.}

\keywords{Weinstein manifold, wrapped Fukaya category, generation, Lagrangian cocore
planes}

\title[The wrapped Fukaya category of Weinstein sectors]{Geometric generation of the
wrapped Fukaya category of Weinstein manifolds and sectors\\[1ex] Engendrement géométrique des catégories de Fukaya enroulées des variétés et secteurs de Weinstein}
\maketitle

\begin{abstract}
We prove that the wrapped Fukaya category of any $2n$-dimensional Weinstein manifold
(or, more generally, Weinstein sector) $W$
is generated by the unstable manifolds of the index $n$ critical points of its Liouville vector
field. Our proof is geometric in nature, relying on a surgery formula for Floer cohomology and
the fairly simple observation that Floer cohomology vanishes for Lagrangian submanifolds that
can be disjoined from the isotropic skeleton of the Weinstein manifold. Note that we do not need any additional assumptions on this skeleton. By applying our generation result to the diagonal in the product
$W \times W$, we obtain as a corollary that the open-closed map from the Hochschild
homology of the wrapped Fukaya category of $W$ to  its  symplectic cohomology is an
isomorphism, proving a conjecture of Seidel. We work mainly in the ``linear
setup'' for the wrapped Fukaya category, but we also extend the proofs to the ``quadratic'' and ``localisation'' setup. This is necessary for dealing with Weinstein sectors and for the applications.
\end{abstract}
\renewcommand{\abstractname}{Résumé}
\begin{abstract}
  Nous démontrons que la catégorie de Fukaya enroulée d'une variété (ou plus généralement d'un secteur) de Weinstein $W$ de dimension $2n$ est engendrée par les variétés instables des points critiques d'indice $n$ de son champ de Liouville. Notre preuve, de nature géométrique, repose sur une formule pour la cohomologie de Floer d'une chirurgie et sur l'observation relativement simple que la cohomologie de Floer d'une lagrangienne disjointe du squelette isotrope de la variété de Weinstein s'annule (aucune condition supplémentaire n'est demandée au squelette). En appliquant le critère d'engendrement au produit $W\times W$ nous obtenons en corollaire que l'application ouverte-fermée de la cohomologie de Hochschild de la catégorie de Fukaya enroulée de $W$ vers sa cohomologie symplectique est un isomorphisme, prouvant une conjecture de Seidel. Nous travaillons principalement avec la définition ``linéaire'' de la catégorie de Fukaya enroulée mais nous étendons les preuves aux définitions ``quadratique'' et ``par localisation''. Ces modifications sont nécessaires pour traiter secteurs de Weinstein et pour certaines applications.
\end{abstract}
\newpage
\setcounter{tocdepth}{1}
\renewcommand{\baselinestretch}{0}\normalsize
\tableofcontents
\renewcommand{\baselinestretch}{1}\normalsize

\section{Introduction}

The {\em wrapped Fukaya category} is an $A_\infty$-category associated to any Liouville manifold.
Its objects are exact Lagrangian submanifolds which are either compact or cylindrical at
infinity, possibly equipped with extra structure, the morphism spaces are wrapped Floer
chain complexes, and the $A_\infty$ operations are defined by counting perturbed holomorphic polygons \color{black}
with Lagrangian boundary conditions. Wrapped Floer cohomology was defined by A. Abbondandolo and M. Schwarz \cite{AbbonFloer}, at least for cotangent fibres, but the general definition and the chain level construction needed to define an $A_\infty$-category is due to M. Abouzaid and P. Seidel  \cite{OpenString}. The definition of the wrapped Fukaya category was further extended to the relative case by Z. Sylvan, who introduced the notions of {\em stop}  and {\em partially wrapped Fukaya category} in \cite{Sylvan}, and by S. Ganatra, J. Pardon and V. Shende, who later introduced the similar notion of {\em Liouville sector} in \cite{GanatraPardonShende}.

In this article we study the wrapped Fukaya category of Weinstein manifolds and sectors. In
the absolute case our main result is the following.
\begin{thm}\label{main}
If $(W, \theta, \mathfrak{f})$ is a $2n$-dimensional Weinstein manifold of finite type, then its wrapped Fukaya category  $\mathcal{WF}(W, \theta)$ is generated by the Lagrangian
cocore planes of the index $n$ critical points of $\mathfrak{f}$.
\end{thm}

In the relative case (i.e. for sectors) our main result is the following. We refer to Section
\ref{subsec:sectors} for the definition of the terminology used in the statement.
\begin{thm}\label{main-sectors}
The wrapped Fukaya category of the Weinstein sector  $(S, \theta, \mathfrak{f})$ is generated by the Lagrangian cocore planes of its completion $(W, \theta_W, \mathfrak{f}_W)$ and by the spreading of the Lagrangian cocore planes of its belt $(F, \theta_F, \mathfrak{f}_F)$.
\end{thm}
\begin{rem}
Exact Lagrangian submanifolds are often enriched with some extra structure: Spin structures,
grading or local systems. We ignore them for simplicity, but the same arguments carry over
also when that extra structure is considered.
\end{rem}

Generators of the wrapped Fukaya category are known in many particular cases. We will not
try to give a comprehensive overview of the history of this recent but active subject because we would not be able to make justice to everybody who has contributed to it. However, it
is important to mention that F. Bourgeois, T. Ekholm and Y. Eliashberg in \cite{EffectLegendrian} sketch a proof that the Lagrangian cocore discs split-generate the wrapped Fukaya category of a Weinstein manifold of finite type.
Split-generation is a weaker notion than generation,  which is
sufficient for most applications, but not for all; see for example
\cite{LazarevGeomAlgPres}. Moreover, \color{black} Bourgeois, Ekholm
and Eliashberg's proposed proof relies on their Legendrian surgery
formula, whose analytic details { are} not complete ({see
\cite{EkholmCurve} for recent development in that direction).}

Most generation results so far, including that of
Bourgeois, Ekholm and Eliashberg, rely on Abouzaid's split-generation criterion \cite{AbouzaidGeneration}. On the contrary, our proof is more direct and similar in
spirit to Seidel's proof in \cite{Seidel_Fukaya} that the Lagrangian thimbles generate the Fukaya-Seidel category of a Lefschetz fibration or to Biran and Cornea's cone decomposition of Arnol'd type Lagrangian cobordisms \cite{BirCo2}.  Theorems \ref{main} and \ref{main-sectors} have been proved independently also by Ganatra, Pardon and Shende in \cite[Theorem 1.9]{GanatraPardonShende2}.

A product of Weinstein manifolds is a Weinstein manifold. Therefore, by applying Theorem \ref{main} to the diagonal in a twisted product, and using
results of S. Ganatra \cite{Ganatra} and Y. Gao \cite{Gao-product}, we obtain the following result.
\begin{thm}\label{thm: HH and SH}
Let $(W, \theta, \mathfrak{f})$ be a Weinstein manifold of finite type. Let ${\mathcal D}$
the full $A_\infty$ subcategory of ${\mathcal WF}(W, \theta)$ whose objects are the Lagrangian cocore planes. Then the open-closed map
\begin{equation}\label{eqn: open-closed}
{\mathcal OC} \colon HH_*({\mathcal D}, {\mathcal D}) \to SH^*(W)
\end{equation}
is an isomorphism.
\end{thm}
In Equation \eqref{eqn: open-closed} $HH_*$ denotes Hochschild homology, $SH^*$ denotes
symplectic cohomology and ${\mathcal OC}$ is the open-closed map defined in \cite{AbouzaidGeneration}. Theorem \ref{thm: HH and SH} in particular proves that
\begin{equation}\label{eqn: open-closed2}
{\mathcal OC} \colon HH_*({\mathcal WF}(W, \theta), {\mathcal WF}(W, \theta)) \to SH^*(W)
\end{equation}
is an isomorphism. This proves a conjecture of Seidel in \cite{SeidelConjecture} for Weinstein manifolds of finite type. Note that a proof of this conjecture, assuming the Legendrian surgery
formula of Bourgeois, Ekholm and Eliashberg was given by S. Ganatra and  M. Maydanskiy in the appendix of \cite{EffectLegendrian}.

The above result implies in particular that Abouzaid's generation criterion \cite{AbouzaidGeneration} is satisfied for the subcategory consisting of the cocore planes of a Weinstein manifold, from which one can conclude that the cocores split-generate the wrapped Fukaya category. In the exact setting under consideration this of course follows a fortiori from Theorem \ref{main}, but there are extensions of the Fukaya category in which this generation criterion has nontrivial implications. Notably, this is the case for the version of the wrapped Fukaya category for monotone Lagrangians, as we proceed to explain.

The wrapped Fukaya category as well as symplectic cohomology were defined in the monotone symplectic setting in \cite{RitterSmith} using coefficients in the Novikov field. When this construction is applied to exact Lagrangians in an exact symplectic manifold, a change of variables $x \mapsto t^{-\mathcal{A}(x)}x$, where $\mathcal{A}(x)$ is the action of the generator $x$ and $t$ is the formal Novikov parameter,  allows for an identification of the Floer complexes and the open-closed map  with the original complexes and map tensored with the Novikov field. \color{black} The generalisation of Abouzaid's generation criterion to the monotone setting established in \cite{RitterSmith} thus shows that
\begin{cor}
The wrapped Fukaya category of monotone Lagrangian submanifolds of a Weinstein manifold  which are unobstructed in the strong sense (i.e.~ with $\mu^0=0$, where $\mu^0$ is the number of Maslov index two holomorphic discs passing through a generic point) is split-generated by the Lagrangian cocore planes of the Weinstein manifold.
\end{cor}
\begin{rem}
The strategy employed in the proof of Theorem \ref{main} for showing generation fails for non-exact Lagrangian submanifolds in two crucial steps: in Section \ref{sec: trivial triviality} and Section \ref{sec: surgeries and cones}. First, there are well known examples of unobstructed monotone Lagrangian submanifolds in a Weinstein manifold which are Floer homologically nontrivial even if they are disjoint from the skeleton. Second, our treatment of Lagrangian surgeries requires that we lift the Lagrangian submanifolds in $W$ to Legendrian submanifolds of $W \times \R$, and this is possible only for exact Lagrangian submanifolds. It is unclear to us whether it is true that the cocores generate (and not merely split-generate) the $\mu^0=0$ part of the monotone wrapped Fukaya category.
\end{rem}

\subsection{Comparison of setups}
There are three ``setups'' in which the wrapped Fukaya category is defined: the ``linear setup'', where the Floer equations are perturbed by Hamiltonian functions with linear growth at infinity and the wrapped Floer chain complexes are defined as homotopy limits over Hamitonians with higher and higher slope, the ``quadratic setting'', where the Floer equations are perturbed by Hamiltonian functions with quadratic growth at infinity, and the ``localisation setting'', where the Floer equations are unperturbed and the wrapped Fukaya category is defined by a categorical construction called  {\em localisation}. The linear setup was introduced by Abouzaid and Seidel in \cite{OpenString} and the quadratic setup by  Abouzaid in \cite{AbouzaidGeneration}. The latter is used also in Sylvan's definition of the partially wrapped Fukaya category and in the work of Ganatra \cite{Ganatra} and Gao \cite{Gao-product} which we use in the proof of Theorem \ref{thm: HH and SH}. The localisation setup is used in \cite{GanatraPardonShende} because the linear and quadratic setups are not available on sectors for technical reasons.
All three setups are expected to produce equivalent $A_\infty$ categories on Liouville manifolds.
We chose to work in the linear setup for the proof of Theorem \ref{main}, but our proof can be adapted fairly easily to the other setups. Moreover these extensions are necessary to prove Theorem \ref{main-sectors} and Theorem \ref{thm: HH and SH}. We will describe the small modifications we need in the localisation setup in Section \ref{sec:wrapp-fukaya-categ} and those we need in the quadratic setup in Section \ref{hochschild and symplectic}.

\subsection{Strategy of the proof of Theorem 1.1}
The strategy of the proof of Theorem \ref{main} is the following. Given a cylindrical Lagrangian submanifold $L$, by a compactly supported Hamiltonian isotopy we make it transverse to the stable manifolds of the Liouville
flow.  Thus, for dimensional reasons, \color{black} it will be disjoint from the stable manifolds of the critical points of index less than $n$ and will intersect the stable manifolds of the critical points of index $n$ in finitely many points $a_1, \ldots, a_k$. For each $a_i$ we consider a Lagrangian plane $D_{a_i}$ passing through $a_i$, transverse both to $L$ and to the stable manifold, and
Hamiltonian isotopic to the unstable manifold of the same critical point.
We assume that \color{black} the Lagrangian planes are all pairwise disjoint. The unstable manifolds of the index $n$ critical points are what we call the Lagrangian cocore planes.

At each $a_i$ we perform a Lagrangian surgery between $L$ and $D_{a_i}$ so that the resulting Lagrangian $\overline{L}$ is disjoint from the skeleton of $W$. Since $\overline{L}$
will be in general immersed, we have to develop a version of wrapped Floer cohomology for
immersed Lagrangian submanifolds. To do that we borrow heavily from the construction of Legendrian contact cohomology in \cite{LCHgeneral}. In particular our wrapped Floer cohomology between
immersed Lagrangian submanifolds uses augmentations of the Chekanov-Eliashberg algebras of the Legendrian lifts as bounding cochains. A priori there is no reason why such a bounding cochain
should exists for $\overline{L}$, but it turns out that we can define it inductively provided that  $D_{a_1}, \ldots, D_{a_k}$ are isotoped in a suitable way. A large part of the technical work in this paper is devoted to the proof of this claim.

Then we prove a correspondence between twisted complexes in the wrapped Fukaya category and Lagrangian surgeries by realising a Lagrangian surgery as a Lagrangian cobordism between the Legendrian lifts  and applying the Floer theory for Lagrangian cobordisms we defined in \cite{Floer_cob}. This result can have an independent interest.
Then we can conclude that $\overline{L}$ is isomorphic, in an appropriated triangulated completion of the wrapped Fukaya category, to a twisted complex $\mathbb{L}$ built from $L, D_{a_1}, \ldots, D_{a_k}$.

Finally, we prove that the wrapped Floer cohomology of $\overline{L}$ with any other cylindrical Lagrangian is trivial. This is done by a fairly simple action argument based on
the fact that the Liouville flow displaces $\overline{L}$ from any compact set because
$\overline{L}$ is disjoint from the skeleton of $W$. Then the twisted complex $\mathbb{L}$ is a trivial object, and therefore some simple homological algebra shows that $L$ is isomorphic
to a twisted complex built from $D_{a_1}, \ldots, D_{a_k}$.

This article is organised as follows. In Section \ref{sec: setup} we recall some generalities about Weinstein manifolds and sectors. In Section \ref{sec: lch} we recall the definition and the basic properties of Legendrian contact cohomology. In Sections \ref{sec: Floer homology for immersions} and \ref{sec: continuation maps} we define the version of Floer cohomology for Lagrangian immersions that we will use in the rest of the article. Despite their length, these sections contain mostly routine verifications and can be skipped by the readers who are willing to accept that such a theory exists. In Section \ref{sec: immersed wrapped} we define wrapped Floer cohomology for Lagrangian immersions using the constructions of the previous two sections. In Section \ref{sec: trivial triviality} we prove that an immersed Lagrangian submanifold which is disjoint from the skeleton is Floer homologically trivial. In Section \ref{sec: surgeries and cones} we prove that Lagrangian surgeries correspond to twisted complexes in the wrapped Fukaya category.
In Section \ref{sec: generation} we finish the proof of Theorem \ref{main} and in particular we construct the bounding cochain for $\overline{L}$.  In Section \ref{sec:wrapp-fukaya-categ} we prove Theorem \ref{main-sectors}.  We briefly recall  the construction of the wrapped Fukaya category for sectors in the localisation setup from \cite{GanatraPardonShende} and show how all previous arguments adapt in this setting. Finally, in Section \ref{hochschild and symplectic} we prove the isomorphism between Hochschild homology and symplectic cohomology.  This requires that we adapt the proof of Theorem \ref{main} to the quadratic setting.

\subsection*{Acknowledgements}
We thank Tobias Ekholm for pushing us to extend our original result to Weinstein sectors
and Yasha Eliashberg for a useful conversation about Weinstein pairs. We also thank an anonymous referee for pointing out that it was necessary to use the localisation setup to define the wrapped Fukaya category of Liouville sectors.

\small{The third author would like to thank the Isaac Newton Institute for Mathematical Sciences
for support and hospitality during the program {\em Homology theories in low dimensional
topology} when work on this paper was undertaken. This work was supported by EPSRC Grant Number EP/K032208/1. The first author is partially supported by the ANR project COSPIN (ANR-13-JS01-0008-01). The second author is supported by the grant KAW 2016.0198 from the Knut and Alice Wallenberg Foundation. In addition, he would like to express his gratitude to Universit\'{e} de Nantes where he was professeur invit\'e in October 2016 when this project was started. The fourth author is supported by the ERC Consolidator Grant 646649 ``SymplecticEinstein'',  GA\v{C}R EXPRO Grant 19-28628X and 4EU+/20/F3/21 Minigrant. Finally, the first and third authors are partially supported by the ANR project QUANTACT (ANR-16-CE40-0017).
}

\normalsize{}
\section{Geometric setup}
\label{sec: setup}
In this section we revise some elementary symplectic geometry with the purpose of fixing
notation and conventions.
\subsection{Liouville manifolds}\label{ss: liouville}
Let $(W, \theta)$ be a Liouville manifold of finite type, from now on called simply a
Liouville manifold. This means that $d \theta$ is a
symplectic form, the Liouville vector field ${\mathcal L}$ defined by the equation
$$\iota_{\mathcal L} d \theta = \theta$$
is complete and, for some $R_0 < 0$, there exists a proper smooth function
$\mathfrak{r} \colon W \to [R_0, + \infty)$  such that, for $w \in W$,
\begin{itemize}
\item[(i)] $d_w\fr ({\mathcal L}_w) >0$ if $\fr(w)>R_0$, and
\item[(ii)] $d \fr_w({\mathcal L}_w)=1$ if $\fr(w) \ge R_0+1$.
\end{itemize}
In particular, $R_0$ is the unique critical value of $\fr$, which is of course highly
nondegenerate, and every other  level set is a contact type hypersurface.

We use the function $\fr$ to define some useful subsets of $W$.
\begin{dfn}
For every $R \in [R_0, + \infty)$ we denote $W_R= \fr^{-1}([R_0, R])$, $W_R^e =
W \setminus  \mathrm{int}(W_R)$ and $V_R = \fr^{-1}(R)$.
\end{dfn}
The subsets $W_R^e$ will be called the {\em ends} of $W$. The Liouville flow of
$(W, \theta)$ induces an identification
\begin{equation}\label{end as symplectisation}
([R_0+1, + \infty) \times V, e^r \alpha) \cong (W_{R_0+1}^e, \theta),
\end{equation}
where $V=V_0$ and $\alpha$ is the pull-back of $\theta$ to $V_0$. More precisely,
if $\phi$ denotes the flow of the Liouville vector field, the identification
\eqref{end as symplectisation} is given by $(r, v) \mapsto \phi_r(v)$. Let $\xi = \ker \alpha$
be the contact structure defined by $\alpha$. Every $V_R$, for $R>R_0$, is contactomorphic
to $(V, \xi)$.  Under the identification \eqref{end as symplectisation}, the
function $\mathfrak{r}$, restricted to $W_{R_0+1}^e$, corresponds to the projection to $[R_0+1,+ \infty)$ in the sense that the following diagram commutes
$$\xymatrix{
([R_0+1, + \infty) \times V \ar[rr]^\phi \ar[dr]& & W_{R_0+1}^e \ar[dl]_{\mathfrak{r}}.\\
& [R_0+1, + \infty) &
}$$
\begin{rem}
\label{rem: change of r}
The choice of $R_0$ in the definition of $\mathfrak{r}$ is purely arbitrary because the
Liouville flow is complete. In fact, for every map $\mathfrak{r} \colon W \to [R_0, + \infty)$
as above and for any $R_0' < R_0$ there is  a map $\mathfrak{r}' \colon W \to [R_0', + \infty)$
satisfying (i) and (ii), which moreover coincides with $\mathfrak{r}$ on $\mathfrak{r}^{-1}(
[R_0+1, + \infty))$.
\end{rem}

A diffeomorphism $\psi \colon W \to W$ is an {\em exact symplectomorphism} if
$\psi^* \theta = \theta + dq$ for some function $q \colon W \to \R$. Flows of Hamiltonian
vector fields are, of course, the main source of exact symplectomorphisms. Given a function
$H \colon [- t_-, t_+] \times W \to \R$, where $t_{\pm}\ge 0$ and are allowed to be infinite, we
define the Hamiltonian vector field $X_H$ by
$$\iota_{X_H} d \theta = -dH.$$
Here $dH$ denotes the differential in the directions tangent to $W$, and therefore
$X_H$ is a time-dependent vector field on $W$.

We spell out the change in the Liouville form induced by a Hamiltonian flow because it is a
computation that will be needed repeatedly.
\begin{lemma}\label{variation of Liouville}
Let $H \colon [- t_-, t_+] \times W \to \R$ be a Hamiltonian function and
$\varphi_t$ its Hamiltonian flow.
Then, for all $t \in [t_-, t_+]$, we have $\varphi_t^* \theta = \theta+ dq_t$, where
$$q_t=  \int_0^t (- H_\sigma + \theta(X_{H_\sigma})) \circ \varphi_\sigma  d\sigma.$$
\end{lemma}
\begin{proof}
We compute
\begin{align*}
&\varphi_t^*\theta - \theta = \int_0^t \frac{d}{d\sigma}(\varphi_\sigma^* \theta)
d\sigma =  \int_0^t \varphi_\sigma^*(L_{X_{H_\sigma}} \theta) d\sigma =  \\ & \int_0^t
\varphi_\sigma^* (\iota_{X_{H_\sigma}} d \theta + d \iota_{X_{H_\sigma}} \theta)
d\sigma = \int_0^t \varphi_\sigma^*  (- dH_\sigma + d (\theta(X_{H_\sigma}))) d\sigma.\qedhere
\end{align*}
\end{proof}
\subsection{Weinstein manifolds}\label{ss: Weinstein manifolds}
In this article we will be concerned mostly with Weinstein manifolds of finite type.
We recall their definition, referring to \cite{SteinWeinstein} for further details.
\begin{dfn}\label{dfn:weinstein}
A Weinstein manifold $(W, \theta, \mathfrak{f})$ consists of:
\begin{itemize}
\item[(i)] an even dimensional smooth manifold $W$ without boundary,
\item[(ii)] a one-form $\theta$ on $W$ such that $d \theta$ is a symplectic form and the
Liouville vector field ${\mathcal L}$ associated to $\theta$ is complete, and
\item[(iii)] a proper Morse function $\mathfrak{f} \colon W \to \R$ bounded from below such that
${\mathcal L}$ is a pseudogradient of $\mathfrak{f}$ in the sense of
\cite[Equation (9.9)]{SteinWeinstein}:
i.e.\
$$ d \mathfrak{f}(\mathcal L)   \geq \delta(\|{\mathcal L} \|^2 + \| d \mathfrak{f} \|^2),$$
where $\delta>0$ and the norms are computed with respect to some
Riemannian metric on $W$.
\end{itemize}
The function $\mathfrak{f}$ is called a \emph{Lyapunov function} (for $\mathcal L$).

If $\mathfrak{f}$ has finitely many critical points, then $(W, \theta, \mathfrak{f})$ is
a Weinstein manifold of finite type. From now on, Weinstein manifold will always mean
Weinstein manifold of finite type.
\end{dfn}

Given a regular value $M$ of $\mathfrak{f}$ the compact manifold $\{\mathfrak{f}\leq M\}$ is called a \textit{Weinstein domain}. Any Weinstein domain can be completed to a Weinstein manifold in a canonical way by adding half a symplectisation of the contact boundary.

By Condition (iii) in Definition \ref{dfn:weinstein}, the zeroes of $\mathcal{L}$ coincide with the critical points of $\mathfrak{f}$. If $W$ has dimension $2n$, the critical points of $\mathfrak{f}$ have index at most $n$.
For each critical point $p$ of $\mathfrak{f}$ of index $n$, there is a stable manifold
$\Delta_p$  and an unstable manifold $D_p$ which are both exact Lagrangian submanifolds.
We will call the unstable manifolds $\Delta_p$ of the critical points of index $n$ the {\em Lagrangian cocore planes}.

\begin{dfn}\label{dfn:skeleton}
Let $W_0 \subset W$ be a Weinstein domain containing all critical points of $\mathfrak{f}$. The {\em Lagrangian skeleton} of $(W, \theta, \mathfrak{f})$ is the attractor of the negative
flow of the Liouville vector field on the compact part of $W$, i.e.
$$W^{\mathrm{sk}}:=\mathop{\bigcap}\limits_{t>0}\phi^{-t}(W_0),$$
where $\phi$ denotes the flow of the Liouville vector field ${\mathcal L}$.
Alternatively, $W^{\mathrm{sk}}$ can be defined as the union of unstable manifolds of all critical points of $\mathfrak{f}$.
\end{dfn}
The stable manifolds of the index $n$ critical points form the top dimensional stratum of the Lagrangian skeleton.

A Morse function gives rise to a handle decomposition. In the case of a Weinstein manifold
$(W, \theta, \mathfrak{f})$, the handle decomposition induced by $\mathfrak{f}$ is
compatible with the symplectic structure and is called the Weinstein handle decomposition
of $(W, \theta, \mathfrak{f})$. By the combination of \cite[Lemma~12.18]{SteinWeinstein}
and \cite[Corollary~12.21]{SteinWeinstein} we can assume that
${\mathcal L}$ is Morse-Smale.  This implies that we can assume that handles of
higher index are attached after handles of lower index. The deformation making
${\mathcal L}$ Morse-Smale can be performed without changing the symplectic form
$d \theta$ and so that the unstable manifolds corresponding to the critical points of index
$n$ before and after such a deformation are Hamiltonian isotopic.

We will denote the union of the handles of index strictly less than $n$ by $W^{sc}$.
This will be called the {\em subcritical subdomain} of $W$. By construction,
$\partial W^{sc}$ is a contact type hypersurface in $W$.

We choose $\fr \colon W \to [R_0, + \infty),$  and we homotope the Weinstein structure so that
$$W_{R_0} = W^{sc} \cup {\mathcal H}_1 \cup \ldots \cup {\mathcal H}_l,$$
where ${\mathcal H}_1, \ldots {\mathcal H}_l$ all are \emph{standard} Weinstein handles corresponding to the
critical points $p_1, \ldots, p_l$ of $\mathfrak{f}$ of Morse index $n$; see \cite{WeinsteinHandles} for the description of the standard model, and \cite[Section 12.5]{SteinWeinstein} for how to produce the Weinstein homotopy.
\begin{rem}
We could easily modify $\mathfrak{f}$ so that it agrees with $\fr$ on $W_0^e$. However,
this will not be necessary.
\end{rem}
The {\em core} of the Weinstein handle ${\mathcal H}_i$ is the Lagrangian disc $C_i=
\Delta_{p_i} \cap {\mathcal H}_i$. Let $D_\delta T^*C_i$ denote the disc cotangent bundle
of $C_i$ of radius $\delta >0$. By the Weinstein neighbourhood theorem, there is a
symplectic identification ${\mathcal H}_i \cong D_\delta T^*C_i$ for some $\delta$.
However, the restriction of $\theta$ to ${\mathcal H}_i$ does not correspond to the
restriction of the canonical Liouville form to  $D_\delta T^*C_i$.

\subsection{Weinstein sectors}
\label{subsec:sectors}
In this section we introduce Weinstein sectors. These will be particular cases of Liouville sectors as defined in \cite{GanatraPardonShende} characterised, roughly speaking, by retracting over a Lagrangian skeleton with boundary. In Section \ref{subsec:pairtosector} below we will then show that any Weinstein pair as introduced in \cite[Section 2]{Eliash_Wein_Revi} can be completed  to a Weinstein sector.

\begin{dfn}\label{dfn:weinsect}
A \emph{Weinstein sector} $(S,\theta, I, \mathfrak{f})$ consists of:
\begin{enumerate}
\item an even dimensional smooth manifold with boundary $S$;
\item a one-form $\theta$ on $S$ such that $d \theta$ is a symplectic form and the associated Liouville vector field ${\mathcal L}$ is complete and everywhere tangent to $\partial S$;
\item  a smooth function $I \colon \partial S \to \R$
    which satisfies
\begin{enumerate}
\item $dI({\mathcal L})= \alpha I$ for some function $\alpha\colon \partial S \to \R_+$ which is constant outside of a compact set and
\item $dI(C)>0$, where $C$ is a tangent vector field on $\partial S$ such that $\iota_C d \theta|_{\partial S}=0$ and $d\theta(C,N)>0$ for an outward pointing normal vector field $N$;
\end{enumerate}
\item a proper Morse function $\mathfrak{f} \colon S \to \R$ bounded from below having finitely many critical points, such that ${\mathcal L}$ is a pseudogradient of $\mathfrak{f}$ and satisfying moreover
\begin{enumerate}
\item $d \mathfrak{f}(C)>0$ on $\{ I >0 \}$ and $d \mathfrak{f}(C)<0$ on $\{ I <0 \}$,
\item the Hessian of a critical point of $\mathfrak{f}$ on $\partial S$ evaluates negatively on
the normal direction $N$, and
\item there is a constant $c \in \R$ whose sublevel set satisfies $\{ \mathfrak{f} \le c \} \subset S \setminus \partial S$ and contains all interior critical points of $\mathfrak{f}$.
\end{enumerate}
\end{enumerate}
\end{dfn}

For simplicity we will often drop part of the data from the notation. We will always assume that $S$ is a Weinstein sector of {\em finite type}, i.e.\ that $\mathfrak{f}$ has only finitely many critical points. A Weinstein sector is a particular case of an exact Liouville sector in the sense of \cite{GanatraPardonShende}.

 \begin{ex}
After perturbing the canonical Liouville form, the cotangent bundle of a smooth manifold $Q$ with boundary admits the structure of a Weinstein sector.
\end{ex}
To a Weinstein sector $(S, \theta, I, \mathfrak{f})$ we can associate
two Weinstein manifolds in a canonical way up to deformation: the
\emph{completion} and the \emph{belt}. The completion of $S$ is the Weinstein manifold  $(W, \theta_W, \mathfrak{f}_W)$ obtained by completing the Weinstein domain $W_0 = \{ \mathfrak{f} \le c \}$, which contains all interior critical points of $\mathfrak{f}$. The belt of $S$ is the Weinstein manifold $(F, \theta_F, \mathfrak{f}_F)$ where $F=I^{-1}(0)$, $\theta_F= \theta|_{F}$\footnote{We abuse the notation by denoting the pull back by the inclusion as a restriction.} and $\mathfrak{f}_F= \mathfrak{f}|_{F}$. To show that the belt is actually a Weinstein manifold it is enough to observe that $d \theta_F$ is a symplectic form because $F$ is transverse to the vector field $C$, and that the Liouville vector field ${\mathcal L}$ is tangent to $F$ because $dI({\mathcal L}) = \alpha I$, and therefore the Liouville vector field of $\theta_F$ is ${\mathcal L}_F={\mathcal L}|_F$.

Let $\kappa \in \R$ be a number such that all critical points of $\mathfrak{f}$ are contained in $\{ \mathfrak{f} \le \kappa \}$. We denote $S_0 = \{ \mathfrak{f} \le \kappa \}$ and $F_0=
F \cap S_0 = \{ \mathfrak{f}_F \le \kappa \}$. By Condition (4a) of Definition \ref{dfn:weinsect}, the boundary $\partial S_0$ is a contact manifold with convex boundary
with dividing set $\partial F_0$. Moreover $S \setminus S_0$ can be identified to a half symplectisation. Thus, given $R_0 \ll 0$, we can define a function  $\mathfrak{r} \colon S \to [R_0,+\infty)$ satisfying the properties analogous to those in Section \ref{ss: liouville}. We then write $S_R \coloneqq \mathfrak{r}^{-1}[R_0,R]$ and $S_R^e = S \setminus \mathrm{int}(S_R)$.

\begin{dfn}
Let $\phi$ be the flow of ${\mathcal L}$.
The \emph{skeleton} $S^{\mathrm{sk}} \subset S$ of a Weinstein sector $(S,\theta,\mathfrak{f})$ is given by
$$ S^{\mathrm{sk}} \coloneqq \bigcap_{t >0} \phi^{-t}(S_0).$$
\end{dfn}

\begin{rem}
Let $W$ and $F$ the completion and the belt, respectively, of the Weinstein sector $S$.
To understand the skeleton $S^{\mathrm{sk}}$ it is useful to note the folowing:
\begin{enumerate}
\item critical points of $\mathfrak{f}$ on $\partial S$ are also critical points of $\mathfrak{f}|_{\partial S}$ and vice versa,
\item any critical point $p \in \partial S$ of $\mathfrak{f}$ lies inside $\{I=0\} = F$  and is also a critical point of $\mathfrak{f}_F$,
 \item the Morse indices of the two functions satisfy
  the relation
$$\op{ind}_{\mathfrak{f}}(p)=\op{ind}_{\mathfrak{f}_F}(p)+1,$$
\item the skeleton satisfies $ S^{\mathrm{sk}} \cap \partial S = F^{\mathrm{sk}}$.
\end{enumerate}
\end{rem}

The top stratum of the skeleton of $(S,\theta,\mathfrak{f})$ is given by the union of the
stable manifolds of the critical points of $\mathfrak{f}$ of index $n$, where $2n$ is the dimension of $S$. Those are of two types: the stable manifolds $\Delta_p$ where $p$ is an interior critical point of $\mathfrak{f}$, which are also stable manifolds for $\mathfrak{f}_W$ in the completion, and the stable manifolds $\Theta_p$ where $p$ is a boundary critical point of $\mathfrak{f}$,  for which   $\Delta_p' = \Theta_p \cap \partial S$ is the stable manifold of $p$ for $\mathfrak{f}_F$ in $F$.

Thus the Weinstein sector $S$ can be obtained by attaching Weinstein handles, corresponding to the critical points of $\mathfrak{f}$ in the interior of $\partial S$, and Weinstein half-handles, corresponding to the critical points of $\mathfrak{f}$ in the boundary $\partial S$. We denote by $S^{sc}$ the subcritical part of $S$, i.e.\ the union of the handles and half-handles of index less than $n$ (where $2n$ is the dimension of $S$), by  $\{\mathcal{H}_i\}$ the critical handle corresponding to $\Delta_i$ and by $\{\mathcal{H}'_j\}$ the half-handle corresponding to $\Theta_j$. Finally we also choose the function $\mathfrak{r} \colon S\rightarrow [R_0,\infty)$ as in Section \ref{ss: liouville} which furthermore satisfies
$$S_{R_0}=S^{sc}\cup \mathcal{H}_1\cup\ldots\cup \mathcal{H}_l\cup \mathcal{H}'_1\cup\ldots\cup \mathcal{H}'_{l'}$$
and modify the Liouville form $\theta$ so that $\mathcal{H}_1, \ldots, \mathcal{H}_l, \mathcal{H}'_1,\ldots, \mathcal{H}'_{l'}$ are {\em standard} Weinstein handles.

It follow the from the symplectic standard neighbourhood theorem that a collar neighbourhood of $\partial S$ is symplectomorphic to
$$ (F \times T^*(-2\epsilon,0],d\theta_F + dp\wedge dq).$$

\begin{dfn}\label{dfn:spreading}
Let $S$ be a Weinstein sector and let $L$ be a Lagrangian submanifold of its belt $F$.
The {\em spreading} of $L$ is
$$\operatorname{spr}(L)=L\times T^*_{-\epsilon}(-2\epsilon, 0] \subset
F\times T^*(-2\epsilon,0)\subset S.$$
\end{dfn}

\begin{rem}
 The spreading of $L$ depends on the choice of symplectic standard neighbourhood of the collar. However, given two different choices, the corresponding spreadings are Lagrangian isotopic. Furthermore, if $L$ is exact in $F$, then $\operatorname{spr}(L)$ is exact in $S$, and thus two different spreadings are Hamiltonian isotopic.
\end{rem}

\begin{ex}
When the Weinstein sector is the cotangent bundle of a manifold with boundary, the spreading of a cotangent fibre of $T^*\partial Q$ is simply a cotangent fibre of $T^*Q$.
\end{ex}

The proof of the following lemma is immediate.
\begin{lemma}
The cocore planes of the index $n$ half-handles of $S$ are the spreading of the cocore planes of the corresponding index $n-1$ handles of $F$.
\end{lemma}

\subsection{Going from a Weinstein pair to a Weinstein sector}
\label{subsec:pairtosector}
In this section we describe how to associate a Weinstein sector to a Weinstein pair. We recall the definition of Weinstein pair, originally introduced in \cite{Eliash_Wein_Revi}.

\begin{dfn}\label{dfn: Weinstein pair}
A {\em Weinstein pair} $(W_0, F_0)$ is a pair of Weinstein domains $(W_0,\theta_0,\mathfrak{f}_0)$ and $(F_0,\theta_F,\mathfrak{f}_F)$ together with a codimension one Liouville embedding of $F_0$ into $\partial W_0$.
\end{dfn}

We denote the completions of $(W_0, \theta_0, \mathfrak{f}_0)$ and $(F_0, \theta_F, \mathfrak{f}_F)$ by $(W, \theta_0, \mathfrak{f}_0)$ and $(F, \theta_F, \mathfrak{f}_F)$ respectively. Let $F_1 \subset F$ be a Weinstein domain retracting on $F_0$
If $F_1$ is close enough to $F_0$,
the symplectic standard neighbourhood theorem provides us with a Liouville embedding
\begin{equation} \label{neighbourhood of F_0}
((1-3\epsilon, 1] \times [-3\delta,3\delta] \times F_1, sdu+s\theta_F) \hookrightarrow (W_0, \theta_0).
\end{equation}
Here $s$ and $u$ are coordinates on the first and second factors, respectively, and we require that the preimage of $\partial W_0$ is $\{s=1\}$
and $F_1 \subset \partial W_0$ is identified with $\{(1,0,x) :  x\in F_1\}$.
We denote $\mathcal{U}$ the image of the embedding \eqref{neighbourhood of F_0}.
After deforming $\mathfrak{f}_0$ we may assume that it is of the form $\mathfrak{f}_0(s,u,x)=s$ in the same coordinates.

Let ${\mathcal L}_F$ be the Liouville vector field of $(F, \theta_F)$. We choose a smooth function
$\mathfrak{r}_F \colon F \to [R_0,+\infty)$, $R_0 \ll 0$, such that
\begin{itemize}
\item $F_0=\mathfrak{r}_F^{-1}([R_0, 0])$,
\item $d \mathfrak{r}_F ({\mathcal L}_F) > 0$ holds inside $\mathfrak{r}_F^{-1}(R_0,+\infty)$, and
\item $d \mathfrak{r}_F({\mathcal L}_F) = 1$ holds inside $\mathfrak{r}_F^{-1}(R_0+1,+\infty)$.
\end{itemize}
For simplicity of notation we  also assume that
\begin{itemize}
\item  $F_1=\mathfrak{r}_F^{-1}([R_0, 1])$,
\end{itemize}
where $F_1$  denotes the manifold appearing in Formula \eqref{neighbourhood of F_0}. This condition is apparently a loss of generality because it cannot be satisfied for every Liouville form on $F$. However, the general case can be treated with minimal changes.

Consider the smooth function
\begin{gather*}
r \colon [-3 \delta, 3 \delta] \times F_1 \to \R,\\
r(u,x) \coloneqq 2- \left (\frac{u}{3\delta} \right)^2 - \mathfrak{r}_F(x)-c
\end{gather*}
for some small number $c>0.$
\begin{lemma}
There exists a Weinstein domain $\widetilde{W}_0 \subset W$ containing all critical points of $\mathfrak{f}_0$ and which intersects some collar neighbourhood of $W_0$ containing ${\mathcal U}$ precisely in the subset
$${\mathcal C} \coloneqq \{ s \le r(u,x)\} \subset {\mathcal U}.$$
\end{lemma}
The goal is now to deform the Liouville form $\theta_0$ on
$$ S_0 \coloneqq W_0 \cap \widetilde{W}_0 $$
to obtain a Liouville form $\theta$ so that the completion of $(S_0, \theta)$ is the sought Weinstein sector. The deformation will be performed inside $\mathcal{C}$. Given a smooth function $\rho \colon [1-3 \epsilon, 1] \to \R$ such that:
\begin{itemize}
\item $\rho(s)=0$ for $s \in [1-3 \epsilon, 1-2 \epsilon]$,
\item $\rho(s)= 2s-1$ for $s \in [1-\epsilon, 1]$, and
\item $\rho'(s) \ge 0$ for $s \in [1-3 \epsilon, 1]$,
\end{itemize}
we define a Liouville form $\theta_{\mathcal U}$ on $\mathcal{U}$ by
$$\theta_{\mathcal U} = s(du+\theta_F)-d(\rho(s)u).$$
The proof of the following lemma is a simple computation.
\begin{lemma}\label{modified Liouville vector field}
Let $\rho \colon [1-3 \epsilon, 1] \to \R$ be a smooth function such that:

The Liouville vector field $\mathcal{L}_{\mathcal U}$ corresponding to the Liouville form
$\theta_{\mathcal U}$ on $\mathcal{U}$ is equal to
$$ \mathcal{L}_{\mathcal U}= (s-\rho(s))\partial_s+\rho'(s)u\partial_u+\frac{\rho(s)}{s}\mathcal{L}_F.$$
\end{lemma}
We define the Liouville form $\theta$ on $S_0$ as $\theta|_{\mathcal C}= \theta_{\mathcal U}$ and $\theta|_{S_0 \setminus \mathcal{C}}=\theta_0$. By Lemma \ref{modified Liouville vector field} the Liouville vector field ${\mathcal L}$ of $\theta$ is transverse to $\partial S_0 \setminus \partial W_0$ and is equal to
$$(1-s)\partial_s+2u\partial_u+\frac{2s-1}{s}\mathcal{L}_F$$
in a neighbourhood of $\partial S_0 \setminus \partial W_0$; in particular it is tangent to $\partial S_0 \setminus \partial W_0$. Thus we can complete $(S_0, \theta)$ to $(S, \theta)$ by adding a half-symplectisation of $\partial S_0 \setminus \mathrm{int}(\partial W_0)$. We define a function $I \colon \partial S \to \R$ by setting $I=u$ on $\partial S_0 \cap \partial W_0$ and extending it to $\partial S$ so that $dI(\mathcal{L})=2I$ everywhere. It is easy to check that
$(S, \theta, I)$ is an exact Liouville sector.

A Lyapunov function $\mathfrak{f} \colon S_0 \to \R_+$ for ${\mathcal L}$ can be obtained by interpolating between $\mathfrak{f}_0$ on $S_0 \setminus \mathcal{C}$ and $u^2 - (s-1)^2 + \mathfrak{f}_F+C$ on $\mathcal{C} \cap \{ s \in [1-\epsilon, 1] \}$ for sufficiently large $C$. The Lyapunov function $\mathfrak{f}$ can then be extended to a Lyapunov function $\mathfrak{f} \colon S \to \R_+$ in a straightforward way. The easy verification that $(S, \theta, I, \mathfrak{f})$ is a Weinstein sector is left to the reader.

\subsection{Exact Lagrangian immersions with cylindrical end}
\begin{dfn}
  Let $(W, \theta)$ be a $2n$-dimensional Liouville manifold. An {\em exact Lagrangian immersion with cylindrical end} (or, alternately, an {\em immersed exact
Lagrangian submanifold with cylindrical end}) is an immersion $\iota \colon L \to W$ such that:
\begin{enumerate}
\item $L$ is an $n$-dimensional manifold and $\iota$ is a proper immersion which is an
embedding outside finitely many points,
\item  $\iota^* \theta = df$ for some function $f \colon L \to \R$, called the {\em potential} of
$(L, \iota)$, and
\item the image of $\iota$ is tangent to the Liouville vector field of $(W, \theta)$ outside of
a compact set of $L$.
\end{enumerate}
\end{dfn}
In the rest of the article, {\em immersed exact Lagrangian submanifold} will always mean
immersed exact Lagrangian submanifold with cylindrical end.
Note that $L$ is allowed to be compact, and in that case Condition (3) is empty: a closed immersed Lagrangian submanifold is a particular case of immersed Lagrangian submanifolds with cylindrical ends.
With an abuse of notation, we will often write $L$ either for the pair $(L, \iota)$ or for the
image $\iota(L)$.
\begin{ex}
Let $(W, \theta, \mathfrak{f})$ be a Weinstein manifold. The Lagrangian cocore planes $D_p$
introduced in Section~\ref{ss: Weinstein manifolds} are Lagrangian submanifolds
with cylindrical ends.
\end{ex}

Properness of $\iota$ and Condition (3) imply that for every immersed exact Lagrangian
submanifold with cylindrical end $\iota \colon L \to W$ there is $R>0$ sufficiently large such that $\iota(L)
\cap W_R^e$ corresponds to $[R, + \infty) \times \Lambda$ under the identification
$$(W_R^e, \theta) \cong ([R, + \infty) \times V, e^r \alpha),$$
where $\Lambda$ is a Legendrian submanifold of $(V, \xi)$. Then we say that $L$ is {\em
cylindrical over $\Lambda$}. Here $\Lambda$ can be empty (if $L$ is compact) or
disconnected.

There are different natural notions of equivalence between immersed exact Lagrangian
submanifolds. The stronger one is Hamiltonian isotopy.
\begin{dfn}
Two exact Legendrian immersions $(L, \iota_0)$ and $(L, \iota_1)$ with
cylindrical ends are Hamiltonian isotopic if there exists a function $H \colon [0,1] \times W \to \R$ with Hamiltonian flow $\varphi_t$ such that $\iota_1 = \varphi_1 \circ \iota_0$, and moreover $\iota_t = \varphi_t \circ \iota_0$ has cylindrical ends for all $t \in [0,1]$.
\end{dfn}
\begin{rem}
If $f_0 \colon L \to \R$ is the potential of $(L, \iota_0)$, by Lemma~\ref{variation of Liouville}
we can choose
\begin{equation}\label{eq: change of potential}
f_1=f_0 +  \int_0^1(- H_\sigma + \theta(X_{H_\sigma})) \circ \varphi_\sigma  d\sigma.
\end{equation}
as potential for $(L, \iota_1)$.
\end{rem}

The weakest one is exact Lagrangian regular homotopy.
\begin{dfn}
Two exact Legendrian immersions $(L, \iota_0)$ and $(L, \iota_1)$ with cylindrical ends are exact Lagrangian regular homotopic if there exists a smooth path of immersions  $\iota_t \colon L \to W$ for $t \in [0,1]$ such that $(L, \iota_t)$ is an exact Lagrangian immersion with cylindrical ends for every $t \in [0,1]$.
\end{dfn}
We recall that any exact regular homotopy $\iota_t \colon L
\to W$ can be generated by a local Hamiltonian defined on $L$ in the following sense.
\begin{lemma}
\label{lem:GeneratingHamiltonian}
An exact regular Lagrangian homotopy $\iota_t$ induces a smooth family of functions $G_t \colon L \to \R$ determined uniquely, up to a constant depending on $t$, by the requirement that the equation
$$ \iota_t^*(d\theta(\cdot,X_t)) = dG_t$$
be satisfied, where $X_t \colon L \to TW$ is the vector field along the immersion that
generates $\iota_t$. When $\iota_t$ has compact support, then $dG_t$ has compact support as well.

Conversely, any Hamiltonian $G_t \colon L \to \R$ generates an exact Lagrangian isotopy
$\iota_t \colon L \to (W, \theta)$ for any initial choice of exact immersion $\iota=\iota_0.$
\end{lemma}
\begin{rem} If $\iota_t$ is generated by an ambient Hamiltonian isotopy, then $H$ extends
to a single-valued Hamiltonian on $W$ itself. However, this is not necessarily the case for an
arbitrary exact Lagrangian regular homotopy.
\end{rem}

The limitations of our approach to define Floer cohomology for exact Lagrangian immersions require that we work with a restricted class of exact immersed Lagrangian submanifolds.
\begin{dfn}
We say that a Lagrangian immersion $(L, \iota)$ is {\em nice} if the singularities of $\iota(L)$
are all transverse double points, and for every double point $p$ the values of the potential at
the two points in the preimage of $p$ are distinct.
Then, given a double point $p$, we will denote
$\iota^{-1}(p) = \{ p^+, p^- \}$, where $f(p^+)>f(p^-)$.
\end{dfn}

\begin{rem}
If $L$ is not connected, we can shift the potential on different connected components by
independent constants. If $\iota^{-1}(p)$ is contained in a connected component of $L$,
then $f(p^+)- f(p^-)$ is still well defined. However, if the points in $\iota^{-1}(p)$ belong
to different connected components, the choice of $p^+$ and $p^-$ in $\iota^{-1}(p)$, and
$f(p^+)- f(p^-)$, depend of the choice of potential. For technical reasons related to our
definition of Floer cohomology, it seems useful, although unnatural, to consider the potential (up
to shift by an overall constant) as part of the data of an exact Lagrangian immersion.
\end{rem}

For nice immersed exact Lagrangian submanifolds we define a stronger form of exact Lagrangian regular homotopy.
\begin{dfn}\label{dfn: safe isotopy}
Let $(L, \iota_0)$ and $(L, \iota_1)$  be nice exact Lagrangian immersed submanifolds with cylindrical ends. An exact Lagrangian regular homotopy $(L, \iota_t)$ is a {\em safe isotopy}
if $(L, \iota_t)$ is nice for every $t \in [0,1]$.
\end{dfn}

Niceness can always be achieved after a $C^1$-small exact Lagrangian regular homotopy. In the rest of this article exact Lagrangian immersions will always be assumed nice.
\subsection{Contactisation and  Legendrian lifts}
We define a contact manifold $(M, \beta)$, where $M=W \times \R$, with a coordinate $z$
on $\R$, and $\beta = dz + \theta$. We call $(M, \beta)$ the {\em contactisation} of $(W,
\theta)$. A Hamiltonian isotopy $\varphi_t \colon W \to W$ which is generated by a
Hamiltonian function $H \colon [0,1] \times W \to \R$ lifts to a contact isotopy $\psi^+_t
\colon M \to M$ such that
$$\psi_t^+ (x, z) = (\psi_t(x), z-q_t(x)),$$
where $q_t \colon W \to \R$ is the function defined in Lemma~\ref{variation of Liouville}.

An immersed exact Lagrangian $(L, \iota)$ with potential $f \colon L \to \R$ uniquely
defines a Legendrian immersion
$$\iota^+ \colon L \to W \times \R, \quad \iota^+ (x) = (\iota(x), -f(x)).$$
Moreover $\iota^+$ is an embedding when $(L, \iota)$ is nice. We denote
the image of $\iota^+$ by $L^+$ and call it the {\em Legendrian lift} of $L$.
On the other hand, any Legendrian submanifold of $(M, \beta)$ projects to an immersed
Lagrangian in $W$. This projection is called the {\em Lagrangian projection} of the Legendrian
submanifold.

Double points of $L$ correspond to Reeb chords of $L^+$, and the action (i.e.~length) of the Reeb chord
projecting to a double point $p$ is $f(p^+)-f(p^-)$. If $L$ is connected, different potentials
induce Legendrian lifts which are contact isotopic by
a translation in the $z$-direction. In particular, the action of Reeb chords is independent of the
lift. On the other hand, if $L$ is disconnected, different potentials can induce
non-contactomorphic Legendrian lifts and the action of Reeb chords between different
connected components depends on the potential.

\section{Legendrian contact cohomology}
In this section we provide an overview of Legendrian contact cohomology. We recall the notion of augmentation and exaplin how they are used to define bilinearized Legendrian contact cohomology. \color{black}
\label{sec: lch}
\subsection{The Chekanov-Eliashberg algebra}
In view of the correspondence between Legendrian submanifolds of $(M, \beta)$ and exact
Lagrangian immersions in $(W, \theta)$, Floer cohomology for Lagrangian immersions will be
a variation on the theme of Legendrian contact cohomology. The latter was proposed by
Eliashberg and Hofer and later defined rigorously by Chekanov, combinatorially, in $\R^3$
with its standard contact structure in \cite{Chekanov_DGA_Legendrian}, and by Ekholm,
Etnyre and Sullivan, analytically, in the contactisation of any Liouville manifold in
\cite{LCHgeneral}. In this subsection we summarise the analytical definition.

For $d>0$, let $\widetilde{\mathcal R}^{d+1} = \op{Conf}^{d+1}(\partial D^2)$ be the
space of parametrised discs with $d+1$ punctures on the boundary. The automorphism
group $Aut(D^2)$ acts on $\widetilde{\mathcal R}^{d+1}$ and its quotient is the
Deligne-Mumford moduli space ${\mathcal R}^{d+1}$. Given $\boldsymbol{\zeta} =
(\zeta_0, \ldots, \zeta_d) \in \widetilde{\mathcal R}^{d+1}$, we will denote
$$\Delta_{\boldsymbol{\zeta}} = D^2 \setminus \{ \zeta_0, \ldots, \zeta_d \}.$$
Following \cite{Seidel_Fukaya}, near every puncture $\zeta_i$ we will define positive and
negative universal striplike ends with coordinates $(\sigma_i^+, \tau_i^+) \in (0, + \infty)
\times [0,1]$ and $(\sigma_i^-, \tau_i^-) \in ( - \infty, 0) \times [0,1]$ respectively. We will
assume that $\sigma_i^- = - \sigma_i^+$ and $\tau_i^- = 1- \tau_i^+$.
\begin{rem}
Putting both positive and negative strip-like ends near each puncture could be useful to
compare wrapped Floer cohomology and contact cohomology, which use different conventions for positive and negative punctures.
\end{rem}

\begin{dfn}
Let $(V, \alpha)$ be a contact manifold with contact structure $\xi$ and Reeb vector field
${\mathcal R}$. An almost complex structure $J$ on $\R \times V$ is {\em cylindrical} if
\begin{enumerate}
\item $J$ is invariant under translations in $\R$,
\item $J(\partial_r)= {\mathcal R}$, where $r$ is the coordinate in $\R$,
\item $J(\xi) \subset \xi$, and $J|_{\xi}$ is compatible with $d \alpha|_{\xi}$.
\end{enumerate}
\end{dfn}

\begin{dfn}
An almost complex structure $J$ on a Liouville manifold $W$ is {\em compatible with
$\theta$} if it is compatible with $d \theta$ and, outside a compact set, corresponds to
a cylindrical almost complex structure under the identification \eqref{end as symplectisation}.
We denote by ${\mathcal J}(\theta)$ the set of almost complex structures on $W$ which are
compatible with $\theta$.
\end{dfn}
It is well known that ${\mathcal J}(\theta)$ is a contractible space.

Given an exact Lagrangian immersion $(L, \iota)$ in $W$, we will consider almost complex
structures $J$ on $W$ which satisfy the following
\begin{enumerate}[label=($\dagger$), ref=($\dagger$)]
\item \label{acs} $J$ is compatible with $\theta$, integrable in a neighbourhood of the
double points of $(L, \iota)$, and for which $L$ moreover is real-analytic near the double
points.
\end{enumerate}
We will denote the set of double points of $(L, \iota)$ by $D$.

Let $u \colon \Delta_{\boldsymbol{\zeta}} \to W$ be a $J$-holomorphic map with boundary in
$L$. If $u$ has finite area and no puncture at which the lift of $u|_{\partial \Delta_{\boldsymbol{\zeta}}}$ to $L$ has a continuous extension, then $\lim \limits_{z \to \zeta_i} u(z)=p_i$ for some $p_i \in D$. Since the boundary of $u$ switches branch near $p_i$, the following dichotomy thus makes sense:
\begin{dfn}\label{positive negative}
We
say that $\zeta_i$ is a {\em positive puncture} at $p_i$ if
$$\lim \limits_{\sigma_i^+ \to + \infty} (\iota^{-1} \circ u)(\sigma_i^+, 0) = p_i^+$$
and that $\zeta_i$ is a {\em negative puncture} at $p_i$ if
$$\lim \limits_{\sigma_i^- \to - \infty} (\iota^{-1} \circ u)(\sigma_i^-, 0) = p_i^+.$$
\end{dfn}

Let $L$ be an immersed exact Lagrangian. If $p_1, \ldots, p_d$ are double points of $L$
(possibly with repetitions), we denote by $\widetilde{\mathfrak{N}}_L(p_0; p_1, \ldots,
p_d; J)$ the set of pairs $(\boldsymbol{\zeta}, u)$ where:
\begin{enumerate}
\item $\boldsymbol{\zeta} \in \widetilde{R}^{d+1}$ and
$u \colon \Delta_{\boldsymbol{\zeta}} \to W$ is a $J$-holomorphic map,
\item $u(\partial \Delta_{\boldsymbol{\zeta}}) \subset L$, and
\item $\zeta_0$ is a positive puncture at $p_0$ and $\zeta_i$, for $i= 1, \ldots, d$, is a
negative puncture at $p_i$.
\end{enumerate}
The group $\op{Aut}(D^2)$ acts on $\widetilde{\mathfrak{N}}_L(p_0; p_1, \ldots,
p_d; J)$ by reparametrisations; the quotient is the moduli space $\mathfrak{N}_L(p_0; p_1,
\ldots, p_d; J)$. Note that the set $p_1, \ldots, p_d$ can be empty. In this case, the elements
of the moduli spaces $\mathfrak{N}(p_0; J)$ are called {\em teardrops}.

Given $u \in \widetilde{\mathfrak{N}}_L(p_0; p_1, \ldots, p_d; J)$, let $D_u$ be the
linearisation of the Cauchy-Riemann operator at $u$.  By standard Fredholm theory,
$D_u$ is a Fredholm operator with index $\op{ind}(D_u)$.
We define the {\em index} of $u$ as
$$\op{ind}(u)= \op{ind}(D_u)+ d -2.$$
It is locally constant, and we denote by
$\mathfrak{N}^k_L(p_0; p_1, \ldots, p_d; J)$ the subset of
$\mathfrak{N}_L(p_0; p_1, \ldots, p_d; J)$ consisting of classes of maps $u$ with $\op{ind}(u)
=k$.

The following proposition is a version of \cite[Proposition 2.3]{LCHgeneral}:
\begin{prop}
For a generic $J$ satisfying the condition \ref{acs},  the moduli space
$$\mathfrak{N}^k_L(p_0; p_1, \ldots, p_d; J)$$ is a transversely cut out
manifold of dimension $k$. In particular, if $k<0$ it is empty; if $k=0$ it is compact, and therefore consists of a finite number of points; and if $k=1$, it can be compactified in the sense of Gromov, see
\cite[Section 2.2]{LCHgeneral}.

The boundary of the compactification of the moduli space
$\mathfrak{N}^1_L(p_0; p_1, \ldots, p_d; J)$ is
\begin{equation}\label{compactification of N^1}
\bigsqcup_{q \in D} \bigsqcup_{0 \le i < j \le d} \mathfrak{N}^0_L(p_0; p_1, \ldots,
p_i , q, p_{j+1}, \ldots, p_d; J) \times  \mathfrak{N}^0_L(q; p_{i+1}, \ldots, p_j; J).
\end{equation}
If $L$ is spin, the moduli spaces are orientable and a choice of spin structure induces a
coherent orientation of the moduli spaces; see \cite{Ekholm_&_Orientation_Homology}.
\end{prop}

\begin{dfn}
We say that an almost complex structure $J$ on $W$ is $L$-regular if it satisfies $(\dagger)$ for $L$ and all moduli spaces
$\mathfrak{N}_L(p_0; p_1, \ldots, p_d; J)$ are transversely cut out.
\end{dfn}

To a Legendrian submanifold $L^+$ of $(M, \beta)$ we can associate a differential graded
algebra $(\mathfrak{A}, \mathfrak{d})$ called the {\em Chekanov-Eliashberg algebra}
(or {\em Legendrian contact cohomology algebra}) of $L^+$. As an algebra, $\mathfrak{A}$
is the free unital noncommutative algebra generated by the double points of the Lagrangian
projection $L$ or, equivalently, by the Reeb chords of $L^+$. The grading takes values in
$\Z / 2 \Z$ and is simply given by the self-intersection of the double points. If $2c_1(W)=0$
and the Maslov class of $L^+$ vanishes, it can be lifted to an integer valued grading by the
Conley-Zehnder index. We will not make explicit use of the integer grading, and therefore we
will not describe it further, referring the interested reader to \cite[Section 2.2]{LCHgeneral}
instead.

The differential $\mathfrak{d}$ is defined on the generators as:
  $$\mathfrak{d}(p_0) = \sum_{d \ge 0} \sum_{p_1, \ldots, p_d} \#  \mathfrak{N}_L^0(p_0; p_1,
\ldots, p_n;J) p_1 \ldots p_d.$$
According to \cite[Proposition 2.6]{LCHgeneral}, the Chekanov-Eliashberg algebra is a
Legendrian invariant:
\begin{thm}[\cite{LCHgeneral}] \label{dga invariance}
If $L_0^+$ and $L_1^+$ are Legendrian isotopic Legendrian submanifolds of $(M, \beta)$,
then their Chekanov-Eliashberg algebras $(\mathfrak{A}^0, \mathfrak{d}_0)$ and
$(\mathfrak{A}^1, \mathfrak{d}_1)$ are stable tame isomorphic.\color{black}
\end{thm}
The definition of stable tame isomorphism of DGAs was introduced by Chekanov in
\cite{Chekanov_DGA_Legendrian}, and then discussed by Ekholm, Etnyre and Sullivan
in \cite{LCHgeneral}. We will not use it in this article but note that on the homological level a stable tame isomorphism induces an isomorphism.

\subsection{Bilinearised Legendrian contact cohomology}
Differential graded algebras are difficult objects to manipulate, and therefore Chekanov
introduced a linearisation procedure. The starting point of this procedure is the existence of an augmentation.
\begin{dfn}
Let $\mathfrak{A}$ be a differential graded algebra over a commutative ring $\F$. An
{\em augmentation} of $\mathfrak{A}$ is a unital differential graded algebra morphism
$\varepsilon \colon \mathfrak{A} \to \F$.
\end{dfn}

Let $L_0^+$ and $L_1^+$ be Legendrian submanifolds of $(M, \beta)$ with Lagrangian
projections $L_0$ and $L_1$ with potentials $f_0$ and $f_1$ respectively. We recall that the
 potential is the negative of the $z$ coordinate.
We will assume that $L_0^+$ and $L_1^+$ are {\em chord generic},
which means in this case that $L_0^+ \cap L_1^+ = \emptyset$ and all singularities of
$L_0 \cup L_1$ are transverse double points.

Let $\mathfrak{A}_0$ and $\mathfrak{A}_1$ be the Chekanov-Eliashberg algebras of
$L_0^+$ and $L_1^+$ respectively. Let $\varepsilon_0 \colon \mathfrak{A}_0 \to \F$ and
$\varepsilon_1 \colon \mathfrak{A}_1 \to \F$ be augmentations. Now we describe
the construction of the {\em bilinearised Legendrian contact cohomology} complex
$\mathrm{LCC}_{\varepsilon_0, \varepsilon_1}(L_0^+, L_1^+;J)$.

First, we introduce some notation. We denote by $D_i$ the set of double points of
$L_i$ (for $i=1,2$) and by ${\mathcal C}$ the intersection points of $L_0$ and $L_1$ such that $f_0(q) < f_1(q)$. Double points in $D_i$ correspond to
Reeb chords of $L_i^+$, and double points in ${\mathcal C}$ corresponds to Reeb
chords from $L_1^+$ to $L_0^+$ (note the order!). We define the exact immersed Lagrangian $L =
L_0 \cup L_1$ and, for $\mathbf{p}^0= (p_1^0, \ldots, p_{l_0}^0) \in D_0^{l_0}$,
$\mathbf{p}^1= (p_1^1, \ldots, p_{l_1}^1) \in D_1^{l_1}$ and $q_\pm \in {\mathcal C}$,
we denote $$\mathfrak{N}_L^i(q_+; \mathbf{p}^0 , q_-, \mathbf{p}^1 ;J):=
\mathfrak{N}_L^i(q_+;  p_1^0, \ldots, p_{l_0}^0, q_-, p_1^1, \ldots, p_{l_1}^1; J),$$ where
$J$ is an $L$-regular almost complex structure. If $\varepsilon_i$ is an
augmentation of $\mathfrak{A}_i$, we denote $\varepsilon_i(\mathbf{p}^i):=
\varepsilon_i(p_1^i) \cdots \varepsilon_i(p_{l_i}^i)$.

As an $\F$-module,  $\mathrm{LCC}_{\varepsilon_0, \varepsilon_1}(L_0^+, L_1^+;J)$ is freely generated by
 the set ${\mathcal C}$ and the differential of a generator $q_- \in {\mathcal C}$ is defined
as
\begin{equation}\label{LCH differential}
\partial_{\varepsilon_0, \varepsilon_1} (q_-) = \sum_{q_+ \in \mathcal C} \sum_{l_0, l_1 \in \N}
\sum_{\mathbf{p}^i \in D_i^{l_i}} \# \mathfrak{N}_L^0(q_+; \mathbf{p}^0 , q_-, \mathbf{p}^1;
J) \varepsilon_0(\mathbf{p}^0) \varepsilon_1(\mathbf{p}^1) q_+.
\end{equation}
The bilinearised Legendrian contact cohomology $\mathrm{LCH}_{\varepsilon_0, \varepsilon_1}(L_0^+, L_1^+)$
is the homology of this complex. The set of isomorphism classes of bilinearised
Legendrian contact cohomology groups is independent of
the choice of $J$ and is a Legendrian isotopy invariant by the
adaptation of Chekanov's argument from
\cite{Chekanov_DGA_Legendrian} due to the first author and Bourgeois
\cite{augcat}.

\section{Floer cohomology for exact Lagrangian immersions}
\label{sec: Floer homology for immersions}
In this section we define a version of Floer cohomology for exact Lagrangian immersions. Recall that our exact Lagrangian immersions are equipped with choices of potentials making their Legendrian lifts embedded. This is not new material; similar or even more general accounts can be found, for example,
in \cite{Akaho:Immersed}, \cite{ExactGradedImmersedLagrFloerTheory} and
\cite{DehnTwistsExactSequenceThroughLagrangianCobordisms}.
\subsection{Cylindrical Hamiltonians}
\begin{dfn}\label{cylindrical Hamiltonian}
Let $W$ be a Liouville manifold.
A Hamiltonian function $H \colon [0,1] \times W \to \R$ is {\em cylindrical} if there is a
function $h \colon \R^+ \to \R$ such that $H(t, w)=h(e^{\fr(w)})$ outside a compact set of
$W$.
\end{dfn}
The following example describes the behaviour of the Hamiltonian vector field of a cylindrical
Hamiltonian in an end of $W$, after taking into account the identification \eqref{end as
symplectisation}.
\begin{ex}
Let $(V, \alpha)$ be a contact manifold with Reeb vector field ${\mathcal R}$ and let
$(\R \times V, d(e^r \alpha))$ be its symplectisation. Given a smooth function
$h \colon \R^+ \to \R$, we define the autonomous Hamiltonian $H \colon \R \times V \to \R$
by $H(r, v) = h(e^r)$. Then the Hamiltonian vector field of $H$ is
$$X_H(r, v)= h'(e^r){\mathcal R}(r, v).$$
\end{ex}

Let $(L_0, \iota_0)$ and $(L_1, \iota_1)$ be two immersed exact Lagrangian submanifolds of
$W$ with cylindrical ends over Legendrian submanifolds $\Lambda_0$ and $\Lambda_1$ of
$(V, \alpha)$. Given a cylindrical Hamiltonian $H \colon [0,1] \times W \to \R$, we denote
by ${\mathcal C}_H$ --- or simply ${\mathcal C}$ when there is no risk of confusion ---
the set of Hamiltonian chords $x \colon
[0,1] \to W$ of $H$ such that $x(0) \in L_0$ and $x(1) \in L_1$.
If $\varphi_t$ denotes the Hamiltonian flow of $H$, then ${\mathcal C}_H$
is in bijection with $\varphi_1(L_0) \cap L_1$.
\begin{dfn}\label{compatible hamiltonian}
A cylindrical Hamiltonian $H \colon [0,1] \times W \to \R$, with Hamiltonian flow $\varphi_t$,
is {\em compatible} with $L_0$ and $L_1$ if
\begin{itemize}
\item[(i)] no starting point or endpoint of a chord $x \in {\mathcal C}_H$ is a
double point of $(L_0, \iota_0)$ or $(L_1, \iota_1)$,
\item[(ii)] $\varphi_1(L_0)$ intersects $L_1$ transversely,
\item[(iii)] for $\rho$ large enough $h'(\rho)= \lambda$ is constant, and
\item[(iv)] all time-one Hamiltonian chords from $L_0$ to $L_1$ are contained in a compact
region.
\end{itemize}
\end{dfn}
Condition (iv) is equivalent to asking that $\lambda$ should not be the
length of a Reeb chord from $\Lambda_0$ to $\Lambda_1$.
\begin{rem}
If cylindrical Hamiltonian $H$ is {\em compatible} with $L_0$ and $L_1$, then
${\mathcal C}_H$ is a finite set.
\end{rem}

\subsection{Obstructions}
\label{sec: obstructions}

If one tries to define Floer cohomology for immersed Lagrangian submanifolds by extending the
usual definition naively, one runs into the problem that the ``differential'' might not square
to zero because of the bubbling of teardrops in one-dimensional families of Floer strips. Thus,
if $(L, \iota)$ is an immersed Lagrangian submanifold and $J$ is $L$-regular, we define a map
$\mathfrak{d}^0 \colon D \to \Z$ by
$$\mathfrak{d}^0 (p) = \# \mathfrak{N}_L^0(p;J)$$
and extend it by linearity to the free module generated by $D$.
The map $\mathfrak{d}^0$ is called the {\em obstruction} of $(L, \iota)$. If
$\mathfrak{d}^0=0$ we say that $(L, \iota)$ is {\em uncurved}.

Typically, asking that an immersed Lagrangian submanifold be unobstructed is too
much, and a weaker condition will ensure that Floer cohomology can be defined.
We observe that $\mathfrak{d}^0$ is a component of the Chekanov-Eliashberg algebra
of the Legendrian lift of $L$, and make the following definition.

\begin{dfn}
Let $(L, \iota)$ be an immersed exact Lagrangian submanifold. The {\em obstruction algebra}
$(\mathfrak{D}, \mathfrak{d})$ of $(L, \iota)$ --- or of $(L, \iota, J)$ when the almost
complex structure is not clear from the context --- is the Chekanov-Eliashberg algebra of the
Legendrian lift $L^+$.
\end{dfn}

If $L$ is connected, its obstruction algebra $(\mathfrak{D}, \mathfrak{d})$ does not
depend on the potential. On the other hand, if $L$ is disconnected, the potential differences
at the double points of $L$ involving different connected components, and therefore what
holomorphic curves are counted in $(\mathfrak{D}, \mathfrak{d})$, depend on the choice
of the potential.

\begin{dfn}
An exact immersed Lagrangian $(L, \iota)$ is {\em unobstructed} if $(\mathfrak{D},
\mathfrak{d})$ admits an augmentation.
\end{dfn}
Unobstructedness does not depend on the choice of $L$-regular almost complex structure and is invariant under general Legendrian isotopies as a consequence of Theorem~\ref{dga invariance} (this fact will not be needed). However, we will need the invariance statement from the following proposition.
\begin{prop}\label{invariance of D}
If $L_0$ and $L_1$ are safe isotopic exact Lagrangian immersions, and $J_0$ and $J_1$ are $L_0$-regular and $L_1$-regular almost complex structures respectively, then the obstruction algebras $(\mathfrak{D}_0, \mathfrak{d}_0)$ of $L_0$ and $(\mathfrak{D}_1, \mathfrak{d}_1)$ of $L_1$ are isomorphic.
In particular, there is a bijection between the augmentations of $(\mathfrak{D}_0,
\mathfrak{d}_0)$ and the augmentations of $(\mathfrak{D}_1, \mathfrak{d}_1)$.

\end{prop}
\begin{proof}
The safe isotopy between $L_0$ and $L_1$ induces a Legendrian isotopy between the Legendrian lifts $L_0^+$ and $L_1^+$ without birth or death of Reeb chords.
By \cite[Proposition 2.6]{LCHgeneral} the Chekanov-Eliashberg algebras of $L_0^+$ and $L_1^+$ are stably tame isomorphic, and moreover stabilisation occurs only at the birth or death of a Reeb chord.
\end{proof}
Proposition~\ref{invariance of D} is the main reason why we have made the choice of
distinguishing between the obstruction algebra of a Lagrangian immersion and the
Chekanov-Eliashberg algebra of its Legendrian lift. In fact Legendrian submanifolds are
more naturally considered up to Legendrian isotopy. However, in this article we will consider
immersed Lagrangian submanifolds only up to the weaker notion of safe isotopy.
\subsection{The differential}
We denote by $Z$ the strip $ \R \times [0,1]$ with coordinates $(s,t)$. Let
$\widetilde{\mathcal R}^{l_0|l_1} \cong \op{Conf}^{l_0}(\R) \times \op{Conf}^{l_1}(\R)$
be the set of pairs $(\boldsymbol{\zeta}^0, \boldsymbol{\zeta}^1)$ such that
$\boldsymbol{\zeta}^0= \{ \zeta_1^0, \ldots, \zeta_{l_0}^0 \} \subset \R \times \{ 0 \}$ and
$\boldsymbol{\zeta}^1= \{ \zeta_{l_1}^1, \ldots, \zeta_1^1 \} \subset \R \times \{ 1 \}$. We
assume that the $s$-coordinates of $\zeta_j^i$ are increasing in $j$ for $\zeta_j^0$ and
decreasing for $\zeta_j^1$.  We define
$$Z_{\boldsymbol{\zeta}^0, \boldsymbol{\zeta}^1}:= Z \setminus \{\zeta_1^0, \ldots,
\zeta_{l_0}^0, \zeta_1^1, \ldots, \zeta_{l_1}^1  \}.$$
The group $\op{Aut}(Z)=\R$ acts on $\widetilde{\mathcal R}^{l_0|l_1}$.

In the rest of this section we will assume that $H$ is compatible with $(L_0, \iota_0)$ and
$(L_1, \iota_1)$.
We will consider time-dependent almost complex structure $J_\bullet$ on $W$
which satisfy the following
\begin{enumerate}[label=($\dagger \dagger$), ref=($\dagger \dagger$)]
\item \label{double acs}
\begin{itemize}
\item[(i)] $J_t$ is compatible with $\theta$ for all $t$,
\item[(ii)] $J_t = J_0$ for $t \in [0, 1/4]$ in a neighbourhood of the double points of $L_0$
and $J_t=J_1$ for $t \in [3/4,1]$ in a neighbourhood of the double points of $L_1$,
\item[(iii)] $J_0$ satisfies \ref{acs} for $L_0$ and  $J_1$ satisfies \ref{acs} for $L_1$.
\end{itemize}
\end{enumerate}
Condition (ii) is necessary to ensure that $J_\bullet$ is independent of the coordinate
$\sigma_{i,j}^-$ in some neighbourhoods of the boundary punctures $\zeta_j^i$, so that
we can apply standard analytical results.

For the same reason we fix once and for all a nondecreasing function $\chi \colon \R \to [0,1]$
such that
$$\chi(t) = \begin{cases}
0 & \text{for } t \in [0, 1/4], \\
1 &  \text{for } t \in [3/4, 1]
\end{cases}$$
and $\chi'(t) \le 3$ for all $t$. It is easy to see that, if $\varphi_t$ is the Hamiltonian
flow of a Hamiltonian function $H$,  then $\varphi_{\chi(t)}$ is the Hamiltonian flow
of  the Hamiltonian function $H'$ such that $H'(t,w)=\chi'(t) H(\chi(t), w)$. We will
use $\chi'$ to cut off the Hamiltonian vector field in the Floer equation to ensure that
it has the right invariance properties in the strip-like ends corresponding to the boundary
punctures $\zeta_j^i$.

Given Hamiltonian chords
$x_+, x_- \in {\mathcal C}$ and self-intersections $p_1^0, \ldots,
p_{l_0}^0$ of $L_0$ and $p_1^1, \ldots, p_{l_1}^1$ of $L_1$ we define the moduli space
$$\widetilde{\mathfrak{M}}_{L_0, L_1}( p_1^1, \ldots, p_{l_1}^1, x_- , p_1^0,
\ldots, p_{l_0}^0, x_+; H, J_\bullet)$$
of triples $(\boldsymbol{\zeta}^0, \boldsymbol{\zeta}^1, u)$ such that
\begin{itemize}
\item $(\boldsymbol{\zeta}^0, \boldsymbol{\zeta}^1) \in \widetilde{\mathcal R}^{l_0|l_1}$
and $u \colon Z_{\boldsymbol{\zeta}^0, \boldsymbol{\zeta}^1} \to W$ is a map satisfying the
 Floer equation
\begin{equation}\label{eq: Floer}
\frac{\partial u}{\partial s} + J_t \left (\frac{\partial u}{\partial t} -  \chi'(t)
X_{H}( \chi(t), u) \right )=0,
\end{equation}
\item $\lim \limits_{s \to \pm \infty} u(s,t) = x_\pm(\chi(t))$ uniformly in $t$,
\item $u(s,0) \in L_0$ for all $(s,0) \in Z_{\boldsymbol{\zeta}^0, Z_{\boldsymbol{\zeta}^1}}$,
\item $u(s,1) \in L_1$ for all $(s,1) \in Z_{\boldsymbol{\zeta}^0, Z_{\boldsymbol{\zeta}^1}}$,
\item $\lim \limits_{z \to \zeta_j^i} u(z) = p_j^i$, and
\item $\zeta^i_j$ is a negative puncture at $p^i_j$ for $i = 0,1$ and $j = 1, \ldots, l_i$.
\end{itemize}
The group $\op{Aut}(Z)= \R$ acts on $\widetilde{\mathfrak{M}}_{L_0, L_1}( p_1^1, \ldots,
p_{l_1}^1, x_- , p_1^0, \ldots, p_{l_0}^0, x_+; H, J_\bullet)$ by
reparametrisations. We will denote the quotient by
$$\mathfrak{M}_{L_0, L_1}( p_1^1, \ldots, p_{l_1}^1, x_- , p_1^0, \ldots, p_{l_0}^0,
x_+; H, J_\bullet).$$
For $u \in \widetilde{\mathfrak{M}}_{L_0, L_1}( p_1^1, \ldots, p_{l_1}^1,
x_- , p_1^0, \ldots, p_{l_0}^0, x_+; H, J_\bullet)$, we denote by $F_u$
the linearisation of the Floer operator
$${\mathcal F}(u)= \frac{\partial u}{\partial s} + J_t \left (\frac{\partial u}{\partial t} -
 \chi'(t) X_{H}(\chi(t), u) \right )$$
at $u$. By standard Fredholm theory, $F_u$ is a Fredholm operator with index
$\op{ind}(F_u)$. We define
$$\op{ind}(u)= \op{ind}(F_u)+ l_0+l_1.$$
The index is locally constant, and we denote by
$$\mathfrak{M}^k_{L_0, L_1}( p_1^1, \ldots, p_{l_1}^1, x_- , p_1^0,
\ldots, p_{l_0}^0, x_+; H, J_\bullet)$$
the subset of $\mathfrak{M}_{L_0, L_1}( p_1^1, \ldots, p_{l_1}^1, x_- , p_1^0, \ldots,
p_{l_0}^0, x_+; H, J_\bullet)$ consisting of classes of maps $u$ with index
$\op{ind}(u)=k$.

Observe that similar construction for closed, exact, graded, immersed
Lagrangian submanifolds was considered by Alston and Bao in
\cite{ExactGradedImmersedLagrFloerTheory}, where the regularity
statement appears in \cite[Proposition 5.2]{ExactGradedImmersedLagrFloerTheory}
and compactness is discussed in \cite[Section 4]{ExactGradedImmersedLagrFloerTheory}.
In addition, the corresponding statement in the case of Legendrian contact cohomology in
$P\times \R$ was proven by Ekholm, Etnyre and Sullivan, see
\cite[Proposition 2.3]{LCHgeneral}. The following proposition translates those compactness
and regularity statements to the settings under consideration.

\begin{prop}\label{moduli spaces of Floer trajectories}
For a generic time-dependent almost complex structure $J_\bullet$ satisfying \ref{double acs},
for which moreover $J_0$ is $L_0$-regular and  $J_1$ is $L_1$-regular, the moduli space
$$\mathfrak{M}^k_{L_0, L_1}( p_1^1, \ldots, p_{l_1}^1, x_- , p_1^0, \ldots, p_{l_0}^0, x_+; H,
J_\bullet)$$
is a transversely cut-out manifold of dimension $k-1$. If $k=1$ it is compact, and
therefore consists of a finite number of points. If $k=2$ it can be compactified in
the sense of Gromov-Floer.

The boundary of the compactification of the one-dimensional moduli space\\
$\mathfrak{M}^2_{L_0, L_1}( p_1^1, \ldots, p_{l_1}^1, x_- , p_1^0, \ldots, p_{l_0}^0, x_+; H,
J_\bullet)$ is
\begin{align}\label{compactification of M^2}
 \bigsqcup_{y \in {\mathcal C}_H} \bigsqcup_{0 \le h_i \le l_i} & \mathfrak{M}^1_{L_0, L_1}(p_{h_1+1}^1, \ldots, p_{l_1}^1, y, p_1^0,
\ldots, p_{h_0}^0, x_+; H, J_\bullet) \\
& \times  \mathfrak{M}^1_{L_0, L_1} (p_1^1,
\ldots, p_{h_1}^1, x_-, p^0_{h_0+1}, \ldots, p^0_{l_0}, y; H, J_\bullet) \nonumber \\
\bigsqcup_{q \in D_1} \bigsqcup_{0 \le i \le j \le l_1} & \mathfrak{M}^1_{L_0, L_1}( p_1^1, \ldots,
p_i^1, q, p_{j+1}^1, \ldots p_{l_1}^1, x_-, p_1^0, \ldots, p_{l_0}^0, x_+; H, J_\bullet) \nonumber \\
& \times \mathfrak{N}^0_{L_1}(q; p_{i+1}^1, \ldots, p_j^1; J_1) \nonumber \\
 \bigsqcup_{q \in D_0} \bigsqcup_{0 \le i \le j \le l_0} & \mathfrak{M}^1_{L_0, L_1}( p_1^1, \ldots,
p_{l_1}^1, x_-, p_1^0, \ldots, p_i^0, q, p_{j+1}^0, \ldots p_{l_0}^0, x_+; H, J_\bullet) \nonumber \\
& \times \mathfrak{N}^0_{L_0}(q; p_{i+1}^0, \ldots, p_j^0; J_0). \nonumber
\end{align}

If both $L_0$ and $L_1$ are spin the moduli spaces are orientable, and a choice of spin
structure on each Lagrangian submanifold induces a coherent orientation on the moduli
spaces.
\end{prop}
\begin{rem}
We use different conventions for the index of maps involved in the definition of the
obstruction algebra and for maps involved in the definition of Floer cohomology. Unfortunately
this can cause some confusion, but it is necessary to remain consistent with the standard
conventions in the literature.
\end{rem}

\begin{dfn}
We say that a time-dependent almost complex structure $J_\bullet$ on $W$ is
$(L_0, L_1)$-regular if
it satisfies \ref{double acs} for $L_0$ and $L_1$ and all moduli spaces
$$\mathfrak{M}^k_{L_0, L_1}( p_1^1, \ldots, p_{l_1}^1, x_- , p_1^0, \ldots, p_{l_0}^0, x_+; H,
J_\bullet)$$ are transversely cut out.
\end{dfn}
Note that, strictly speaking, the condition of being $(L_0, L_1)$-regular depends also on the
Hamiltonian, even if we have decided to suppress it from the notation.

Suppose that the obstruction algebras $\mathfrak{D}_0$ and $\mathfrak{D}_1$ admit
augmentations $\varepsilon_0$ and $\varepsilon_1$ respectively. To simplify the notation
we write $\mathbf{p}^i = (p_1^i, \ldots, p^i_{l_i})$,
$$\mathfrak{M}^k_{L_0,L_1}(\mathbf{p}^1, x_- , \mathbf{p}^0, x_+; H, J_\bullet) =
\mathfrak{M}^k_{L_0, L_1}( p_1^1, \ldots, p_{l_1}^1, x_- , p_1^0, \ldots, p_{l_0}^0, x_+;
H, J_\bullet)$$
 and $\varepsilon_i({\mathbf{p}^i}) = \varepsilon_i(p^1_1) \ldots \varepsilon_i(p^i_{l_i})$
for $i=0,1$. We also introduce the weighted count
$$\mathfrak{m}(\mathbf{p}^1, x_- , \mathbf{p}^0, x_+) =
\# \mathfrak{M}^0_{L_0,L_1}(\mathbf{p}^1, x_- , \mathbf{p}^0, x_+; H, J_\bullet)
\varepsilon_0(\mathbf{p}^0)\varepsilon_1(\mathbf{p}^1).$$

Then we define the Floer complex over the commutative ring $\F$
$$\mathrm{CF}((L_0, \varepsilon_0), (L_1, \varepsilon_1); H, J_\bullet) = \bigoplus_{x \in {\mathcal C}} \F x$$
with differential
$$\partial \colon \mathrm{CF}((L_0, \varepsilon_0), (L_1, \varepsilon_1); H, J_\bullet)  \to
 \mathrm{CF}((L_0, \varepsilon_0), (L_1, \varepsilon_1); H, J_\bullet)$$
defined on the generators by
\begin{equation}\label{eq: immersed differential}
\partial x_+ = \sum_{x_- \in {\mathcal C}_H} \sum_{l_0, l_1 \in \N} \sum_{\mathbf{p}^i \in D_i^{l_i}}
 \mathfrak{m}(\mathbf{p}^1, x_- , \mathbf{p}^0, x_+) x_-.
\end{equation}

The algebraic interpretation of the Gromov-Floer compactification of the one-dimensional
moduli spaces in Proposition~\ref{moduli spaces of Floer trajectories} is that $\partial^2=0$.
We will denote the homology by $\mathrm{HF}((L_0, \varepsilon_0), (L_1, \varepsilon_1);H)$. The
suppression of $J_\bullet$ from the notation is justified by Subsection~\ref{sss: changing J}.

\subsection{Comparison with bilinearised Legendrian contact cohomology}
In this subsection we compare the Lagrangian Floer cohomology of a
pair of immersed exact Lagrangian submanifolds with the bilinearised Legendrian contact cohomology of a particular Legendrian lift of theirs.

Let $L_0$ and $L_1$ be exact Lagrangian immersions, $H \colon [0,1]
\times W \to \R$ a Hamiltonian function compatible with $L_0$ and
$L_1$ with Hamiltonian flow $\varphi_t$, and $J_\bullet$ an $(L_0,
L_1)$-regular almost complex structure. We will introduce the
``backward'' isotopy $\overline{\varphi}_t = \varphi_1 \circ
\varphi_t^{-1},$ where $\varphi_1 \circ \varphi_t^{-1} \circ \varphi_1^{-1}$ can be generated by the Hamiltonian $ -H(t,\varphi_t^{-1} \circ \varphi_1^{-1}).$

Given an almost complex structure $J_\bullet$ and an arbitrary path $\phi_t$ of symplectomorphisms, we denote by
$\phi_* J_\bullet$ the almost complex structure
defined as $\phi_*J_t = d \phi_{\chi(t)}
\circ J_t \circ d \phi_{\chi(t)}^{-1}$.  The time
rescaling by $\chi$ ensures that $\phi_*J_\bullet$
satisfies \ref{double acs} for $\phi_1(L_0)$ and $L_1$ if and only if
$J_\bullet$ does for $L_0$ and $L_1$.

\begin{lemma}\label{lemma: naturality}
Denote by $\boldsymbol{0}$ the constantly zero function on $W$ and
set $q_{\pm}=x_\pm(1) \in \varphi_1(L_0) \cap L_1$, regarded as
Hamiltonian chords of $\boldsymbol{0}$. There is a bijection
$$\overline{\varphi}_* \colon \mathfrak{M}_{L_0, L_1}(\mathbf{p}^1, x_- , \mathbf{p}^0,
x_+; H, J_\bullet) \to \mathfrak{M}_{\varphi_1(L_0),
L_1}(\mathbf{p}^1, q_- , \varphi_1(\mathbf{p}^0), q_+;
\boldsymbol{0}, \overline{\varphi}_*J_\bullet)$$ defined by
$(\overline{\varphi}_*u)(s,t) = \overline{\varphi}_t(u(s,t))$, and
moreover $J_\bullet$ is $(L_0, L_1)$-regular if and only if
$\overline{\varphi}_*J_\bullet$ is $(\varphi_1(L_0), L_1)$-regular.
\end{lemma}
\begin{proof}
Since
\begin{align*}
\frac{\partial v}{\partial s}(s,t) & = d \overline{\varphi}_t \left
(\frac{\partial u}{\partial s}
(s,t) \right) \quad \text{and} \\
\frac{\partial v}{\partial t}(s,t) & = d \overline{\varphi}_t \left
(\frac{\partial u}{\partial t} (s,t) - \chi'(t) X(\chi(t), u(s,t))
\right ),
\end{align*}
$u$ satisfies the Floer equation with Hamiltonian $H$ if and only if
$\overline{\varphi}_*u$ satisfies the Floer equation with
Hamiltonian $\boldsymbol{0}$. The map $\overline{\varphi}_*$ is
invertible because $\overline{\varphi}_t$ is for each $t$. Finally,
we observe that $d \overline{\varphi}$ intertwines the linearised
Floer operators at $u$ and $v$.
\end{proof}

Given two Legendrian submanifolds $\Lambda_0$ and $\Lambda_1$ in $(W
\times \R, dz+\theta)=(M, \beta)$, we say that $\Lambda_0$ is {\em
above} $\Lambda_1$ if the $z$-coordinate of any point of $\Lambda_0$
is larger than the $z$-coordinate of any point of $\Lambda_1$.
\begin{lemma}\label{CF and LCC compared}
Let $L_0$ and $L_1$ be immersed exact Lagrangian submanifolds of
$(W, \theta)$ and $H$ a cylindrical Hamiltonian compatible with
$L_0$ and $L_1$. We denote by $\varphi_t$ the Hamiltonian flow of
$H$ and $\widetilde{L}_0= \varphi_1(L_0)$. We choose Legendrian
lifts of $\widetilde{L}_0$ and $L_1$ to Legendrian submanifolds
$\widetilde{L}_0^+$ and $L_1^+$ of $(M, \beta)$ such that
$\widetilde{L}_0^+$ is above $L_1^+$. If $\widetilde{J}$ is an
$L$-regular almost complex structure on $W$ for $L=\widetilde{L}_0
\cup L_1$, let $J_\bullet =
(\overline{\varphi}_t^{-1})_*\widetilde{J}$. For every pair of
augmentations  $\varepsilon_0$ and $\varepsilon_1$ of the
obstruction algebras of $(L_0, J_0)$ and $(L_1, J_1)$ respectively,
there is an isomorphism of complexes
$$\mathrm{CF}((L_0, \varepsilon_0), (L_1, \varepsilon_1); H, J_\bullet) \cong
\mathrm{LCC}_{\widetilde{\varepsilon}_0, \varepsilon_1}(\widetilde{L}_0^+,
L_1^+; \widetilde{J})$$ where $\widetilde{\varepsilon}_0 =
\varepsilon_0 \circ \varphi_1^{-1}$ is an augmentation of the
obstruction algebra of $(\widetilde{L}_0, \widetilde{J})$.
\end{lemma}
\begin{proof}
By Lemma~\ref{lemma: naturality} there is an isomorphism of
complexes
$$\mathrm{CF}((L_0, \varepsilon_0), (L_1, \varepsilon_1); H, J_\bullet) \cong
\mathrm{CF}((\widetilde{L}_0, \widetilde{\varepsilon}_0), (L_1,
\varepsilon_1); \boldsymbol{0}, \widetilde{J}).$$ By definition the
obstruction algebras of $\widetilde{L}_0$ and $L_1$ are isomorphic
to the Chekanov-Eliashberg algebras of $\widetilde{L}_0^+$ and
$L_1^+$. As the chain complexes $\mathrm{CF}((\widetilde{L}_0,
\widetilde{\varepsilon}_0), (L_1, \varepsilon_1); \boldsymbol{0},
\widetilde{J})$ and $\mathrm{LCC}_{\widetilde{\varepsilon}_0,
\varepsilon_1}(\widetilde{L}_0^+, L_1^+; \widetilde{J})$ are both
generated by intersection points between $\widetilde{L}_0$ and
$L_1$ it remains only to match the differentials.

For any $i = 1, \ldots, d$ and $\boldsymbol{\zeta}= \{ \zeta_0,
\ldots, \zeta_d\} \in \widetilde{\mathcal R}^d$ there is a
biholomorphism $\Delta_{\boldsymbol{\zeta}} \cong
Z_{\boldsymbol{\zeta}^0, \boldsymbol{\zeta}^1}$, where
$\boldsymbol{\zeta}^0 = \{ \zeta_1, \ldots, \zeta_{i-1} \}$,
$\boldsymbol{\zeta}^1= \{ \zeta_{i+1}, \ldots, \zeta_d \}$,
$\zeta_i$ is mapped to $s = + \infty$ and $\zeta_0$ is mapped to
$s=- \infty$. Such biholomorphisms induce bijections between the
moduli spaces defining the boundary of $\mathrm{CF}((\widetilde{L}_0,
\widetilde{\varepsilon}_0), (L_1, \varepsilon_1); \boldsymbol{0},
\widetilde{J})$ and the moduli spaces defining the boundary of
$\mathrm{LCC}_{\widetilde{\varepsilon}_0, \varepsilon_1}(\widetilde{L}_0^+,
L_1^+; \widetilde{J})$.
\end{proof}

\subsection{Products}
\label{CF products}
After the work done for the differential, the higher order products can be easily defined.
For simplicity, in this section we will consider immersed exact Lagrangian submanifolds $L_0, \ldots, L_d$
which are pairwise transverse and cylindrical over chord generic Legendrian submanifolds.
Thus the generators of the Floer complexes will be intersection points, which we will assume
to be disjoint from the double points. The routine modifications needed to introduce
Hamiltonian functions into the picture are left to the reader.

Given $d \le 2$ and
$l_i \ge 0$ for $i=0, \ldots, d$ we define
$$\widetilde{\mathcal R}^{l_0| \ldots |l_d}= \op{Conf}^{l_0+ \ldots + l_d +d+1}(\partial D^2)$$
where $d+1$ points $\zeta_0^m, \ldots, \zeta_d^m$ (ordered counterclockwise) are labelled
as {\em mixed} and the other ones $\zeta^i_j$, with $i=0, \ldots, d$ and $j=1, \ldots, l_i$
(ordered counterclockwise and contained in the sector from $\zeta_i^m$ to $\zeta_{i+1}^m$)
are labelled as {\em pure}.
Given
$\boldsymbol{\zeta} \in \widetilde{\mathcal R}^{l_0| \ldots |l_d}$, we denote
$\Delta_{\boldsymbol{\zeta}} = D^2 \setminus \boldsymbol{\zeta}$. For $i=0, \ldots, d$ let
$\partial_i \Delta_{\boldsymbol{\zeta}}$ be the subset of $\partial
\Delta_{\boldsymbol{\zeta}}$ whose closure in $\partial D^2$ is the counterclockwise arc from
$\zeta_i^m$ to $\zeta_{i+1}^m$.

We will consider also a (generic) domain dependent almost complex structure $J_\bullet$ such
that every $J_z$, $z \in \Delta_{\boldsymbol{z}}$ satisfies \ref{acs}, and moreover is of the
form \ref{double acs} at the strip-like ends of the mixed punctures and is constant in a
neighbourhood of the arcs $\partial_i \Delta_{\boldsymbol{\zeta}}$ outside those strip-like ends.

Finally we define the moduli spaces $\mathfrak{M}_{L_0, \ldots, L_d}(\mathbf{p}^d, x_0,
\mathbf{p}^0, x_1, \ldots, \mathbf{p}^{d-1}, x_d; J_\bullet)$
of pairs $(u, \boldsymbol{\zeta})$ (up to action of $Aut(D^2)$), where:
\begin{itemize}
\item $\boldsymbol{\zeta} \in  \widetilde{\mathcal R}^{l_0| \ldots |l_d}$ and $u \colon
\Delta_{\zeta} \to W$ satisfies $du + J_\bullet \circ du \circ i =0$,
\item $u(\partial_i \Delta_{\boldsymbol{\zeta}}) \subset L_i$,
\item $\lim \limits_{z \to \zeta_i^m} u(z) = x_i$,
\item $\lim \limits_{z \to \zeta_j^i} u(z)= p_j^i$, and
\item $p_j^i$ is a negative puncture at $\zeta_j^i$ for $i=0, \ldots, d$ and $j = 1, \ldots, l_i$.
\end{itemize}
As usual, we denote by $\mathfrak{M}_{L_0, \ldots, L_d}^0(\mathbf{p}^d, x_0,
\mathbf{p}^0, x_1, \ldots, \mathbf{p}^{d-1}, x_d; J_\bullet)$ the
zero-dimensional part of the moduli spaces.

If $\varepsilon_0, \ldots, \varepsilon_d$ are augmentations for the corresponding Lagrangian
immersions, we define the weighted count
\begin{align*}
& \mathfrak{m}(\mathbf{p}^d, x_0, \mathbf{p}^0, x_1, \ldots, \mathbf{p}^{d-1}, x_d) = \\
&\# \mathfrak{M}_{L_0, \ldots, L_d}^0(\mathbf{p}^d, x_0,
\mathbf{p}^0, x_1, \ldots, \mathbf{p}^{d-1}, x_d; J_\bullet)
\varepsilon_0(\mathbf{p}^0)\ldots \varepsilon_d(\mathbf{p}^d)
\end{align*}
and define a product
\begin{equation}
\label{eqn: products}
\mu^d \colon \mathrm{CF}(L_{d-1}, L_d) \otimes \ldots \otimes \mathrm{CF}(L_0, L_1) \to \mathrm{CF}(L_0, L_d),
    \end{equation}
where we wrote $\mathrm{CF}(L_i, L_j)$ instead of $\mathrm{CF}((L_i, \varepsilon_i), (L_j, \varepsilon_j))$
for brevity sake, via the formula
\begin{equation*}
 \mu^d(x_1, \ldots, x_d) = \sum_{x_0 \in L_0 \cap L_d} \sum_{l_i \ge 0} \sum_{p_j^i \in D_i^{l_i}}
\mathfrak{m}(\mathbf{p}^d, x_0, \mathbf{p}^0, x_1, \ldots, \mathbf{p}^{d-1}, x_d) x_0.
\end{equation*}
\begin{rem}
The operations $\mu^d$ satisfy the $A_\infty$ relations.
\end{rem}

The following lemma will be useful in Section \ref{sec: surgeries and cones}. It is a straightforward corollary \color{black} of the existence of pseudoholomorphic triangles supplied by Corollary \ref{cor: triangles} below. The only point where we need to take some care is due to the fact that the Weinstein neighbourhood considered is immersed.
\begin{lemma}\label{continuation element}
Let $(L, \iota)$ be an exact Lagrangian immersion. We extend $\iota$ to a symplectic immersion $\iota_* \colon (D_\delta T^*L, d\mathbf{q}\wedge d\mathbf{p}) \to (W, d \theta)$. Let $(L, \iota')$ be safe isotopic to $(L, \iota)$ and, moreover, assume that
\begin{itemize}
\item there exists a sufficiently $C^1$-small Morse function $g \colon L \to \R$ with local minima $e_i$, all whose critical points are disjoint from the double points of $L,$ such that $\iota' = \iota_* \circ dg$, and
\item outside of some compact subset $L'$ is obtained by a small perturbation of $L$ by the negative Reeb flow.
\end{itemize}
We will denote $L=(L, \iota)$ and $L' = (L, \iota')$. Then, if $L$ admits an  augmentation
$\varepsilon$ and $\varepsilon'$ is the corresponding augmentation of $L'$, for every
cylindrical exact Lagrangian submanifold $T$ such that $\iota_*^{-1}(T)$ is a union of
cotangent disc fibres, the map
\begin{gather*}
\mu^2 (e, \cdot) \colon \mathrm{CF}(T, (L, \varepsilon)) \to \mathrm{CF}(T, (L', \varepsilon')),\\
e:=\sum_{i} e_i \in \mathrm{CF}((L,\varepsilon),(L',\varepsilon')),
\end{gather*}
is an isomorphism of complexes for a suitable almost complex structure on $W$ as in Corollary \ref{cor: triangles}.
\end{lemma}
In the case when $L$ is closed and embedded, the element $e$ is always a cycle which is
nontrivial in homology as was shown by Floer (it is identified with the minimum class in the Morse cohomology of $L$). In general the following holds.
\begin{lemma}
\label{unit}
Under the hypotheses of Lemma \ref{continuation element}, the element
$$e \in \mathrm{CF}((L,\varepsilon),(L',\varepsilon'))$$
is a cycle. Furthermore, $e$ is a boundary if and only if $\mathrm{CF}(T,(L,\varepsilon))=0$ for every Lagrangian $T.$
\end{lemma}
\begin{proof}
The assumption that the augmentation $\varepsilon'$ is identified canonically with the augmentation $\varepsilon$ implies that $e$ is a cycle by the count of  pseudoholomorphic discs \color{black} with a negative puncture at $e$ supplied by Lemma \ref{lem: triangles on twocopy}.

Assume that $\partial E=e$. The last property is then an algebraic consequence of the Leibniz rule $\partial\mu^2(E,x)=\mu^2(\partial E,x)$ in the case when $\partial x=0$, combined with the fact that $\mu^2(e,\cdot)$ is a quasi-isomorphism as established by the previous lemma.
\end{proof}

Later it will be useful to switch perspectives slightly,
and instead of with the chain $e$, work with an augmentation induced
by that chain. In general, given $(L_0, \varepsilon_0)$, $(L_1,
\varepsilon_1)$ and a chain $c \in
\mathrm{CF}((L_0,\varepsilon_0),(L_0,\varepsilon_1))$ we can consider
Legendrian lifts $L_0^+$ and $L_1^+$ such that $L_0^+$ is
above $L_1^+$ and the unital algebra morphism
$\varepsilon_c \colon \mathfrak{A}(L_0^+ \cup L_1^+) \to \F$
uniquely determined
$$ \varepsilon_c(x) \coloneqq \begin{cases}
\varepsilon_i(x),  & \text{if } x\in \mathfrak{A}(L_i),\\
\langle c,x\rangle, & \text{if } x \in L_0 \cap L_1,
\end{cases}
$$
where $\langle \cdot,\cdot \rangle$ is the coefficient of $x$ in $c$.
\begin{lemma}
\label{lem:CyclesAugmentations}
The element
$$c \in \mathrm{CF}((L_0,\varepsilon_0),(L_1,\varepsilon_1);\boldsymbol{0},J)$$
is a cycle if and only if
$$\varepsilon_c \colon \mathfrak{A}(L_0^+ \cup L_1^+) \to \F$$
is an augmentation, where the almost complex structure $J$ has been used to define the latter algebra as well, and the Legendrian lifts have been chosen so that no Reeb chord starts on $L_0^+$ and ends on $L_1^+$.
\end{lemma}
\begin{proof}
Note that the Floer complex under consideration has a differential which counts $J$-pseudo\-holomorphic  strips, and that the obstruction algebra has a differential counting pseudoholomorphic discs with at least one boundary puncture. Identifying the appropriate counts of discs, the statement can be seen to follow by pure algebra, together with the fact that the differential of the DGA counts punctured pseudoholomorphic discs, and thus respects the filtration induced by the different components. The crucial property that is needed here is that, under the assumptions made on the Legendrian lifts, the differential of the Chekanov-Eliashberg algebra applied to a mixed chord is a sum of words, each of which contains precisely one mixed chord.
\end{proof}

\subsection{Existence of triangles}
In this section we prove an existence result for small pseudoholomorphic triangles with boundary on an exact Lagrangian cobordism and a small push-off. The existence of these triangles can be deduced as a consequence of the fact that the wrapped Fukaya category is homologically unital. Here we take a more direct approach based upon the adiabatic limit of pseudoholomorphic discs on a Lagrangian and its push-off from \cite{Duality_EkholmetAl}; when the latter push-off becomes sufficiently small, these discs converge to pseudoholomorphic discs on the single Lagrangian with gradient-flow lines attached (called \emph{generalised pseudoholomorphic discs} in the same paper).

Let $L \subset W$ be an exact immersed Lagrangian submanifold with cylindrical end. We recall that, as usual, we assume that every immersed Lagrangian submanifold is nice. Consider the Hamiltonian push-off $L_{\epsilon f}$, which we require to be again an exact immersed Lagrangian submanifold with cylindrical end, which is identified with the graph of $d(\epsilon f)$ for a Morse function $f \colon L \to \R$ inside a Weinstein neighbourhood $(T^*_\delta L,-d(pdq)) \looparrowright (W,d\theta)$ of $L$. We further assume that $df({\mathcal L})>0$ outside of a compact subset. (The assumption that the push-off is cylindrical at infinity does of course impose additional constraints on the precise behaviour of the Morse function outside of a compact subset.)

Now consider a Legendrian lift $L^+ \cup L_{\epsilon f}^+$ for which $(L_{\epsilon f})^+$ is above $L^+$. For $\epsilon>0$ sufficiently small, it is the case that $L \cup L_{\epsilon f}$ again has only transverse double points. Moreover, the Reeb chords on the Legendrian lift can be classified as follows, using the notation from \cite[Section 3.1]{Duality_EkholmetAl}:
\begin{itemize}
\item Reeb chords $\mathcal{Q}(L) \cong \mathcal{Q}(L_{\epsilon f})$ on the lifts of $L$ and $L_{\epsilon f}$ respectively, which stand in a canonical bijection;
\item Reeb chords $\mathcal{C}$ being in a canonical bijection with the critical points of $f$; and
\item two sets $\mathcal{Q}$ and $\mathcal{P}$ of Reeb chords from $L$ to $L_{\epsilon f}$, each in canonical bijection with $\mathcal{Q}(L)$, and where the lengths of any Reeb chord in $\mathcal{Q}$ is greater than the length of any Reeb chord in $\mathcal{P}$.
\end{itemize}
See the aforementioned reference for more details, as well as Figure \ref{fig:triangles}.

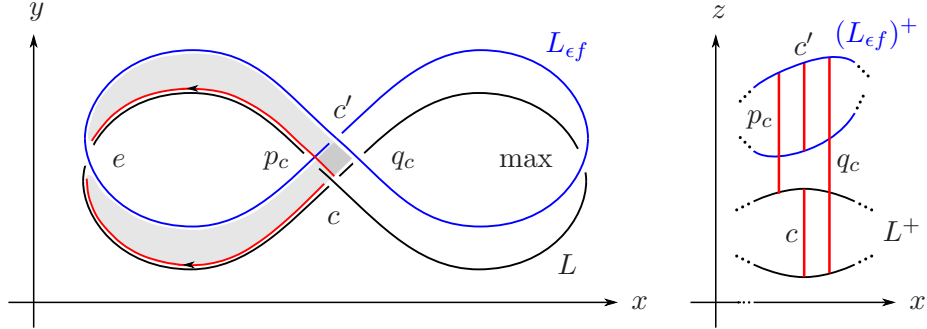
\begin{figure}[t]
\labellist
\pinlabel $c$ at 103 35
\pinlabel $q_c$ at 125 53
\pinlabel $e$ at 35 53
\pinlabel $\text{max}$ at 164 53
\pinlabel $p_c$ at 85 53
\pinlabel $c'$ at 106 70
\pinlabel $\color{blue}L_{\epsilon f}$ at 177 89
\pinlabel $L$ at 177 20
\pinlabel $x$ at 200 8
\pinlabel $x$ at 288 8
\pinlabel $L^+$ at 283 31
\pinlabel $\color{blue}(L_{\epsilon f})^+$ at 275 93
\pinlabel $y$ at 9 100
\pinlabel $z$ at 225 100
\pinlabel $c$ at 248 29
\pinlabel $c'$ at 252 90
\pinlabel $q_c$ at 266 52
\pinlabel $p_c$ at 238 66
\endlabellist
\vspace{3mm} \centering
\includegraphics[scale=1.2]{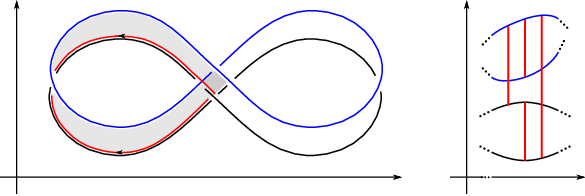}
\caption{The small triangles on the two-copy living near gradient flow-lines of $-\nabla f$ shown in red. The upper copy with respect to the $z$-coordinate is shown in blue, while the lower copy is in black. The contact form used here is $dz-ydx$.} \label{fig:triangles}
\end{figure}

\begin{lemma}[\cite{Duality_EkholmetAl}]
\label{lem: triangles on twocopy}
For a suitable generic Riemannian metric $g$ on $L$ for which $(f,g)$ constitutes a Morse-Smale pair and associated almost complex structure, which can be made to coincide with an arbitrary cylindrical almost complex structure outside of a compact subset, there is a bijection between the set of pseudoholomorphic discs which have
\begin{itemize}
\item boundary on $L^+ \cup L_{\epsilon f}^+$ and precisely one positive puncture,
\item at least one negative puncture at a local minimum $e \in \mathcal{C}$ of $f$, and
\item form a moduli space of expected dimension zero,
\end{itemize}
and the set of negative gradient flow-lines on $(L,g)$ that connect either the starting point or the end point of a Reeb chord $c \in \mathcal{Q}(L)$ with the local minimum $e$, together with the set of negative gradient flow-lines that connect some critical point of index one with $e$.

More precisely, each such pseudoholomorphic disc lives in an small neighbourhood of the aforementioned flow-line. In the first case, it is a triangle with a positive puncture at the Reeb chord $q_c \in \mathcal{Q}$ corresponding to $c$, and its additional negative punctures at $e$ and either $c$ (for the flow-line from the \emph{starting point} of $c$) or $c'$ (for the flow-line from the \emph{endpoint} of $c$); see Figure \ref{fig:triangles}. In the second case, it is a Floer strip corresponding to the negative gradient flow-line connecting the saddle point and the local minimum.
\end{lemma}
\begin{proof}
This is an immediate application of Parts (3) and (4) of \cite[Theorem 5.5]{Duality_EkholmetAl}. A generalised pseudoholomorphic disc with a negative puncture at a local minimum can be rigid only if it a flow-line connecting a saddle point to the minimum, or consists of a constant pseudoholomorphic disc located at one of the Reeb chords at $c \in \mathcal{Q}(L)$ together with a flow-line from that double point to the local minimum. The aforementioned result gives a bijection between such generalised pseudoholomorphic discs and pseudoholomorphic strips and triangles on the two-copy.
\end{proof}

Now consider an auxiliary exact immersed Lagrangian $L'$ intersecting $L \cup L_{\epsilon f}$ transversely. For $\epsilon>0$ sufficiently small there is a bijection between the intersection points $L \cap L'$ and $L_{\epsilon f} \cap L'.$
\begin{cor}
\label{cor: triangles}
For a suitable Morse-Smale pair $(f,g)$ and almost complex structure as in Lemma \ref{lem: triangles on twocopy} there is a unique rigid and transversely cut out pseudoholomorphic triangle with corners at $e$, $c \in L \cap L_{\epsilon f}$ and the corresponding double point $c' \in L_{\epsilon f}$ for any connected gradient flow-line from $c \in L \cap L_{\epsilon f}$ to the local minimum $e \in \mathcal{C}$. The triangle is moreover contained inside a small neighbourhood of the same flow-line.
\end{cor}
\begin{proof}
We need to apply Lemma \ref{lem: triangles on twocopy} in the case when $L^+$ is taken to be the Legendrian lift $(L \cup L_{\epsilon f})^+$, where $L_{\epsilon f}^+$ is above $L^+$, and the push-off is taken to be $(L \cup L')_{\epsilon F}^+$ for a Morse function $F \colon L \cup L' \to \R$ that restricts to $f$ on $L$.
\end{proof}

\section{Continuation maps}\label{sec: continuation maps}
In this section we analyse what happens to the Floer cohomology when we change $J$, $H$
(in some suitable way) or move the Lagrangian submanifolds by a compactly supported safe
exact isotopy.

\subsection{Changing the almost complex structure}\label{sss: changing J}
Following~\cite{LCHgeneral} (see also \cite{FloerHFlag}) we will use the bifurcation
method to prove invariance of Floer cohomology for Lagrangian intersection under change of
almost complex structure.  It seems, in fact, that the more usual continuation method is not
well suited to describe how the obstruction algebras change when the almost complex
structure changes.

Let us fix Lagrangian immersions $(L_0, \iota_0)$ and $(L_1, \iota_1)$ and a cylindrical
Hamiltonian $H$ compatible with $L_0$ and $L_1$. For a generic one-parameter family of
 time-dependent almost complex structure $J_\bullet^\bullet$ parametrised by an interval
$[\delta_-, \delta_+] $ such that
\begin{itemize}
\item the extrema $J_\bullet^{\delta_-}$ and $J_\bullet^{\delta_+}$ are $(L_0, L_1)$-compatible,
and
\item $J_\bullet^\delta$ satisfies \ref{double acs} for all $\delta \in [\delta_-, \delta_+]$
\end{itemize}
we define the {\em parametrised} moduli spaces $$\mathfrak{M}^k_{L_0, L_1}(\mathbf{p}^1,
x_-, \mathbf{p}^0 , x_+; H, J_\bullet^\bullet)$$ consisting of pairs $(\delta, u)$ such that $\delta
\in [\delta_-, \delta_+]$ and $$u \in \mathfrak{M}^k_{L_0, L_1}(\mathbf{p}^1, x_-,
\mathbf{p}^0, x_+; H, J_\bullet^\delta).$$

Using the zero-dimensional parametrised moduli spaces, we will define a continuation map
$$\Upsilon_{J_\bullet^\bullet} \colon \mathrm{LCC}((L_0, \varepsilon_0^+), (L_1, \varepsilon_1^+); H,
J_\bullet^{\delta_+}) \to \mathrm{LCC}((L_0, \varepsilon_0^-), (L_1, \varepsilon_1^-); H,
J_\bullet^{\delta_-}).$$
\begin{prop}
For a generic one-parameter family $J_\bullet^\bullet$ of time-dependent almost complex
structures as above, the parametrised moduli space
$$\mathfrak{M}^k_{L_0, L_1}(\mathbf{q}^1, y_-, \mathbf{q}^0, y_+; H,
J_\bullet^\bullet)$$
is a transversely cut-out manifold of dimension $k$.  If $k = 0$ it is compact, and
therefore consists of a finite number of points. If $k = 1$ it can be compactified in the
sense of Gromov-Floer.

If both $L_0$ and $L_1$ are spin, the moduli spaces are orientable, and a choice of  spin
structure on each Lagrangian submanifold induces a coherent orientation on the
parametrised moduli spaces.
\end{prop}

In the following lemma we look more closely at the structure of the zero-dimensional
parametrised moduli spaces. The analogous statement in the setting of Lagrangian Floer
homology (for Lagrangian submanifolds) appears in \cite[Section 3]{FloerHFlag}.
In the case of Legendrian contact cohomology, the corresponding construction appears in
\cite[Section 2.4]{LCHgeneral}.
\begin{lemma}
For a generic $J_\bullet^\bullet$ there is a finite set $\Delta \subset (\delta_-, \delta_+)$
such that
for $\delta \in \Delta$ exactly one of the following cases holds:
\begin{itemize}
\item[(i)]  there is a unique nonempty moduli space $\mathfrak{N}^{-1}_{L_0}(q_0^0;
q_1^0, \ldots, q_d^0; J_0^\delta)$ and all other moduli spaces are transversely cut out,
\item[(ii)]   there is a unique nonempty moduli space $\mathfrak{N}^{-1}_{L_1}(q_0^1;
q_1^1, \ldots, q_d^1; J_1^\delta)$  and all other moduli spaces are transversely cut out, or
\item[(iii)] there is a unique nonempty  moduli space
$\mathfrak{M}_{L_0, L_1}^0(\mathbf{q}^1, y_-, \mathbf{q}^0, y_+; H, J_\bullet^\delta)$ and
all other moduli spaces are transversely cut out,
\end{itemize}
while for every $\delta \in [\delta_-, \delta_+] \setminus \Delta$ the moduli spaces of
negative virtual dimension are empty.
\end{lemma}
(Of course, the self-intersection points $q_j^i$ appearing in the three cases above have
nothing to do with each other.)  Note that the lemma
does not claim that $J_\bullet^\delta$ is $(L_0, L_1)$-regular for $\delta \not \in \Delta$: for
example, if $\delta$ is a critical value of the projection
$$\mathfrak{M}_{L_0, L_1}^1(\mathbf{q}^1, y_-, \mathbf{q}^0, y_+;
H, J_\bullet^\bullet) \to [\delta_-, \delta_+],$$
then $J_\bullet^\delta$ is not $(L_0, L_1)$-regular, but $\delta \not \in \Delta$.

\begin{rem}
If $\Delta = \emptyset$, then the $\mathrm{CF}((L_0, \varepsilon_0), (L_1, \varepsilon_1); H,
J_\bullet^{\delta_-})$ and $\mathrm{CF}((L_0, \varepsilon_0), (L_1, \varepsilon_1); H,
J_\bullet^{\delta_+})$ are isomorphic. In fact the one-dimensional parametrised moduli spaces
have boundary points only at $J_\bullet^{\delta_-}$ and $J_\bullet^{\delta_+}$, and this implies
that the algebraic count of elements of the zero-dimensional moduli spaces is the same
for  $J_\bullet^{\delta_-}$ and $J_\bullet^{\delta_+}$.
\end{rem}
Next we describe what happens when we cross $\delta \in \Delta$ of type (i) or (ii). Since the
cases are symmetric, we will describe only (i) { and assume, without loss of generality, that $\delta=0$.}
\begin{lemma}\label{variation of J cases i and ii}
Suppose that $\Delta = \{ 0 \}$ and that the unique nontransversely cut out moduli space for $J_\bullet^{0}$ is $\mathfrak{N}^{-1}_{L_0}(q_0^0; q_1^0, \ldots, q_d^0;
J_0^{0})$. Then for $\delta>0$ the  differentials of
$\mathrm{CF}((L_0, \varepsilon_0^-), (L_1, \varepsilon_1^-); H, J_\bullet^{-\delta})$ and\\
$\mathrm{CF}((L_0, \varepsilon_0^+), (L_1, \varepsilon_1^+); H, J_\bullet^{\delta})$ are equal. \color{black}
\end{lemma}
\begin{proof}
We adapt the cobordism method of \cite{Ekholm_&_Orientation_Homology} (see also
\cite{LCHgeneral}). Given a positive Morse function $f \colon \R \to \R$ with local minima at $\pm 1$ satisfying $f(1)= f(-1)=1$, a local maximum at $0$ and no other critical point, we define exact Lagrangian immersions $\widetilde{\iota}_i \colon \R \times L_i \to T^*\R \times W$ following \cite[Subsection 4.3.2]{Ekholm_&_Orientation_Homology}. We denote
$\widetilde{L}_i$ the image of $\widetilde{\iota}_i$.

Let $(r, \rho)$ be the canonical coordinates on $T^*\R$. We consder the Hamiltonian function
$$\widetilde{H} \colon [0, 1] \times T^*\R \times W \to \R$$
such that $\widetilde{H}(t, r, \rho, w)=H(t,w)$.
Then its Hamiltonian vector field $X_{\widetilde{H}}$ is tangent to $\{(r, \rho) \} \times W$ for all $(r, \rho) \in T^*\R$.
Each double point $q$ of $L_i$ gives rise to three double points $q[-1], q[0], q[+1]$ of $\widetilde{L}_i$ and each each Hamiltonian chord $x$ from $L_0$ to $L_1$ gives rise to three Hamiltonian chords $x[-1]$, $x[-1], x[+1]$ from $\widetilde{L}_0$ to $\widetilde{L}_1$.

Let $\widetilde{J}_\bullet$ be the (time dependent) almost complex structure on $T^*\R \times W$ such that on $T_{(r, \rho, w)} \cong \R^2 \oplus T_wW$ is $j_0 \oplus J^{\alpha_\delta(r)}_\bullet$ where $j_0$ is the standard almost complex structure on $\R^2$ and $\alpha_\delta \colon \R \to [- \delta, \delta]$ is as in Equation (4.9) of \cite{Ekholm_&_Orientation_Homology} for $\delta$ small. Proposition~\ref{invariance of D} provides isomorphisms
$$\mathfrak{Y}_i \colon (\mathfrak{D}_i^-, \mathfrak{d}_i^-) \to (\mathfrak{D}_i^+,
\mathfrak{d}_i^+)$$
such that $\varepsilon_i^- = \varepsilon_i^+ \circ \mathfrak{Y}_i$. Those isomorphisms are the ones constructed in \cite[Subsection 4.4.3]{Ekholm_&_Orientation_Homology} and denoted by $\psi$ there. Let $\widetilde{\mathfrak{D}}_i$ be the obstruction algebra of
$\widetilde{L}_i$ and define $\widetilde{\varepsilon}_i \colon \widetilde{\mathfrak{D}}_i \to \F$ as
$$\begin{cases} \widetilde{\varepsilon}_i(q[-1])=\varepsilon_-(q), \\ \widetilde{\varepsilon}_i(q[0])= 0, \\ \widetilde{\varepsilon}_i(q[+1])=\varepsilon_+(q)
\end{cases}$$
for every double point $q$ of $L_i$.
From the structure of the differential of the Chekanov-Eliashberg algebras of $\widetilde{L}_i$ described in  \cite[Subsection 4.4.3]{Ekholm_&_Orientation_Homology} it follows that $\widetilde{\varepsilon}_i$ is an augmentation of $\widetilde{\mathfrak{D}}_i$, and therefore the Floer chain complex $\mathrm{CF}((\widetilde{L}_0, \widetilde{\varepsilon}_0), (\widetilde{L}_1, \widetilde{\varepsilon}_1); \widetilde{H}, \widetilde{J}_\bullet)$ is well defined.

Lemmas 4.14, 4.14, 4.18 and 4.19 of \cite{Ekholm_&_Orientation_Homology} have direct counterparts for Floer solutions because the projection of a solution of the Floer equation with Hamiltonian $\widetilde{H}$ to $T^*\R$ is
a holomorphic map. This implies that the differential $\widetilde{\partial}$ on $\mathrm{CF}((\widetilde{L}_0, \widetilde{\varepsilon}_0), (\widetilde{L}_1, \widetilde{\varepsilon}_1); \widetilde{H}, \widetilde{J}_\bullet)$ has the following form
\begin{equation} \label{partialtilde}
\begin{cases}
\widetilde{\partial}(x[-1])=(\partial_- x)[-1], \\
\widetilde{\partial}(x[0])= x[+1]-x[-1]+ \sum \limits_{y \in {\mathcal C}_H} a_y y[0], \\
\widetilde{\partial}(x[+1])=(\partial_+ x)[+1]
\end{cases}
\end{equation}
where $\partial_\pm$ denotes the differentials of $\mathrm{CF}((L_0, \varepsilon_0^\pm), (L_1, \varepsilon_1^\pm); H, J_\bullet^{\pm\delta})$.
The crucial point here is that no chord $y[\pm 1]$ contributes to $\widetilde{\partial}(x[0])$ if $y \ne x$.  This is a consequence of the assumption that there is no Floer strip of index $0$ for $J^\bullet_\bullet$ and of the Hamiltonian version of \cite[Lemma 4.19]{Ekholm_&_Orientation_Homology}. From $\widetilde{\partial}^2=0$ it is easy to see that $\partial_+=\partial_-$.
\end{proof}
\color{black}

Now we analyse how the complex changes when we cross $\delta \in \Delta$ of type (iii).
\begin{lemma}\label{variation of J case iii}
Suppose that $\Delta = \{ \delta_0 \}$ and that the unique nontransversely cut out moduli
space for $J_\bullet^{\delta_0}$ is $\mathfrak{M}_{L_0, L_1}^0(\mathbf{q}^1, y_-,
\mathbf{q}^0, y_+; H, J_\bullet^\delta)$. Then the map
$$\Upsilon_{J_\bullet^\bullet} \colon \mathrm{CF}((L_0, \varepsilon_0), (L_1, \varepsilon_1); H,
J_\bullet^{\delta_+}) \to \mathrm{CF}((L_0, \varepsilon_0), (L_1, \varepsilon_1); H, J_\bullet^{\delta_-})$$
defined as
$$\Upsilon_{J_\bullet^\bullet} (x) = \begin{cases}
x & \text{if } x \ne y_+, \\
y_+  + \# \mathfrak{M}_{L_0, L_1}^0(\mathbf{q}^1, y_-,
\mathbf{q}^0, y_+; H, J_\bullet^\bullet) \varepsilon_0(\mathbf{q}^0)
\varepsilon_1(\mathbf{q}^1)y_- &  \text{if } x = y_+
\end{cases}$$
is an isomorphism of complexes.
\end{lemma}
\begin{proof}
The proof is the same as in \cite{FloerHFlag}.
However,  the proof in \cite{FloerHFlag} holds only in the case of $\Z_2$-coefficients.
For more general coefficients, we rely on the discussion in \cite{Ekholm_&_Orientation_Homology, LCHgeneral}.
\end{proof}
Given a generic homotopy $J_\bullet^\bullet$, we split it into pieces containing only one point of
$\Delta$ and compose the maps obtained in Lemma~\ref{variation of J cases i and ii}
and~\ref{variation of J case iii}.

\subsection{Changing the Hamiltonian}\label{sss: changing hamiltonian}
In this section we will keep the almost complex structure fixed.
Let $H_-$ and $H_+$ be time-dependent cylindrical Hamiltonian functions which are
compatible with immersed Lagrangian submanifolds $L_0$ and $L_1$.
From a one-parameter family of cylindrical Hamiltonians $H_s$ such that
\begin{itemize}
\item[(i)] $H_s=H_-$ for $s \ll 0$,
\item[(ii)] $H_s =H_+$ for all $s \gg 0$, and
\item[(iii)] $\partial_sh'_s(e^{\mathfrak{r}(w)}) \le 0$ if $\mathfrak{r}(w)$ is sufficiently large,
\end{itemize}
we will define a continuation map
$$\Phi_{H_s} \colon \mathrm{CF}((L_0, \varepsilon_0), (L_1, \varepsilon_1) ; H_+) \to
\mathrm{CF}((L_0, \varepsilon_0), (L_1, \varepsilon_1); H_-).$$

Given a time-dependent almost complex structure $J_\bullet$, an $H_-$-Hamiltonian chord
$x_-$, an $H_+$-Hamiltonian chord $x_+$ and double points $\mathbf{p}^0 =
(p_1^0, \ldots, p_{l_0}^0)$ of $L_0$ and $\mathbf{p}^1 = (p_1^1, \ldots, p_{l_1}^1)$ of
$L_1$, we define the moduli spaces
$$\mathfrak{M}_{L_0, L_1}(\mathbf{p}^1, x_-, \mathbf{p}^0, x_+; H_s, J_\bullet)$$
as the set of triples $(\boldsymbol{\zeta}^0, \boldsymbol{\zeta}^1, u)$
such that:
\begin{itemize}
\item $(\boldsymbol{\zeta}^0, \boldsymbol{\zeta}^1) \in \widetilde{\mathcal R}^{l_0|l_1}$
and $u \colon Z_{\boldsymbol{\zeta}^0, \boldsymbol{\zeta}^1} \to W$ is a map satisfying the
Floer equation
\begin{equation}\label{continuation1}
\frac{\partial u}{\partial s} + J_t \left (\frac{\partial u}{\partial t} - \chi'(t)
X_{H_s}(\chi(t), u) \right )=0,
\end{equation}
\item $\lim \limits_{s \to \pm \infty}u(s, t) = x_\pm(\chi(t))$,
\item $u(s,0) \in L_0$ for all $(s, 0) \in Z_{\boldsymbol{\zeta}^0, \boldsymbol{\zeta}^1}$,
\item $u(s,1) \in L_1$ for all $(s,1) \in Z_{\boldsymbol{\zeta}^0, \boldsymbol{\zeta}^1}$,  and
\item each $\zeta_j^i$ is a negative puncture at $p_j^i$ for $i=0,1$ and $j=1, \ldots, l_i$.
\end{itemize}
Note the only difference between Equation~\eqref{continuation1} and
Equation~\eqref{eq: Floer} is that we made $X_{H_s}$ depend on $s$ in
Equation~\eqref{continuation1}. For this reason there is no action of $\op{Aut}(Z)$
on the
moduli spaces $\mathfrak{M}_{L_0, L_1}(\mathbf{p}^1, x_-, \mathbf{p}^0, x_+; H_s, J_\bullet)$.

Let $F_u$ be the linearisation at $u$ of the Floer operator with $s$-dependent
Hamiltonian.  We define $$\op{ind}(u)=  \op{ind}(F_u) +l_0 +l_1,$$ and define
$\mathfrak{M}_{L_0, L_1}^k(\mathbf{p}^0, x_-, \mathbf{p}^1, x_+; H_s, J_\bullet)$
as the subset of $$\mathfrak{M}_{L_0, L_1}(\mathbf{p}^1, x_-, \mathbf{p}^0, x_+; H_s,
J_\bullet)$$ consisting of the maps $u$ with $\op{ind}(u)=k$.

The following statement is analogous to the statement in Morse theory, \cite[Section 3]{FloerHFlag}. A similar boundary degeneration statement in the case of Legendrian contact cohomology appears in \cite[Section 2.4]{LCHgeneral}.
\begin{prop}
Given $H_s$, for a generic time-dependent almost complex structure $J_\bullet$ satisfying
\ref{double acs} with respect to both $H_+$ and $H_-$ , the moduli space
$\mathfrak{M}_{L_0, L_1}^k (\mathbf{p}^1, x_-, \mathbf{p}^0, x_+; H_s, J_\bullet)$ is a
transversely cut-out manifold of dimension $k$. If $k=0$ it is compact, and therefore
consists of a finite set of points. If $k=1$ it can be compactified in the sense of Gromov-Floer.

If both $L_0$ and $L_1$ are spin, the choice of a spin structure on each
induces a coherent orientation {of the moduli space (see \cite{Ekholm_&_Orientation_Homology})}.
\end{prop}

We denote ${\mathcal C}_-$ the set of Hamiltonian chords of $H_-$ and ${\mathcal C}_+$
the set of Hamiltonian chords of $H_+$. We also introduce the weighted
count
$$\mathfrak{m}(\mathbf{p}^1, x_-  \mathbf{p}^0, x_+; H_s)=
\# \mathfrak{M}_{L_0,L_1}^0( \mathbf{p}^1, x_-  \mathbf{p}^0, x_+; H_s,
J_\bullet) \varepsilon_0( \mathbf{p}^0) \varepsilon_1 ( \mathbf{p}^1).$$
Given $x_+ \in {\mathcal C}_+$, we
define the continuation map
\begin{equation}\label{continuation: change of hamiltonian}
\Phi_{H_s}(x_+) = \sum_{x_- \in {\mathcal C}_-} \sum_{l_0, l_1 \in \N} \sum_{\mathbf{p}^i \in D_i^{l_i}}
\mathfrak{m}(\mathbf{p}^1, x_-  \mathbf{p}^0, x_+; H_s) x_-.
\end{equation}
The Gromov-Floer compactification of the one-dimensional moduli spaces implies the
following lemma.
\begin{lemma}
The map $\Phi_{H_s}$ is a chain map.
\end{lemma}
 We denote by
$$\Phi_{H_-, H_+}^* \colon \mathrm{HF}((L_0, \varepsilon_0), (L_1, \varepsilon_1) ; H_+) \to \mathrm{HF}((L_0,
\varepsilon_0), (L_1, \varepsilon_1); H_-)$$
the map induced in homology by $\Phi_{H_s}$ --- soon it will be apparent that the notation is
justified.
As it happens in the more standard Floer cohomology for Lagrangian submanifolds, the
continuation maps satisfy the following properties.
\begin{lemma}\label{properties of Phi}
\begin{enumerate}
\item Up to homotopy, $\Phi_{H_s}$ depends only on the endpoints $H_+$ and $H_-$ of
$H_s$, \label{qui}
\item $\Phi^*_{H,H}$ is the identity for every $H$, and \label{quo}
\item $\Phi^*_{H_-,H} \circ \Phi^*_{H, H_+} = \Phi^*_{H_-, H_+}$. \label{qua}
\end{enumerate}
\end{lemma}
\begin{proof}[Sketch of proof]
In order to prove \eqref{qui}, we follow the standard procedure for defining chain homotopies
in Floer theory; see \cite{FloerHFlag} for more details. Given a homotopy $H_s^\delta$,
$\delta \in [0,1]$, between $s$-dependent Hamiltonian functions $H_s^0$ and $H_s^1$ with
$H_s^\delta \equiv H_-$ for $s \ll 0$ and $H_s^\delta \equiv H_+$ for $s \gg 0$, we define
the parametrised moduli spaces
$\mathfrak{M}_{L_0, L_1}^k(\mathbf{p}^1, x_-, \mathbf{p}^0, x_+; H_s^\bullet, J_\bullet)$
of pairs $(\delta, u)$ such $\delta \in [0,1]$ and $u \in \mathfrak{M}_{L_0, L_1}^k
(\mathbf{p}^1, x_-, \mathbf{p}^0, x_+; H_s^\delta, J_\bullet)$.
We define the weighted count
$$\mathfrak{m}(\mathbf{p}^1, x_-,  \mathbf{p}^0, x_+; H_s^\bullet) =  \#
\mathfrak{M}_{L_0,L_1}^{-1}( \mathbf{p}^1, x_-,  \mathbf{p}^0, x_+; H_s^\bullet, J_\bullet)
\varepsilon_0( \mathbf{p}^0) \varepsilon_1 ( \mathbf{p}^1).$$
Then the chain homotopy
$$K \colon \mathrm{CF}((L_0, \varepsilon_0), (L_1, \varepsilon_1) ; H_+, J_\bullet) \to \mathrm{CF}((L_0,
\varepsilon_0), (L_1, \varepsilon_1); H_-, J_\bullet)$$
between $\Phi_{H_s^0}$ and $\Phi_{H_s^1}$ is defined as
$$K(x_+) = \sum_{x_- \in {\mathcal C}_-} \sum_{l_0, l_1 \in \N} \sum_{\mathbf{p}^i \in D_i^{l_i}}
\mathfrak{m}(\mathbf{p}^1, x_-,  \mathbf{p}^0, x_+; H_s^\bullet)x_-.$$

In order to prove \eqref{quo} we can choose $H_s \equiv H$: then the moduli space
$\mathfrak{M}_{L_0, L_1}^0(\mathbf{p}^0, x_-, \mathbf{p}^1, x_+; H_s, J_\bullet)$ consists
of constant strips.

We fix $s$-dependent Hamiltonian function $H_s^+$ and $H_s^-$ such that $H_s^+=H_+$
for $s \ge 1$ and $H_s^+=H$ for $s \le 0$, and  $H_s^-=H$ for $s \ge 0$ and $H_s^-=H_-$
for $s \le -1$. In order to prove \eqref{qua} we introduce the family of Hamiltonian functions
$$H_s^R = \begin{cases} H_{s-R}^+ & \text{for } s \ge 0, \text{ and }\\
                                              H_{s+R}  & \text{for } s \le 0
\end{cases}$$
with $R>0$. By \eqref{qui}, $\Phi_{H_s^R}$ induces $\Phi^*_{H_+, H_-}$ for all $R$.
For $R \gg 0$ there is an identification
\begin{align}\label{product of modulispaces}
 & \mathfrak{M}_{L_0,L_1}^0(p_1^1, \ldots, p_{l_1}^1, x_-,  p_1^0, \ldots, p_{l_0}^0, x_+;
H_s^R,J_\bullet) \cong \\
  \bigsqcup \limits_{x \in {\mathcal C}_H} \bigsqcup \limits_{0 \le h_i \le l_i} &
\mathfrak{M}_{L_0,L_1}^0( p_{h_1+1}^1, \ldots, p_{l_1}^1,x, p_1^0,
\ldots, p_{h_0}^0, x_+;
H_s^+,J_\bullet) \times \nonumber \\
& \mathfrak{M}_{L_0,L_1}^0(p_1^1, \ldots, p_{h_1}^1, x_-, p_{1+h_0}^0,
\ldots, p_{l_0}^0, x; H_s^-,J_\bullet)
\nonumber
\end{align}
which follows from standard compactness and gluing techniques, once we know that, for any
$R' >0$, there is $R_0$ such that, for all $R \ge R_0$, if
$$(\boldsymbol{\zeta}^0, \boldsymbol{\zeta}^1, u) \in \mathfrak{M}_{L_0,L_1}^0
(\mathbf{p}^1, x_- \mathbf{p}^0, x_+; H_s^R,J_\bullet),$$
then $\zeta_j^i \not \in [-R', R']$ for $i=0,1$ and $j= 1, \ldots, l_i$.

This follows from a simple compactness argument: if there is $R'$ and a sequence $R_n$
with $$(\boldsymbol{\zeta}^0_n, \boldsymbol{\zeta}^1_n, u_n) \in \mathfrak{M}^0_{L_0,
L_1}(\mathbf{p}^1, x_-, \mathbf{p}^0, x_+; H_s^{R_n}, J_\bullet)$$ and for every $n$
there is some $\zeta_j^i \in [-R', R']$, then the limit for $n \to \infty$ has one level which is
a solution of a Floer equation with $s$-invariant data and at least one boundary puncture.  For
index reasons  this level must have index zero\color{black}, but it cannot be constant
because of the boundary puncture. This is a contradiction.
\end{proof}

With lemma \ref{properties of Phi} at hand, we can prove the following invariance
property in the usual formal way.
\begin{cor}\label{doi coge}
If $H_0$ and $H_1$ are cylindrical Hamiltonian functions which are compatible with $L_0$
and $L_1$ and such that $h_0'(\mathfrak{r}(w))= h_1'(\mathfrak{r}(w))$ for $w$ outside of
a compact set, then the continuation map
$$\Phi_{H_0, H_1}^* \colon \mathrm{HF}((L_0, \varepsilon_0), (L_1, \varepsilon_1); H_0) \to
\mathrm{HF}((L_0, \varepsilon_0), (L_1, \varepsilon_1); H_1)$$
is an isomorphism.
\end{cor}

\subsection{Compactly supported safe isotopies}\label{sss:
compactly supported isotopies}

Let $\psi_t \colon W \to W$ be a compactly supported smooth isotopy such that $\iota_t= \psi_t \circ \iota_1 \colon L_1 \to W$ is a safe isotopy. By Lemma \ref{lem:GeneratingHamiltonian} there exists a local Hamiltonian $G_t$ defined on $L_1$ which generates the $\iota_t$ and for which $dG_t$ has compact support. (Recall that $G_t$ may not extend to a single-valued function on $W$.)

In the following we will make the further assumption that the path
$$(\psi_t)_*J_1 = d\psi_t \circ J_1 \circ d \psi_t^{-1}, \:\: t \in [0,1],$$
consists of compatible almost complex structures. This will cause no restriction, since we only need the case when $\psi_t$ is equal to the Liouville flow, which is conformally symplectic.

\begin{rem}
In the following manner more general safe isotopies can be considered. Since it is possible to present any smooth isotopy as a concatenation of $C^2$-small isotopies, it then suffices to carry out the constructions here for each step separately. Namely, since tameness is an open condition, sufficiently $C^2$-small isotopies may be assumed to preserve any given tame almost complex structure. Further control near the double points can then be obtained by assuming that $\psi_t$ actually is conformally symplectic there, which can be assumed without loss of generality.
\end{rem}

Denote by $L_1'=\psi_1\circ\iota_1(L_1)$ the image. By the usual abuse of
notation, we will write $L_1'$ or $\psi_1(L_1)$ instead of $(L_1, \psi_1 \circ \iota)$. From now
on we will assume that the Hamiltonian $H$ is compatible both with $L_0$ and $L_1$ and
with $L_0$ and $L_1'$.

The obstruction algebras $\mathfrak{D}_1$ of $(L_1, J_1)$ and $\mathfrak{D}_1'$ of
$(L_1', (\psi_1)_*J_1)$ are tautologically isomorphic because $\psi_1$ matches the generators and the holomorphic curves contributing to the differentials, and therefore any augmentation $\varepsilon_1$ of $\mathfrak{D}_1$ corresponds to an augmentation $\varepsilon_1'$ of $\mathfrak{D}_1'$.

We fix time-dependent almost complex structures $J_\bullet^+$ and $J_\bullet^-$ such that
\begin{itemize}
\item $J_t^\pm = J_0$ for $t \in [0, 1/4]$,
\item $J_t^+=J_1$ and $J_t^-= (\psi_1)_*J_1$ for $t \in [3/4, 1]$,
\item $J_\bullet^+$ is $(L_0, L_1)$-regular and $J_\bullet^-$ is $(L_0, L_1')$-regular.
\end{itemize}
Given augmentations $\varepsilon_0$ for $L_0$ and $\varepsilon_1$ for $L_1$, we will define
a chain map
$$\Psi_G \colon \mathrm{CF}((L_0, \varepsilon_0), (L_1, \varepsilon_1); H, J_\bullet^+) \to
 \mathrm{CF}((L_0, \varepsilon_0), (L_1', \varepsilon_1'); H, J_\bullet^-)$$
using a Floer equation with moving boundary conditions. The presence of self-intersection points of $L_1$ makes the construction of the moduli spaces
more subtle than in the usual case because, in order to have strip-like ends, we need to make
the moving boundary conditions constant near the boundary punctures, and therefore domain
dependent.

Recall the sets \begin{align*}&\op{Conf}^n(\R) = \{ (\zeta_1, \ldots, \zeta_n) \in \R^n | \: \zeta_1 < \ldots <
\zeta_n \}\quad \mbox{and}\\ &\overline{\op{Conf}}^n(\R) = \{ (\zeta_1, \ldots, \zeta_n) \in \R^n | \: \zeta_1
\le \ldots \le \zeta_n\}.\end{align*} (Note that this is not how configuration spaces are usually
compactified.)
Given $n \in \N$, we denote $\mathbf{n} = \{1, \ldots, n\}$ and for $m < n$
we denote $\op{hom}(\mathbf{n}, \mathbf{m})$ the set of nondecreasing and
surjective function $\phi \colon \mathbf{n} \to \mathbf{m}$. Every $\phi \in
\op{hom}(\mathbf{n}, \mathbf{m})$ induces an embedding $\phi^* \colon
\overline{\op{Conf}}^m(\R) \to \overline{\op{Conf}}^n(\R)$ defined by
$$\phi^*(\zeta_1,
\ldots, \zeta_m) = (\zeta_{\phi(1)}, \ldots, \zeta_{\phi(m)}).$$
The boundary of
$\overline{\op{Conf}}^n(\R)$ is a stratified space with dimension $m$ stratum
$$\bigsqcup \limits_{\phi \in \op{hom}(\mathbf{n}, \mathbf{m})}\phi^*(\op{Conf}^m(\R)).$$

The embeddings $\phi$ defined above extend to diffeomorphisms
$$\overline{\phi}^* \colon \overline{\op{Conf}}^m(\R) \times \R_+^{n-m} \to
\overline{\op{Conf}}^n(\R)$$
such that
$$(\zeta_1', \ldots, \zeta_n')= \overline{\phi}^*((\zeta_1, \ldots, \zeta_m), (\epsilon_1, \ldots,
\epsilon_{n-m}))$$
if $$\zeta_i' = \zeta_{\phi(i)} + \sum \limits_{k=0}^{i - \phi(i)} \epsilon_k,$$
where $\epsilon_0=0$ for the sake of the formula.

\begin{lemma}\label{boundary dependent reparametrisations}
Fix $\delta >0$. There is a family of constants $\kappa_n >0$ and smooth functions
$$\nu_n \colon \overline{\op{Conf}}^n(\R) \times \R \to [0,1]$$
such that, denoting by $s$ the coordinate in the second factor,
\begin{enumerate}
\item $\nu_n(\zeta_1, \ldots, \zeta_n, s) =0$ for $s > \kappa_n$,
\item $\nu_n(\zeta_1, \ldots, \zeta_n, s) =1$ for $s < - \kappa_n$,
\item $\partial_s \nu_n(\zeta_1, \ldots, \zeta_n, s) \in [-2, 0]$ for all $(\zeta_1, \ldots, \zeta_n, s)
\in \overline{\op{Conf}}^n(\R) \times \R$,
\item $\partial_s \nu_n(\zeta_1, \ldots, \zeta_n, s)=0$ if $|s- \zeta_i| \le \frac{\delta}{2}$ for
some $i=1, \ldots, n$, and
\item $\nu_n \circ \phi^*= \nu_m$ for all $m <n$ and all $\phi \in \op{hom}(\mathbf{n},
\mathbf{m})$.
\end{enumerate}
\end{lemma}
\begin{proof}
We can construct the sequences $\kappa_n$ and $\nu_n$ inductively over $n$ using the fact
that the set of functions satisfying (1)--(4) is convex.
\end{proof}

Given $\boldsymbol{\zeta} = (\zeta_1, \ldots, \zeta_n) \in \op{Conf}^n(\R)$, we define
$\nu_{\boldsymbol{\zeta}} \colon \R \to [0,1]$ by $\nu_{\boldsymbol{\zeta}}(s)=
\nu_n(\zeta_1, \ldots, \zeta_n, s)$.

\begin{lemma}\label{berlin}
For every $n$, there is a contractible set of smooth maps
$$\tilde{J}_n \colon \overline{\op{Conf}}^n(\R) \times Z \to {\mathcal J}(\theta)$$
such that
\begin{enumerate}
\item $\tilde{J}_n(\boldsymbol{\zeta}, s,t) = J_t^+$ if $s > \kappa_n+1$,
\item $\tilde{J}_n(\boldsymbol{\zeta}, s,t) = J_t^-$ if $s < - \kappa_n-1$,
\item $\tilde{J}_n(\boldsymbol{\zeta}, s,t)=J_0$ if $t \in [0, 1/4]$,
\item $\tilde{J}_n(\boldsymbol{\zeta}, s,t)= d\psi_{\nu_{\boldsymbol{\zeta}}(s)} \circ J_1 \circ
d\psi_{\nu_{\boldsymbol{\zeta}}(s)}^{-1}$ if $t \in [3/4, 1]$, and
\item for all $\phi \in \op{hom}(\mathbf{n}, \mathbf{n-1})$,
$\tilde{J}_n(\phi^*(\boldsymbol{\zeta}), s, t)= \tilde{J}_{n-1}(\boldsymbol{\zeta}, s,t)$.
\end{enumerate}
\end{lemma}
\begin{proof}
We build $\tilde{J}_n$ inductively on $n$. At each step, the map $\tilde{J}$ is determined in
the complement of  $\op{Conf}^n(\R) \times [- \kappa_n-1, \kappa_n+1] \times [1/4, 3/4]$.
We can extend it to $\overline{\op{Conf}}^n(\R) \times Z$ because ${\mathcal J}(\theta)$
is contractible.
\end{proof}
Given $\boldsymbol{\zeta} \in \op{Conf}^n(\R)$, we will denote
$\tilde{J}_{\boldsymbol{\zeta}}$ the $s-$ and $t$-dependent almost
complex structure obtained by restricting $\tilde{J}$ to $\{ \boldsymbol{\zeta} \}
\times Z$. Given $(\boldsymbol{\zeta}^0, \boldsymbol{\zeta}^1) \in
\widetilde{\mathcal R}^{l_0|l_1}$, we will not distinguish between $\zeta^1_j \in
\R \times \{ 1 \}$ and its $s$-coordinate, and by this abuse of notation, to
$(\boldsymbol{\zeta}^0, \boldsymbol{\zeta}^1) \in \widetilde{\mathcal R}^{l_0|l_1}$
we will associate $\nu_{\boldsymbol{\zeta}^1}$ and $\tilde{J}_{\boldsymbol{\zeta}^1}$.
For simplicity, the $s-$ and $t$-dependence of  $\tilde{J}_{\boldsymbol{\zeta}^1}$ will be
omitted in writing the Floer equation.

Consider the sets
\begin{align*}
{\mathcal C}_H & = \{ x \colon [0,1] \to W : x(0) \in L_0, x(1) \in L_1 \}, \\
{\mathcal C}_H' & = \{ x \colon [0,1] \to W : x(0) \in L_0, x(1) \in L_1' \}.
\end{align*}
\begin{dfn}\label{non ne posso piuu}
Given $x_+ \in {\mathcal C}_H$, $x_- \in {\mathcal C}_H'$, and
$\mathbf{p}^i \in D_i^{l_i}$ for $i=0,1$ and $l_i \ge 0$ we define the moduli space
$$\mathfrak{M}_{L_0, L_1}(\mathbf{p}^1, x_-, \mathbf{p}^0, x_+; H, G, \tilde{J}, \nu)$$
as the set of triples $(\boldsymbol{\zeta}^0, \boldsymbol{\zeta}^1, u)$ such that:
\begin{itemize}
\item $(\boldsymbol{\zeta}^0, \boldsymbol{\zeta}^1) \in {\mathcal R}^{l_0|l_1}$ and
$u \colon Z_{\boldsymbol{\zeta}^0, \boldsymbol{\zeta}^1} \to W$ satisfies the Floer equation
\begin{equation}
\frac{\partial u}{\partial s} + \tilde{J}_{\boldsymbol{\zeta}^1} \left (\frac{\partial u}{\partial t}
- \chi'(t) X_H(\chi(t), u) \right )=0,
\end{equation}
\item $\lim \limits_{s \to \pm \infty} u(s,t)= x_\pm(\chi(t))$,
\item $u(s, 0) \in L_0$ for all $(s,0) \in Z_{\boldsymbol{\zeta}^0, \boldsymbol{\zeta}^1}$,
\item $u(s,1) \in \psi_{\nu_{\boldsymbol{\zeta}^1}(s)}(L_1)$ for all $(s,1) \in
Z_{\boldsymbol{\zeta}^0, \boldsymbol{\zeta}^1}$, and
\item each $\zeta^0_j$ is a negative puncture at $p_j^0$ and each
$\zeta^1_j$ is a negative puncture at $\psi_{\nu_{\boldsymbol{\zeta}^1}(\zeta^1_j)}(L_1)$.
\end{itemize}
(Recall that $\psi_t \colon W \to W$ here is a smooth isotopy satisfying the assumptions made in the
beginning of this section, whose restriction to $L_1$ in particular is the compactly supported
safe isotopy generated by $G \colon \R \times L \to \R.$) We denote by
$\mathfrak{M}_{L_0, L_1}^k(\mathbf{p}^1, x_-, \mathbf{p}^0, x_+; H, G,\tilde{J}, \nu)$ the
set of triples $(\boldsymbol{\zeta}^0, \boldsymbol{\zeta}^1, u)$ where $\op{ind}(u)=k$.
\end{dfn}

\begin{prop} \label{moduli spaces with moving boundary}
For a generic $\tilde{J}$ as in Lemma~\ref{berlin}, the moduli space
$$\mathfrak{M}_{L_0, L_1}^k(\mathbf{p}^1, x_-, \mathbf{p}^0, x_+; H, G,\tilde{J}, \nu)$$ is a
transversely cut out manifold of dimension $k$. If $k=0$, it is compact, and therefore
consists of a finite set of points.  If $k=1$, it admits a compactification in the Gromov-Floer sense. \color{black}

If $L_0$ and $L_1$ are spin, a choice of a spin structure on each induces a coherent orientation {of the moduli space (see \cite{Ekholm_&_Orientation_Homology})}.
\end{prop}

\begin{dfn}
We define the weighted count
$$\mathfrak{m}(\mathbf{p}^1, x_-, \mathbf{p}^0, x_+; H, G) =
\# \mathfrak{M}^0_{L_0, L_1}(\mathbf{p}^1, x_-, \mathbf{p}^0, x_+; H, G, \tilde{J}, \nu)
\varepsilon_0(\mathbf{p}^0) \varepsilon_1(\mathbf{p}^1)$$
and then we define $\Psi_G$ as
\begin{equation*}
\Psi_G(x_+)= \sum_{x_- \in {\mathcal C}_H} \sum_{l_0, l_1 \in \N}
\sum_{\mathbf{p}^i \in D_i^{l_i}}  \mathfrak{m}(\mathbf{p}^1, x_-, \mathbf{p}^0, x_+;
H, G)x_-.
\end{equation*}
\color{black}
\begin{rem}
The word $\mathbf{p}^1$ consists of double points living on the different Lagrangian immersions $\psi_{\nu_{\boldsymbol{\zeta}^1}(\zeta^1_j)}(L_1).$ However, when using the pushed forward almost complex structures $(\psi_{\nu_{\boldsymbol{\zeta}^1}(\zeta^1_j)})_*J_1,$ their obstruction algebras all become canonically identified with $(\mathfrak{A}(L_1),\mathfrak{d})$ defined using $J_1$. This motivates our abuse of notation $\varepsilon_1$ for an augmentation induced by these canonical identifications.
\end{rem}
\end{dfn}

{A consideration of Proposition~\ref{moduli spaces with moving boundary} together with a count of solutions of index $-1$ which arise in a one-parametric family of moduli spaces implies the following:}

\begin{lemma}
The map $\Psi_G$ is a chain map. Moreover, up to chain homotopy, it does not depend on the
choice of $\nu$ and on the homotopy class of $\psi_t$ relative to the endpoints.
\end{lemma}
\begin{lemma}\label{properties of Psi}
Let $G^0,G^1 \colon L_1 \to \R$ be local Hamiltonian functions generating the safe isotopies
$\psi_t^0 \circ \iota_1$ and $\psi_t^1 \circ \psi_1^0 \circ \iota_1$ respectively, and let $G^2 \colon L_1 \to \R$ be a local Hamiltonian function generating
$$\psi_t^2 = \begin{cases}
\psi^1_{2t} & \text{for } t \in [0, 1/2], \\
\psi^2_{2t-1} \circ \psi^1_1 & \text{for } t \in [1/2, 1].
\end{cases}$$
Then $\Psi_{G^2}$ is chain homotopic to $\Psi_{G^0} \circ \Psi_{G^1}$.
\end{lemma}
The proof of Lemma~\ref{properties of Psi} is analogous to the proof of
Lemma~\ref{properties of Phi}. 
\begin{cor}
If $G \colon \R \times L \to \R$ satisfies $dG_t=0$ outside a compact subset of $(0,1) \times L$, then the map
$\Psi_G$ induces an isomorphism in homology.
\end{cor}
If $\bar J_\bullet^+$ is one $(L_0, L_1)$-regular almost complex structure and $\bar
J_\bullet^-$ is another $(L_0, \psi_1(L_1))$-regular almost complex structure, instead of
repeating the construction of $\tilde{J}$ with  $\bar J_\bullet^\pm$ as starting point, we
prefer to consider $\tilde{J}$ assigned once and for all to the triple $(L_0, L_1, G)$ and
define the continuation map
$$\mathrm{CF}((L_0, \varepsilon_0), (L_1, \varepsilon_1); H, \bar J_\bullet^+) \to
\mathrm{CF}((L_0, \varepsilon_0), (L_1', \varepsilon_1); H, \bar J_\bullet^-)$$
as the composition $\Upsilon_- \circ \Psi_G \circ \Upsilon_+$, where
\begin{align*}
\Upsilon_+ \colon & \mathrm{CF}((L_0, \varepsilon_0), (L_1, \varepsilon_1); H, \bar J_\bullet^+) \to
\mathrm{CF}((L_0, \varepsilon_0), (L_1, \varepsilon_1); H, J_\bullet^+), \\
\Upsilon_- \colon & \mathrm{CF}((L_0, \varepsilon_0), (L_1', \varepsilon_1'); H, J_\bullet^-) \to
\mathrm{CF}((L_0, \varepsilon_0), (L_1', \varepsilon_1'); H, \bar J_\bullet^-)
\end{align*}
are the maps defined in Section \ref{sss: changing J}.

\subsection{{Continuation maps for almost complex structures and Hamiltonian functions commute}}
Let $H_s$ be a homotopy from $H_-$ to $H_+$ as in Section \ref{sss: changing hamiltonian} and
$J_\bullet^\bullet$ a homotopy from $J_\bullet^{-1}$ to $J_\bullet^{+1}$ as in Section \ref{sss: changing
J}. For simplicity we will denote
$$\mathrm{CF}(H_\pm, J_\bullet^\pm)= \mathrm{CF}((L_0, \varepsilon_0^\pm), (L_1, \varepsilon_1^\pm);
H_\pm, J_\bullet^{\pm 1}).$$
If $\varepsilon_i^+ = \varepsilon_i^- \circ \mathfrak{Y}_i$, we have defined there
continuation maps
\begin{align*}
\Upsilon_{\pm} \colon & \mathrm{CF}(H_\pm, J_\bullet^+) \to \mathrm{CF}(H_\pm, J_\bullet^-),\\
\Phi_{\pm} \colon & \mathrm{CF}(H_+, J_\bullet^\pm) \to \mathrm{CF}(H_-, J_\bullet^\pm)
\end{align*}
and now we will to prove that they are compatible in the following sense.
\begin{prop}\label{prop: commutation between continuation maps}
The diagram
\begin{equation} \label{commutation between continuation maps}
\xymatrix{
\mathrm{CF}(H_+, J_\bullet^+) \ar[r]^{\Phi_+} \ar[d]_{\Upsilon_+} & \mathrm{CF}(H_-, J_\bullet^+)
\ar[d]^{\Upsilon_-} \\
\mathrm{CF}(H_+, J^-) \ar[r]_{\Phi_-} & \mathrm{CF}(H_-, J_\bullet^-)
}
\end{equation}
commutes up to homotopy.
\end{prop}

Proposition \ref{prop: commutation between continuation maps} will be proved by applying
the bifurcation method to the definition of the continuation maps $\Phi_{\pm}$: i.e. we will
study the parametrised moduli spaces
$\mathfrak{M}^0_{L_0,L_1}(\mathbf{p}^1, x_-, \mathbf{p}^0, x_+; H_s, J_ \bullet^\bullet)$
consisting of pairs $(\delta, u)$ such that $\delta \in [0,1]$ and $u \in
\mathfrak{M}^0_{L_0,L_1}(\mathbf{p}^1, x_-, \mathbf{p}^0, x_+; H_s, J_ \bullet^\delta)$.
For a generic homotopy $J_ \bullet^\bullet$, these parametrised moduli spaces are transversely
cut out manifolds of dimension one.
As before, there is a finite set $\Delta$ of bifurcation points such that, for all $\delta \in
\Delta$, there is a unique nonempty moduli space of one of the following types:
\begin{enumerate}
\item[(i)] $\mathfrak{N}_{L_0}^{-1}(q_0^0; q_1^0, \ldots, q_d^0; J_0^\delta)$ or
$\mathfrak{N}_{L_1}^{-1}(q_0^1; q_1^1, \ldots, q_d^1; J_1^\delta)$,
\item[(ii)] $\mathfrak{M}_{L_0,L_1}^0(\mathbf{q}^1, y_-, \mathbf{q}^0, y_+; H_-,
J_\bullet^\delta)$ or $\mathfrak{M}_{L_0,L_1}^0(\mathbf{q}^1, y_-,\mathbf{q}^0, y_+; H_+,
J_\bullet^\delta)$,
\item[(iii)] $\mathfrak{M}_{L_0,L_1}^{-1}(\mathbf{q}^1, y_-, \mathbf{q}^0, y_+; H_s,
J_\bullet^\delta)$.
\end{enumerate}
To these moduli spaces correspond four types of boundary configuration for
the compactification of the one-dimensional parametrised moduli spaces
$\mathfrak{M}_{L_0,L_1}^0(\mathbf{p}^1, x_-, \mathbf{p}^0, x_+; H_s, J_\bullet^\bullet)$,
which we write schematically as:
\begin{itemize}
\item[(i)]  $\mathfrak{N}_{L_0}^{-1}(J_0^\delta) \times \mathfrak{M}^0(H_s,
J_\bullet^\delta)$ or $\mathfrak{N}_{L_1}^{-1}(J_1^\delta) \times \mathfrak{M}^0
(H_s, J_\bullet^\delta)$,
\item[(ii)] $\mathfrak{M}^0(H_-, J_\bullet^\delta) \times
\mathfrak{M}^0(H_s, J_\bullet^\delta)$ or $\mathfrak{M}^0(H_s,
J_\bullet^\delta) \times \mathfrak{M}^0(H_+, J_\bullet^\delta)$,
\item[(iii)] $\mathfrak{M}^1(H_-, J_\bullet^\delta) \times
\mathfrak{M}^{-1}(H_s, J_\bullet^\delta)$ or $\mathfrak{M}^{-1}(H_s,
J_\bullet^\delta) \times \mathfrak{M}^1(H_+, J_\bullet^\delta)$,
\item[(iii)'] $\mathfrak{N}^{ 0}_{L_0}(J_0^\delta) \times \mathfrak{M}^{-1}(H_s, J_\bullet^\delta)$ or $\mathfrak{N}^{ 0}_{L_1}(J_1^\delta) \times
\mathfrak{M}^{-1}(H_s, J_\bullet^\delta)$.
\end{itemize}
There is also one fifth type of boundary configuration:
\begin{itemize}
\item[(iv)] $\mathfrak{M}^0(H_s, J_\bullet^{-1})$ or
$\mathfrak{M}^0(H_s, J_\bullet^1)$.
\end{itemize}
In order to prove Proposition \ref{prop: commutation between continuation maps} we split the homotopy $J_\bullet^\bullet$ into pieces, each of which contains only
one element of $\Delta$, and we prove that for each piece the corresponding
diagram \eqref{commutation between continuation maps} commutes up to
homotopy. Putting all pieces together, we will obtain the result. We rescale each
piece of homotopy so that it is parametrised by $[-1,1]$ and the bifurcation point is $0$. \color{black}

\begin{lemma}
Let $\Delta = \{ 0 \}$ be of type (i). Then Diagram \eqref{commutation between
continuation maps} commutes.
\end{lemma}
\begin{proof}
We have proved in Lemma \ref{variation of J cases i and ii} that
$\Upsilon_{\pm}$ are the identity maps. Here we will prove that, under the hypothesis of the lemma, $\Phi_+=\Phi_-$. We use the same construction, and the same notation, as in the proof of Lemma \ref{variation of J cases i and ii}.

The homotopy of Hamiltonians $H_s$ on $W$ induces a homotopy of Hamiltonians $\widetilde{H}_s$ on $T^*\R \times W$. Let
$$\widetilde{\Phi} \colon \mathrm{CF}((\widetilde{L}_0, \widetilde{\epsilon}_0), (\widetilde{L}_1, \widetilde{\varepsilon}_1); \widetilde{H}_+, \widetilde{J}_\bullet) \to \mathrm{CF}((\widetilde{L}_0, \widetilde{\epsilon}_0), (\widetilde{L}_1, \widetilde{\varepsilon}_1); \widetilde{H}_-, \widetilde{J}_\bullet)$$
be the continuation map associated to the homotopy $\widetilde{H}_s$.

Solutions of the Floer equation perturbed by $\widetilde{H}_s$ are still holomorphic when projected to $T^*\R$, and therefore the counterpart of Lemmas 4.14 and 4.19 of \cite{Ekholm_&_Orientation_Homology} give that, for $\kappa=-1,0,+1$,  $\widetilde{\Phi}(x[\kappa])$ is a linear combination of chords $y[\kappa]$.  The statement for $\kappa=0$ is a consequence of the assumption that there is no Floer strip for $\widetilde{H}_s$ of index  $-1$. (Unlike in the proof of Lemma \ref{variation of J cases i and ii} we do not have the two strips between from
$x[0]$ to $x[\pm 1]$ because in the definition of $\widetilde{\Phi}$ we consider solutions of index zero.) Then $\widetilde{\partial} \circ \widetilde{\Phi}= \widetilde{\Phi} \circ \widetilde{\partial}$ and Equation \ref{partialtilde} give $\Phi_+= \Phi_-$.
\color{black}
\end{proof}

\begin{lemma}
Let $\Delta = \{ 0 \}$ be of type (ii). Then Diagram \eqref{commutation between
continuation maps} commutes.
\end{lemma}
\begin{proof}
We assume, without loss of generality, that the moduli space of negative formal dimension is
$$\mathfrak{M}^0_{L_1, L_1}(\mathbf{q}^1, y_-, \mathbf{q}^0, y_+; H_+, J^0_\bullet).$$ By
Lemma~\ref{variation of J case iii} the continuation maps for the change of almost complex structure are $\Upsilon_- =
Id$ and
$$\Upsilon_+ (x) = \begin{cases}
x & \text{if } x \ne y_+, \\
y_+  + \# \mathfrak{M}_{L_0, L_1}^0(\mathbf{q}^1, y_-,
\mathbf{q}^0, y_+; H_+, J_\bullet^\delta) \varepsilon_0(\mathbf{q}^0)
\varepsilon_1(\mathbf{q}^1)y_- &  \text{if } x = y_+.
\end{cases}$$

The structure of the compactification of one-dimensional parametrised moduli spaces implies
that
\begin{align*}
&\# \mathfrak{M}^0(\mathbf{q}^1\mathbf{p}^1, x_-, \mathbf{p}^0
\mathbf{q}^0, y_+; H_s, J^1_\bullet) -  \# \mathfrak{M}^0(\mathbf{q}^1
\mathbf{p}^1, x_-, \mathbf{p}^0 \mathbf{q}^0, y_+; H_s, J^0_\bullet)  \\
& = \# \mathfrak{M}^0(\mathbf{q}^1, y_-,  \mathbf{q}^0, y_+; H_+, J^\delta_\bullet)
\# \mathfrak{M}^0(\mathbf{p}^1, x_-,  \mathbf{p}^0, y_-; H_s, J^0_\bullet),
\end{align*}
while the cardinality of all other moduli spaces remains unchanged. This implies that
Diagram \eqref{commutation between continuation maps} commutes.
\end{proof}
 We have dropped the Lagrangian labels from the notation in order to keep
the formulas compact. We will do the same in the proofs of the following lemma.
\begin{lemma}
Let $\Delta = \{ 0 \}$ be of type (iii). Then Diagram \eqref{commutation between
continuation maps} commutes up to homotopy.
\end{lemma}
\begin{proof}
Let $\mathfrak{M}^{-1}_{L_0, L_1}(\mathbf{q}^1, y_-, \mathbf{q}^+, y_+; H_s,
J_\bullet^0)$ be the nonempty moduli space of negative
formal dimension. In this case $\Upsilon_\pm = Id$ and we define a linear map
$K \colon \mathrm{CF}(H_+, J_\bullet^{+1}) \to
\mathrm{CF}(H_-, J_\bullet^{-1})$ by
$$K(x) = \begin{cases}
0 & \text{if } x \ne y_+, \\
\# \mathfrak{M}^{-1}_{L_0, L_1}(\mathbf{q}^1, y_-, \mathbf{q}^+, y_+; H_s,
J_\bullet^\delta) \varepsilon_0(\mathbf{q}^0) \varepsilon_1(\mathbf{q}^1) y_- &
\text{if } x=y_+.
\end{cases}$$
The structure of the boundary of the compactification of the one-dimensional parametrised
moduli spaces implies that

 \begin{align*}
&\# \mathfrak{M}^0(\mathbf{q}^1\mathbf{p}^1, x_-, \mathbf{p}^0
\mathbf{q}^0, y_+; H_s, J^1_\bullet) -  \# \mathfrak{M}^0(\mathbf{q}^1
\mathbf{p}^1, x_-, \mathbf{p}^0 \mathbf{q}^0, y_+; H_s, J^0_\bullet)  \\
& = \# \mathfrak{M}^{-1}(\mathbf{q}^1, y_-,  \mathbf{q}^0, y_+; H_s, J_\bullet^\delta)
\# \mathfrak{M}^1(\mathbf{p}^1, x_-,  \mathbf{p}^0, y_-; H_-, J^0_\bullet),\ \mbox{and}\\
 &\# \mathfrak{M}^0(\mathbf{q}^1\mathbf{p}^1, y_-, \mathbf{p}^0
\mathbf{q}^0, x_+; H_s, J^1_\bullet) -  \# \mathfrak{M}^0(\mathbf{q}^1
\mathbf{p}^1, y_-, \mathbf{p}^0 \mathbf{q}^0, x_+; H_s, J^0_\bullet)  \\
& = \# \mathfrak{M}^{-1}(\mathbf{q}^1, y_-,  \mathbf{q}^0, y_+; H_s, J_\bullet^\delta)
\# \mathfrak{M}^1(\mathbf{p}^1, x_-,  \mathbf{p}^0, y_-; H_-, J^0_\bullet).
 \end{align*}
From this it follows that $\Phi_+ - \Phi_- = \partial K + K \partial$.
\end{proof}
The degenerations of type (iii)' are cancelled algebraically by the augmentations, and therefore
we obtain the commutativity of the diagram~\eqref{commutation between continuation
maps} for a generic homotopy $J_\bullet^\bullet$.

Now we compare the continuation maps $\Phi$ for the change of Hamiltonian and the continuation maps $\Psi$ for compactly supported safe isotopies of $L_1$.
Let $G \colon \R \times L_1 \to \R$ be a local  Hamiltonian function satisfying $dG_t=0$ outside a compact subset of $(0,1) \times L_1$ which generates the safe isotopy $\psi_t \circ \iota_1,$ and denote $L_1' = \psi_1(L_1)$. If $H_+$ and $H_-$ are two Hamiltonian functions which are compatible both with $L_0$ and
$L_1$ and with $L_0$ and $L_1'$, then we have continuation maps
$$\Psi_G^\pm \colon \mathrm{CF}((L_0, \varepsilon_0), (L_1, \varepsilon_1); H_\pm, J_\bullet^+) \to
\mathrm{CF}((L_0, \varepsilon_0), (L_1', \varepsilon_1'); H_\pm, J_\bullet^-)$$
and continuation maps
\begin{align*}
\Phi \colon & \mathrm{CF}((L_0, \varepsilon_0), (L_1, \varepsilon_1); H_+, J_\bullet^+) \to
\mathrm{CF}((L_0, \varepsilon_0), (L_1, \varepsilon_1); H_-, J_\bullet^+), \\
\Phi'  \colon & \mathrm{CF}((L_0, \varepsilon_0), (L_1', \varepsilon_1'); H_+, J_\bullet^-) \to
\mathrm{CF}((L_0, \varepsilon_0), (L_1', \varepsilon_1'); H_-, J_\bullet^-)
\end{align*}
induced by a homotopy of Hamiltonians $H_s$ with $H_s=H_+$ for $s \ge 1$ and $H_s=
H_-$ for $s \le -1$.
\begin{lemma}\label{lemma: commutation between Phi and Psi}
The diagram
$$\xymatrix{
\mathrm{CF}((L_0, \varepsilon_0), (L_1, \varepsilon_1); H_+, J_\bullet^+) \ar[r]^\Phi \ar[d]_{\Psi_G^+}
&  \mathrm{CF}((L_0, \varepsilon_0), (L_1, \varepsilon_1); H_-, J_\bullet^+) \ar[d]^{\Psi_G^+} \color{black} \\
\mathrm{CF}((L_0, \varepsilon_0), (L_1', \varepsilon_1'); H_+, J_\bullet^-) \ar[r]^{\Phi'} &
\mathrm{CF}((L_0, \varepsilon_0), (L_1', \varepsilon_1'); H_-, J_\bullet^-)
}$$
commutes up to homotopy.
\end{lemma}
\begin{proof}[Sketch of proof]
For $R \in \R$ we define $H_s^R= H_{s-R}$. We define the moduli spaces\\
$\mathfrak{M}_{L_0, L_1}^k(\mathbf{p}^1, x_-, \mathbf{p}^0, x_+; H_s^R, G, \tilde{J},
\nu)$
as in Definition~\ref{non ne posso piuu} by replacing the Hamiltonian $H$ by the
$s$-dependent Hamiltonian $H_s^R$. Counting pairs $(R, u)$ where $R \in \R$ and
$$u \in \mathfrak{M}_{L_0, L_1}^{-1} (\mathbf{p}^1, x_-, \mathbf{p}^0, x_+; H_s^R, G,
\tilde{J}, \nu),$$ weighted by the augmentations, we obtain a homotopy between $\Psi_G^-
\circ \Phi$ and $\Phi' \circ \Psi_G^+$.
\end{proof}

\section{Wrapped Floer cohomology for exact Lagrangian immersions}
\label{sec: immersed wrapped}
In this section we define wrapped Floer cohomology  for unobstructed immersed exact
Lagrangian submanifolds. With the preparation of the previous sections in place, the definition is not
different from the usual one for Lagrangian submanifolds.
\subsection{Wrapped Floer cohomology as direct limit}
We start by defining wrapped Floer cohomology as a direct limit. This point of view will be
useful in the vanishing theorem of the following section.  A sketch of the chain level
construction, following \cite{OpenString}, will be given in the next subsection.

Let $(L_0, \iota_0)$ and $(L_1, \iota_1)$ be immersed exact Lagrangian submanifolds with augmentations
$\varepsilon_0$ and $\varepsilon_1$ respectively. We assume that all intersection points
between $L_0$ and $L_1$ are transverse, $L_0$ and $L_1$ are cylindrical over Legendrian
submanifolds $\Lambda_0$ and $\Lambda_1$ respectively, and all Reeb chords between
$\Lambda_0$ and $\Lambda_1$ are nondegenerate.

For every $\lambda \in \R$ we denote by $h_\lambda \colon  \R^+ \to \R$ the function
\begin{equation}\label{H_lambda}
h_\lambda(\rho) = \begin{cases}
 0 & \text{if} \quad \rho \in (0,1], \\
\lambda \rho - \lambda & \text{if} \quad \rho \ge 1.
\end{cases}
\end{equation}
We smooth $h_\lambda$ inside the interval $[4/5, 6/5]$ (or any sufficiently small
neighbourhood of $1$ independent of $\lambda$) and, by abuse of
notation, we still denote the resulting function by $h_\lambda$. We assume also that the
resulting smooth function satisfies $h''(\rho) \ge 0$ for all $\rho \in \R^+$.
We define time-independent cylindrical Hamiltonians $H_\lambda \colon W \to \R$ by
\begin{equation}\label{wrapping hamiltonian}
H_\lambda(w)= h_\lambda(e^ {\mathfrak{r}(w)}).
\end{equation}
Hamiltonian functions of this form will be called
{\em wrapping Hamiltonian functions}.

We fix a sequence of positive real number $\lambda_n$ such that $\lim \limits_{n \to +\infty} \lambda_n=
+ \infty$ and, for any $n$, $\lambda_n$ is not the length of a Reeb
chord from $\Lambda_0$ to $\Lambda_1$. The set of $(L_0, L_1)$-regular almost complex
structures for every $H_{\lambda_n}$ is dense, and we pick an element $J_\bullet$.

By Subsection~\ref{sss: changing hamiltonian}, for every  $m \ge n$ there are continuation
maps
$$\Phi_{\lambda_n, \lambda_m} \colon \mathrm{HF}((L_0, \varepsilon_0), (L_1, \varepsilon_1);
H_{\lambda_n}, J_\bullet) \to \mathrm{HF}((L_0, \varepsilon_0), (L_1, \varepsilon_1); H_{\lambda_m},
J_\bullet)$$
forming a direct system.
\begin{dfn}
 The wrapped Floer cohomology
of $(L_0, \varepsilon_0)$ and $(L_1, \varepsilon_1)$ is defined as
\begin{equation} \label{definition of WF}
\mathrm{HW}((L_0, \varepsilon_0), (L_1, \varepsilon_1); J_\bullet) = \varinjlim \mathrm{HF}((L_0, \varepsilon_0),
(L_1, \varepsilon_1); H_{\lambda_n}, J_\bullet).
\end{equation}
\end{dfn}

Wrapped Floer cohomology is well defined, in the sense that it is independent of the choice
of the almost complex structure $J_\bullet$, and of the smoothing of the piecewise linear
functions $H_{\lambda_n}$ and of the sequence $\lambda_n$. Invariance of the almost
complex structure follows form Proposition~\ref{prop: commutation between continuation
maps}. Invariance of the smoothing of $H_{\lambda_n}$ follows from Lemma~\ref{properties
of Phi} and Corollary~\ref{doi coge}. Finally, if $\lambda_n' \to + \infty$ is another
sequence such that no $\lambda_n'$ is not the length of a Reeb chord from $\Lambda_0$ to
$\Lambda_1$, we can make both $\lambda_n$ and $\lambda_n'$ subsequences of a diverging
sequences $\lambda_n''$ and standard properties of the direct limit give canonical
isomorphisms
\begin{align*}
\varinjlim \mathrm{HF}((L_0, \varepsilon_0), (L_1, \varepsilon_1); H_{\lambda_n''}, J_\bullet) \cong &
\varinjlim \mathrm{HF}((L_0, \varepsilon_0), (L_1, \varepsilon_1); H_{\lambda_n}, J_\bullet) \\
\cong & \varinjlim \mathrm{HF}((L_0, \varepsilon_0), (L_1, \varepsilon_1); H_{\lambda_n'}, J_\bullet).
\end{align*}
Therefore  $\mathrm{HW}((L_0, \varepsilon_0), (L_1, \varepsilon_1))$ does not depend on the sequence
$\lambda_n$ up to isomorphism. It can also be proved that it is invariant under  safe isotopies,
but we will need, and prove, only invariance under compactly supported ones.
\begin{lemma}
Let $G \colon \R \times L_1 \to \R$ be a local Hamiltonian function which satisfies $dG_t=0$ outside a compact subset of $(0,1) \times L_1$ and let $\psi_t \circ \iota_1$ be the exact regular homotopy it generates, which is assumed to be a safe isotopy. If $L_1'= \psi_1(L_1)$,
$J_\bullet'$ is an $(L_0, L_1')$-regular almost complex structure and $\varepsilon_1'$ is
the augmentation for $L_1'$ with respect to $J_\bullet'$ corresponding to $\varepsilon_1$,
then there is an isomorphism
$$\mathrm{HW}((L_0, \varepsilon_0), (L_1, \varepsilon_1); J_\bullet) \cong
\mathrm{HW}((L_0, \varepsilon_0), (L_1', \varepsilon_1'); J_\bullet').$$
\end{lemma}
\begin{proof}
It is enough to observe that, for every $n$, the isomorphisms
$$\mathrm{CF}((L_0, \varepsilon_0), (L_1, \varepsilon_1); H_{\lambda_n}, J_\bullet) \cong
\mathrm{CF}((L_0, \varepsilon_0), (L_1', \varepsilon_1'); H_{\lambda_n}, J_\bullet')$$
defined in Subsection~\ref{sss: compactly supported isotopies} commute with the
continuation maps $\Phi_{\lambda_n, \lambda_m}$ and therefore define isomorphisms
of direct systems. This follows from Lemma~\ref{lemma: commutation between Phi and Psi}
and Proposition~\ref{prop: commutation between
continuation maps}.
\end{proof}
\subsection{A sketch of the chain level construction}
\label{chain level}
Here we recall very briefly the definition of the wrapped Floer complex and the
$A_\infty$-operations. Since Lagrangian immersions will appear only in an
intermediate step of the proof of the main theorem, we will not try to make them objects
of an enlarged wrapped Fukaya category. Presumably this can be done as in the embedded
case, but we have not checked the details of the construction of the necessary coherent
Hamiltonian perturbations.

Let $L_0$ and $L_1$ be exact Lagrangian immersions which intersect transversely and are
cylindrical over chord generic Legendrian submanifolds. We fix a wrapping Hamiltonian
$H \ge 0$ as in Equation~\eqref{H_lambda} such that, for every $w \in \N$, the Hamiltonian
$wH$ is compatible with $L_0$ and $L_1$ (in the sense of Definition~\ref{compatible
hamiltonian}). We also fix an $(L_0, L_1)$-regular almost complex structure $J_\bullet$.

Let $\varepsilon_0$ and $\varepsilon_1$ be augmentations of the obstruction algebras of
$(L_0, J_0)$ and $(L_1, J_1)$ respectively. Following \cite{OpenString} we define the
{\em wrapped Floer chain complex} as the $\F[q]/(q^2)$-module
\begin{equation}\label{wrapped complex}
\mathrm{CW}((L_0, \varepsilon_0), (L_1, \varepsilon_1); J_\bullet) = \bigoplus_{w=0}^\infty
\mathrm{CF}((L_0, \varepsilon_0), (L_1, \varepsilon_1); wH, J_\bullet)[q]
\end{equation}
with a differential $\mu^1$ such that, on $x+yq \in \mathrm{CF}((L_0, \varepsilon_0), (L_1,
\varepsilon_1); wH, J_\bullet)[q]$, it is defined as
$$\mu^1(x+yq)= \partial x + y + \Phi_w(y) + (\partial y)q$$
where
$$\Phi_w \colon  \mathrm{CF}((L_0, \varepsilon_0), (L_1, \varepsilon_1); wH, J_\bullet) \to
\mathrm{CF}((L_0, \varepsilon_0), (L_1, \varepsilon_1); (w+1)H, J_\bullet)$$
is the continuation map for the change of Hamiltonian defined in Subsection~\ref{sss:
changing hamiltonian}.
\begin{rem}
The endomorphism $\iota_q$ (denoted $\partial_q$ in \cite{OpenString}) defined as
$$\iota_q(x+yq)=y$$
is a chain map. However, its action in homology is trivial.
\end{rem}
\begin{rem}
\label{w=0}
The direct sum \eqref{wrapped complex} starts from
$w=1$ in \cite{OpenString}. It is equivalent to start from $w=0$, when possible, by
\cite[Lemma~3.11]{OpenString}. The homology of $\mathrm{CW}((L_0, \varepsilon_0), (L_1,
\varepsilon_1); J_\bullet)$ is isomorphic to $\mathrm{HW}((L_0, \varepsilon_0), (L_1, \varepsilon_1))$
defined as the direct limit in Equation~\eqref{definition of WF} by
\cite[Lemma~3.12]{OpenString}.
\end{rem}

The $A_\infty$-operations between wrapped Floer complexes are defined by counting
pseudoholomorphic polygons with carefully constructed Hamiltonian perturbations. In the
immersed case, those polygons will be allowed to have boundary punctures at double points
and, as usual, must be counted with a weight coming from the augmentations. The only
thing we need to know about the operations between wrapped Floer cohomology groups is
that the component
$$\mu^d \colon \mathrm{CF}(L_{d-1}, L_d) \otimes \ldots \otimes  \mathrm{CF}(L_0, L_1) \to \mathrm{CF}(L_0, L_d)$$
of the operation
$$\mu^d_{\mathcal WF} \colon \mathrm{CW}(L_{d-1}, L_d) \otimes \ldots \otimes \mathrm{CW}(L_0, L_1) \to
\mathrm{CW}(L_0, L_d)$$
coincides with the operation $\mu^d$ defined in Equation \eqref{eqn: products}. For
simplicity of notation we have dropped the augmentations from the above formulas.

\section{A trivial triviality result}
\label{sec: trivial triviality}
An exact Lagrangian embedding with cylindrical ends which is disjoint from the skeleton is known to have vanishing
wrapped Floer cohomology. This was proven in \cite[Theorem 9.11(b)]{CieliebakOancea}
but also follows from e.g.~\cite[Section 5.1]{OpenString}. Note that the statement is false
in the more general case when the Lagrangian is only assumed to be monotone. In this section we extend this
classical vanishing result to our setting of \emph{exact} Lagrangian immersions.

\subsection{Action and energy}
In this subsection we define an action for double points of immersed exact Lagrangian submanifolds and for
Hamiltonian chords and prove action estimates for various continuation maps. Let $p \in W$
be a double point of a Lagrangian immersion $(L, \iota)$ with potential $f$. We recall that
there are points $p_\pm \in L$ characterised by $\iota^{-1}(p)= \{ p_+, p_- \}$ and $f(p_+) >
f(p_-)$. We define the {\em action} of $p$ as
$$\mathfrak{a}(p)= f(p_+) - f(p_-).$$
If $L$ is disconnected and $p$ is in the intersection between the images of two connected
components, then $\mathfrak{a}(p)$ depends on the choice of the potential function $f$,
otherwise it is independent of it.

Given a holomorphic map $(\boldsymbol{\zeta}, u) \in
\widetilde{\mathfrak{N}}_L(p_0; p_1, \ldots, p_d; J)$, Stokes's theorem immediately yields
$$\int_{\Delta_{\boldsymbol{\zeta}}} u^*d \theta = \mathfrak{a}(p_0) - \sum_{i=1}^d
\mathfrak{a}(p_i).$$
Since $\int_{\Delta_{\boldsymbol{\zeta}}} u^*d \theta >0$ for a nonconstant $J$-holomorphic
map, if
$$\widetilde{\mathfrak{N}}_L(p_0; p_1, \ldots, p_d; J) \neq \emptyset,$$
we obtain
\begin{equation}
\mathfrak{a}(p_0) - \sum_{i=1}^d
\mathfrak{a}(p_i) >0.
\end{equation}

Given two Lagrangian submanifolds $L_0$ and $L_1$ with potentials $f_0$ and $f_1$ and a
Hamiltonian function $H$, we  define the action of a Hamiltonian chord $x \colon [0,1] \to W$
from $L_0$ to $L_1$ as
\begin{equation} \label{action of a chord}
{\mathcal A}(x) = \int_0^1 x^* \theta - \int_0^1 H(x(t))dt +
f_0(x(0))-f_1(x(1)).
\end{equation}
Note that this is the negative of the action used in \cite{ritter-tqft}.
\begin{ex}
Let $H \colon W \to \R$ be a cylindrical Hamiltonian such that $H(w)= h(e^{\mathfrak{r}(w)})$,
where $h \colon \R^+ \to \R$.
Then a Hamiltonian chord $x \colon [0,1] \to W$ from $L_0$ to $L_1$ is contained
in a level set $\mathfrak{r}^{-1}(r)$ and has action
\begin{equation}\label{action}
{\mathcal A}(x)= h'(e^r)e^r - h(e^r) + f_0(x(0)) - f_1(x(1)).
\end{equation}
\end{ex}

The following lemma, which we prove in the more general case of the moduli spaces of Floer
solutions with an $s$-dependent Hamiltonian, applies equally to the particular case of moduli
spaces used in the definition of the Floer differential. We introduce the following notation.
Given a set $A$ and a function $f \colon A \to \R$, we denote $\| f \|_\infty^+ := \sup \limits_{a
\in A} \max \{f(a), 0 \}$.
\begin{lemma}
Let $H_s \colon \R \times [0,1] \times W \to \R$ be an $s$-dependent cylindrical Hamiltonian
function satisfying conditions (i), (ii), and (iii) of Subsection~\ref{sss: changing hamiltonian}.
We make the simplifying assumption that $\partial_sH_s \equiv 0$ if $s \not \in [-1, 1]$. If
$\mathfrak{M}_{L_0, L_1}(\mathbf{p}^1, x_-, \mathbf{p}^0, x_+; H_s, J_\bullet) \ne
\emptyset$, then
\begin{equation}\label {esasperazione}
{\mathcal A}(x_-) \le {\mathcal A}(x_+) + 6 \| \partial_sH_s \|_\infty^+.
\end{equation}
\end{lemma}
Note that Equation~\eqref{esasperazione} is far from being sharp, but there will be no need
for a sharper estimate.
\begin{proof}
Let $(u, \boldsymbol{\zeta}^0,   \boldsymbol{\zeta}^1) \in
\mathfrak{M}_{L_0, L_1}(\mathbf{p}^1, x_-, \mathbf{p}^0, x_+; H_s, J_\bullet)$.
Then, in a metric on $u^*TW$ induced by $d \theta$ and $J_\bullet$,
\begin{align*}  \int_{- \infty}^{+\infty} \int_0^1 | \partial_s u |^2 dt ds & = \int_{- \infty}^{+\infty}
\int_0^1 d \theta (\partial_s u, \partial_t u - \chi'(t) X_{H_s} ( \chi(t) ,u)) dt ds \\
& = \int_{Z_{\boldsymbol{\zeta}^0,   \boldsymbol{\zeta}^1}} u^* d \theta - \int_{- \infty}^{+\infty}
\int_0^1 \chi'(t) dH_s( \chi(t) \partial_su) dt ds.
\end{align*}
Using Stokes's theorem we obtain
\begin{align*}\int_{Z_{\boldsymbol{\zeta}^0,   \boldsymbol{\zeta}^1}} u^* d \theta &= f_1(x_-(1))
-  f_1(x_+(1)) -\\ & f_0(x_-(0)) + f_0(x_+(0))  - \sum_{i=0}^1 \sum_{j=1}^{l_i}
\mathfrak{a}(p_j^i).
\end{align*}

Using the equality $\partial_s (H_s \circ u)= (\partial_s H_s) \circ u + dH_s(\partial_s u)$
we obtain
\begin{align*}
& \int_{- \infty}^{+\infty} \int_0^1 \chi'(t) dH_s(\chi(t), (\partial_s u)) dt ds = \\
& \int_{- \infty}^{+\infty} \int_0^1 \chi'(t) \partial_s (H_s (\chi(t),u(s,t)))dt ds - \\  &
\int_{- \infty}^{+\infty} \int_0^1 \chi'(t) (\partial_s H_s)(\chi(t), u(s,t)) dt ds.
\end{align*}
We can compute
\begin{align*}
& \int_{- \infty}^{+\infty} \int_0^1 \chi'(t) \partial_s (H_s (\chi(t),u(s,t)))dt ds = \\
& \int_0^1 \chi'(t) H_+ (\chi(t), x_+(\chi(t))) dt - \int_0^1 \chi'(t) H_- (\chi(t), x_-(\chi(t))) dt=
\\
& \int_0^1 H_+ (t, x_+(t)) dt - \int_0^1 H_- (t, x_-(t)) dt.
\end{align*}

Thus, rearranging the equalities, we have
\begin{align*}
{\mathcal A}(x_+) - {\mathcal A}(x_-) = & \int_{- \infty}^{+\infty}
\int_0^1 | \partial_s u |^2 dt ds + \sum_{i=0}^1 \sum_{j=1}^{l_i} \mathfrak{a}(p_j^i) -\\
&  \int_{- \infty}^{+\infty} \int_0^1 \chi'(t) \partial_sH_s(\chi(t) u(s,t)) dsdt.
\end{align*}
Finally, we estimate
$$\int_{- \infty}^{+\infty} \int_0^1 \chi'(t) (\partial_s H_s)(\chi(t), u(s,t)) dt ds \le 6 \|
\partial_sH_s \|_\infty^+$$
and obtain Equation~\eqref{esasperazione}.
\end{proof}
\begin{cor}\label{cor: action decreases}
The differential in $\mathrm{CF}(L_0, L_1; H, J_\bullet)$ decreases the action. If $\partial_sH_s \le 0$, then the continuation map $\Phi_{H_s}$ also decreases the action.
\end{cor}

Now we turn our attention at the continuation map $\Psi_G$ defined in
Subsection~\ref{sss: compactly supported isotopies}. Let
$G \colon \R \times L_1 \to \R$ be a local Hamiltonian function such that
 $dG_t=0$ outside a compact subset of $(0,1) \times L_1$ and
let $\psi_t \circ \iota_1$ be the compactly supported exact regular
homotopy it generates. Now assume that $\psi_t$ is a safe
isotopy. Denote, as usual, $L_1'= \psi_1(L_1)$.

First we make the following remark about a special type of safe
isotopy and the action of the image of the double points.
\begin{rem} Let $(L, \iota)$ be an exact Lagrangian immersion and
 $\psi_t \colon W \to W$ a smooth isotopy. If $\psi_t^*\theta=e^{c(t)}\theta$, then

\begin{enumerate}
\item $\psi_t(L)$ is a safe isotopy, and
\item if $p$ is a double point of $(L, \iota)$, then $\psi_t(p)$ is a
double point of $(L, \psi_t \circ \iota)$ whose action satisfies $\mathfrak{a}(p) =
e^{c(t)}\mathfrak{a}(\psi_t(p))$.
\end{enumerate}
\end{rem}
Given a Hamiltonian function $H \colon [0,1] \times W \to \R$ which is compatible both
with $L_0$ and $L_1$, as well as with $L_0$ and $L_1'$, let ${\mathcal C}_H$ be the
set of Hamiltonian chords of $H$ from $L_0$ to $L_1$ and let ${\mathcal C}_H'$ be the set
of Hamiltonian chords of $H$ from $L_0$ to $L_1'$.

Observe that any safe Lagrangian isotopy from $L_1$ to $L_1'$ induces a continuous
family of potentials $f_1^s.$ Fixing a choice of of local Hamiltonian $G \colon \R \times
L_1 \to \R$ generating the safe isotopy makes the potential $f_1'$ on $L_1'$ determined by the choice of potential $f_1$ on $L_1$ via a computation as in the proof of Lemma \ref{variation of Liouville}.

\begin{lemma}\label{frignolo}
For every chords $\mathbf{x}_- \in {\mathcal C}_H'$ and $\mathbf{x}_+ \in
{\mathcal C}_H$ and for every sets of self-intersection points $\mathbf{p}^0 = (p_1^0,
\ldots, p_{l_0}^0)$ of $L_0$ and $\mathbf{p}^1 = (p_1^1, \ldots, p_{l_1}^1)$ of $L_1$,
if $\mathfrak{M}_{L_0, L_1}(\mathbf{p}^1, x_-, \mathbf{p}^0, x_+; H, G, \tilde{J}, \nu) \ne
\emptyset$, then
$${\mathcal A}(x_-) \le {\mathcal A}(x_+)+ 2 \mu \| G \|_\infty,$$
where $\| G \|_\infty$ is the supremum norm of $G$ and $\mu \ge 0$ is the measure of the subset $\{s \in \R\}$ for which $G_s \colon L \to \R$ is not constantly zero.
\end{lemma}
\begin{proof}
Consider $(\boldsymbol{\zeta}^0, \boldsymbol{\zeta}^1, u) \in
\mathfrak{M}_{L_0, L_1}(\mathbf{p}^1, x_-, \mathbf{p}^0, x_+; H, G, \tilde{J}, \nu)$.
Observe that the map $u \colon Z_{\boldsymbol{\zeta}^0, \boldsymbol{\zeta}^1} \to W$
extends to a continuous map $u \colon Z \to W$. We have:
\begin{align*}
\int_{- \infty}^{+\infty} \int_0^1 | \partial_s u |^2 dt ds = &
\int_Z d u^*\theta - \int_0^1 \left (\int_{- \infty}^{+\infty}\left (
\frac{\partial}{\partial s}(H \circ u) \right ) ds \right ) dt =   \\
& \int_Zu^*d \theta  - \int_0^1 H(x_+(t))dt + \int_0^1 H(x_-(t))dt.
\end{align*}

We denote by $u_i \colon \R \to W$, for
$i=0,1$, the continuous and piecewise smooth maps  $u_i(s)=u(s,i)$ and use Stokes theorem:
$$\int_Z u^*d\theta = \int_0^1 x_+^* \theta - \int_0^1 x_-^* \theta + \int_{\R} u_0^*\theta -
\int_{\R}  u_1^*\theta.$$
Let $f_0$ and $f_1$ be the potentials of $L_0$ and $L_1$ respectively, and $\tilde{f}_1$
the potential of $\psi_1(L_1)$.
The map $u_0$ takes values in $L_0$, and therefore
$$\int_{\R} u_0^*\theta = f_0(x_+(0)) - \sum_{j=1}^{l_0} \mathfrak{a}(p_j^0) -
f_0(x_-(0)).$$
We are left with the problem of estimating $\int_{\R} u_1^* \theta$, which is slightly more
complicated here because $u_1(s) \in \psi_{\nu_{\boldsymbol{\zeta}^1}(s)}(L_1)$.
We will denote $\psi^{\nu}_s \coloneqq \psi_{\nu_{\boldsymbol{\zeta}^1}(s)}$. Recall that
$\psi^\nu_s$ is a smooth isotopy inducing a safe isotopy of $L_1$ generated by the local
Hamiltonian function $G^\nu(s, w) = \nu_{\boldsymbol{\zeta}^1}'(s)
G(\nu_{\boldsymbol{\zeta}^1}(s), w),$ $w \in L_1.$

We use the following trick.
Consider $W \times \R \times \R$ with the Liouville form $\Theta \coloneqq \theta + \tau d\sigma$,
where $\sigma$ is the coordinate of the first copy of $\R$ and $\tau$ is the coordinate in the
second copy. The notation here conflicts with the use of $(\sigma, \tau)$ as coordinates in
the strip-like ends near the boundary punctures, but this will not cause confusion.

Consider the symplectic suspension
$$\Sigma \coloneqq \{(x,s,t) \in W \times \R^2; \:\: x \in \psi^{\nu}_s(y),\:y \in L_1,\:t=-G^\nu(s, y).\}$$
of the isotopy $\psi^{\nu}_s(L_1),$ which is an exact Lagrangian immersion. This should be seen as a corrected version of the trace of the isotopy, in order to make it Lagrangian.

Lift $u_1 \colon \R \to W$ to $\tilde{u}_1 \colon \R \to W \times \R \times \R$ by defining
$$\tilde{u}_1(s) = (u_1(s), s, -G^{\nu}(s, \overline{u}_1(s))) \in \Sigma,$$
and where $\overline{u}_1$ is the lift of $u_1$ to $L$ which is smooth away from the
punctures.

Using the computation in the proof of Lemma \ref{variation of Liouville}, together with the
Lagrangian condition satisfied by $\Sigma,$ we obtain
$$\int_{-\infty}^{+\infty}\widetilde{u}_1^*\,\Theta= f'_1(x_+(1)) - \sum_{j=1}^{l_1}
\mathfrak{a}(p_j^1) - f_1(x_-(1)),$$
as well as
$$u_1^*\,\theta-\widetilde{u}_1^*\,\Theta = G^\nu(s, \overline{u}_1(s)) d\sigma.$$
Observe that we here abuse notation, and use $\mathfrak{a}(p_j^1)>0$ for the action computed with respect to the induced potential function on $\psi^{\nu}_s(L_1)$ for the corresponding value of $s \in \R.$

Since
$\|\nu_{\boldsymbol{\zeta}^1}' \|_{\infty} \le 2$, we finally obtain
$${\mathcal A}(x_-) \le {\mathcal A}(x_+) + 2 \mu \| G \|_{\infty},$$
where $\mu \ge 0$ is as required.
\end{proof}

\subsection{Pushing up}\label{ss: pushup}
In this subsection we prove the following proposition.
\begin{prop}\label{prop: trivial triviality}
Let $(L_0, \iota_0)$ and $(L_1, \iota_1)$ be exact Lagrangian immersions in a Liouville
manifold $(W, \theta)$ and let $J_\bullet$ be an $(L_0, L_1)$-regular almost
complex structure.  If the Liouville flow of $(W, \theta)$ displaces $L_1$ from any compact
set, then, for all pair of augmentations $\varepsilon_0$ and $\varepsilon_1$ of the obstruction
algebras of $(L_0, J_0)$ and $(L_1, J_1)$ respectively,
$$\mathrm{HW}((L_0, \varepsilon_0), (L_1, \varepsilon_1), J_\bullet)=0.$$
\end{prop}

We postpone the proof to after a couple of lemmas. Given $\Lambda > \lambda >0$ and $L>0$,
we define a function
$h_{\lambda, \Lambda, R} \colon \R^+ \to \R$ such that
\begin{equation}\label{two steps hamiltonian}
h_{\lambda, \Lambda, R}(\rho) = \begin{cases}
 0 & \text{for } \rho \le 1, \\
 \lambda \rho - \lambda & \text{for } \rho \in [1, e^R], \\
\Lambda \rho - (\Lambda - \lambda)e^R - \lambda & \text{for } \rho \ge e^R.
\end{cases}
\end{equation}
See Figure \ref{fig:graph1} for the graph of $h_{\lambda, \Lambda, R}$.

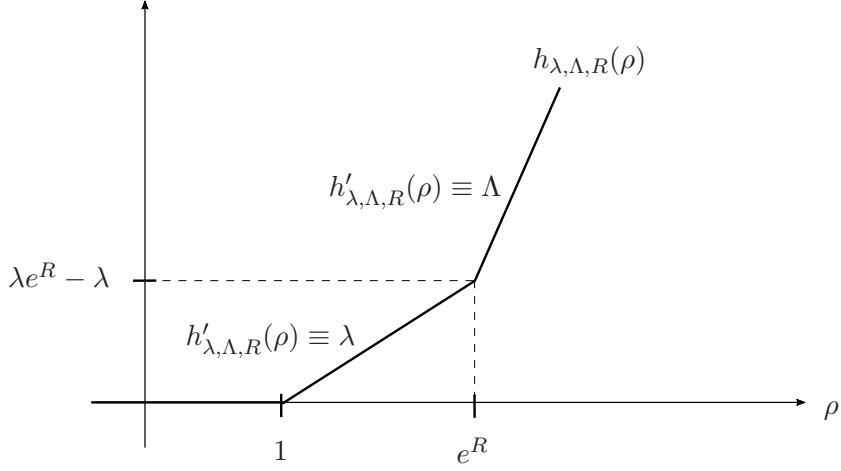
\begin{figure}[t]
\labellist
\pinlabel $\rho$ at 297 15
\pinlabel $1$ at 76 -1
\pinlabel $e^R$ at 153 -1
\pinlabel $\lambda e^R-\lambda$ at -13 69
\pinlabel $h'_{\lambda,\Lambda,R}(\rho)\equiv\Lambda$ at 130 103
\pinlabel $h'_{\lambda,\Lambda,R}(\rho)\equiv\lambda$ at 72 43
\pinlabel $h_{\lambda,\Lambda,R}(\rho)$ at 200 155
\endlabellist
\vspace{3mm}
\centering
\includegraphics[height=5cm]{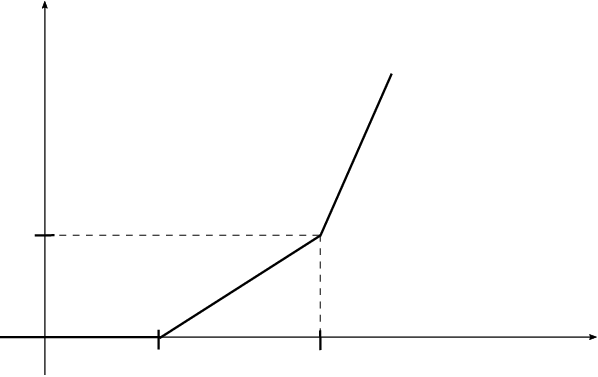}
\caption{The graph of $h_{\lambda, \Lambda, R}$.}
\label{fig:graph1}
\end{figure}

The function $h_{\lambda, \Lambda, R}$ has two corners: one at $(1,0)$ and one at $(e^R,
\lambda e^R - \lambda)$. We smooth $h_{\lambda, \Lambda, R}$ in a small neighbourhood of
these corners so that the new function (which we still denote by $h_{\lambda, \Lambda, R}$) satisfies
$h_{\lambda, \Lambda, R}'' (\rho) \ge 0$ for all $\rho$. We define the (time independent)
cylindrical Hamiltonian $H_{\lambda, \Lambda, R} \colon W \to \R$ by $H_{\lambda, \Lambda, R}(w)=
h_{\lambda, \Lambda, R}(e^{\mathfrak{r}(w)})$.
We make the assumption that there is no
Hamiltonian time-1 chord from $L_0$ to $L_1$  on $\partial W_r$ when $r$ satisfies either
$h_{\lambda, \Lambda, R}'(e^r)= \lambda$ or $h_{\lambda, \Lambda, R}(e^r)= \Lambda$.
This is equivalent to assuming that there is no Reeb chord from
$\Lambda_0$ to $\Lambda_1$ of length either $\lambda$ or $\Lambda$.

We assume, without loss of generality, that $L_0$ and $L_1$ intersect transversely, that $L_i
\cap W_1^e$ is a cylinder over a Legendrian submanifold $\Lambda_i$ and that all Reeb chords
from $\Lambda_0$ to $\Lambda_1$ are nondegenerate. Then we have three types of chords:
\begin{itemize}
\item constant chords, i.e. intersection points between $L_0$ and $L_1$, which are contained
in $W_0$,
\item chords coming from smoothing the corner of $h_{\lambda, \Lambda, R}$, which are
concentrated around $\partial W_1$, and
\item chords coming from smoothing the corner of $h_{\lambda, \Lambda, R}$ at $(e^R, \lambda
e^R - \lambda)$, which are concentrated around $\partial W_L$.
\end{itemize}
Constant chords and chords coming from smoothing the first corner will be called {\em type
I chords}, while chords coming from smoothing the second corner will be called {\em type II
chords}. We say that a chord of $H_{\Lambda, \lambda, L}$ {\em appears at slope $s$} if
it is contained in $\partial W_r$ for $r$ such that $h_{\Lambda, \lambda, L}'(e^r)=s$. By abuse of
terminology, we will consider the intersection points between $L_0$ and $L_1$ as chords
appearing at slope zero.
\begin{lemma}\label{action separation from I and II}
Given $\lambda>0$, there exists $C>0$ such that, for every $\Lambda > \lambda$ and
every $R \ge C$, every chord of type II of $H_{\lambda, \Lambda, R}$ has larger action than
any chord of type I.
\end{lemma}
\begin{proof}
If $x$ is a Hamiltonian chord contained in $\partial W_r$, then the action of $x$ is
\begin{equation}\label{eq: action of chords}
{\mathcal A}(x)= h_{\lambda, \Lambda, R}'(e^r)e^r - h_{\lambda, \Lambda, R}(e^r) +
f_0(x(0)) - f_1(x(1)).
\end{equation}

Observe that $|f_0(x(0)) - f_1(x(1))|$ is uniformly bounded because
$f_0$ and $f_1$ are locally constant outside of a compact set. The Hamiltonian chords  of
 type I appear at slope $\lambda_- < \lambda$ and near $\partial W_0$, and therefore $r$ in
Equation~\eqref{eq: action of chords} is close to zero. Then, there is a
constant $C_-$, depending on $f_0$, $f_1$ and the smoothing procedure at the first corner
such that, if $x$ is a chord of type I, then ${\mathcal A}(x) \le \lambda_- + C_- $.

On the other hand, if $x$ is a chord of type II, then it appears at slope
$\lambda_+>\lambda$ and around $r= \log(R)$. Then there is a constant $C_+$, depending
on $f_0$, $f_1$ and the smoothing procedure at the second corner such that
${\mathcal A}(x) \ge R \lambda_+ - \lambda R + \lambda -C_+ =
R(\lambda_+ - \lambda) + \lambda -C_+$. The lemma
follows from $\lambda_+ - \lambda>0$ and the fact that chords arise at a discrete set of
slopes.
\end{proof}
From now on we will always take $R \ge C$.
The consequence of Lemma \ref{action separation from I and II} is that the chords of type I generate a subcomplex of
$$\mathrm{CF}((L_0, \varepsilon_0), (L_1, \varepsilon_1);H_{\lambda, \Lambda, R}, J_\bullet)$$
which we will denote by $\mathrm{CF}^I((L_0, \varepsilon_0), (L_1,
\varepsilon_1); H_{\lambda, \Lambda, R}, J_\bullet)$. The main ingredient in the proof of
Proposition~\ref{prop: trivial triviality} is the following lemma.
\begin{lemma}\label{too much energy}
If the Liouville flow of $(W, \theta)$ displaces $L_1$ from any compact set, then the
inclusion map
$$\mathrm{CF}^I((L_0, \varepsilon_0), (L_1, \varepsilon_1); H_{\lambda, \Lambda, R}, J_\bullet)
\hookrightarrow \mathrm{CF}((L_0, \varepsilon_0), (L_1, \varepsilon_1); H_{\lambda, \Lambda, R}, J_\bullet)$$
is trivial in homology whenever $\Lambda \gg 0$ is sufficiently large.
\end{lemma}
\begin{proof}
The Liouville flow applied to $L_1$ gives rise to a compactly supported safe isotopy from $L_1$ to $L_1',$ and is generated by the time-dependent local Hamiltonian $G \colon \R \times L \to \R$ for which $dG_t=0$ outside of a compact subset of $(0,1) \times L \to \R$; see Lemma \ref{lem:GeneratingHamiltonian}. Since the Liouville form is \emph{conformally} symplectic, it actually preserves the space of compatible almost complex structures which are cylindrical at infinity.

We will choose to apply the Liouville flow so that $L_1' \subset \{\rho \ge e^R\}$; recall that this is possible by our assumptions.

The continuation map
$$\Psi_G \colon \mathrm{CF}((L_0, \varepsilon_0), (L_1, \varepsilon_1); H_{\lambda, \Lambda, R}, J_\bullet)
\to  \mathrm{CF}((L_0, \varepsilon_0), (L_1', \varepsilon_1'); H_{\lambda, \Lambda, R}, J_\bullet')$$
defined in Subsection~\ref{sss: compactly supported isotopies}
induces an isomorphism in homology if $J_\bullet'$ is an $(L_0, L_1')$-regular almost complex
structure such that $J_0'=J_0$ and $J_1' = (\psi_1)_*J_1$, and $\varepsilon_1'$ is the
augmentation of the obstruction algebra of $(L_1', J_1')$ defined by $\varepsilon_1' =
\varepsilon_1 \circ \psi_1^{-1}$.

 By Lemma~\ref{frignolo}, there is a constant $C$, independent of
$\Lambda$, such that $\Psi_G(x)$ is a linear combination of chords of action
at most $C$ whenever $x$ is a chord from $L_0$ to $L_1$ of type I.

On the other hand, $\mathrm{CF}((L_0, \varepsilon_0), (L_1', \varepsilon_1);
H_{\lambda, \Lambda, R})$ is generated by Hamiltonian chords $x$ of action
$${\mathcal A}(x)= (\Lambda - \lambda)e^R + \lambda + f_0(x(0))-
f_1'(x(1)).$$
Here we have used $L_1' \subset \{\rho \ge e^R\},$ together with the particular form of $H_{\lambda, \Lambda, R}$ in the same subset. This implies that, for $\Lambda$ large enough, $\Psi_G(x)=0$ for all chords $x$
of type $I$.
\end{proof}
\begin{proof}[Proof of Proposition~\ref{prop: trivial triviality}]
For every $\lambda$ and $\Lambda$ there are continuation maps
\begin{align*}
\Phi_{\lambda, \Lambda}^{(1)} \colon & \mathrm{CF}((L_0, \varepsilon_0), (L_1, \varepsilon_1);
H_\lambda, J_\bullet) \to \mathrm{CF}((L_0, \varepsilon_0), (L_1, \varepsilon_1); H_{\lambda, \Lambda, R}, J_\bullet), \\
\Phi_{\lambda, \Lambda}^{(2)} \colon & \mathrm{CF}((L_0, \varepsilon_0), (L_1, \varepsilon_1);
H_{\lambda, \Lambda, R}) \to \mathrm{CF}((L_0, \varepsilon_0), (L_1, \varepsilon_1); H_\Lambda, J_\bullet)
\text{ and}\\
\Phi_{\lambda, \Lambda} \colon & \mathrm{CF}((L_0, \varepsilon_0), (L_1, \varepsilon_1);
H_\lambda, J_\bullet) \to \mathrm{CF}((L_0, \varepsilon_0), (L_1, \varepsilon_1); H_\Lambda, J_\bullet)
\end{align*}
such that there is a chain homotopy between $\Phi_{\lambda, \Lambda}$ and
$\Phi_{\lambda, \Lambda}^{(2)} \circ\Phi_{\lambda, \Lambda}^{(1)}$.

We can assume that $\Phi_{\lambda, \Lambda}$, $\Phi_{\lambda, \Lambda}^{(1)}$ and
$\Phi_{\lambda, \Lambda}^{(2)}$ are defined using $s$-dependent Hamiltonians $H_s^*$ ($*=
\emptyset, (1), (2)$) such that $\partial_s H_s^* \le 0$, and therefore they decrease the action.
Hence, the image of  $\Phi_{\lambda, \Lambda}^{(1)}$
is contained in $$\mathrm{CF}^I((L_0, \varepsilon_0), (L_1, \varepsilon_1); H_{\lambda, \Lambda, R},
J_\bullet)$$
(here we use Lemma \ref{action separation from I and II}) and therefore it follows from Lemma~\ref{too much energy} that $\Phi_{\lambda, \Lambda}=0$
in homology. By the definition of wrapped Floer cohomology as a direct limit, this implies that
\begin{equation*}
\mathrm{HW}((L_0, \varepsilon_0), (L_1, \varepsilon_1), J_\bullet) =0.\qedhere
\end{equation*}

\end{proof}

\section{Floer cohomology and Lagrangian surgery}
\label{sec: surgeries and cones}
Lalonde and Sikorav in \cite{Lalonde_&_Lagrangiennes_exactes} and then
Polterovich in \cite{Polterovich_Surgery_Lagrange} defined a surgery
operation on Lagrangian submanifolds. It is expected that Lagrangian
surgery should correspond to a twisted complex (i.e.~an iterated
mapping cone) in the Fukaya category. Results in this direction have
been proved by Seidel in \cite{SeidelExact}, Fukaya, Oh, Ohta and Ono in
\cite{fooo-c10} and by Biran and Cornea in \cite{BirCo2}. After a first version of this article had appeared, Palmer and Woodward gave a more comprehensive treatment of Lagrangian surgery in \cite{PalmerWoodward}. Our goal in this section is to establish Proposition \ref{prop:surgeryTW}, which provides us with a result along these lines in the generality that we need.

The difficult point in handling the Lagrangian surgery from the Floer
theoretic perspective is that, except in very favourable situations,
the Lagrangian submanifolds produced are not well behaved from the point-of-view
of pseudoholomorphic discs. In our situation, we turn out to be lucky,
since only surgeries that preserve exactness are needed. Nevertheless,
there still is a complication stemming from the fact that the resulting Lagrangian only is immersed (as opposed to embedded). This is the main reason for the extra work needed, and here we rely on the theory developed in the previous sections.

The bounding cochains that we will consider in this exact immersed setting are those corresponding to augmentations of the corresponding obstruction algebras introduced in Section \ref{sec: obstructions}. This turns out to be a very useful perspective, since it enables us to apply techniques from Legendrian contact cohomology in order to study them.

\subsection{The Cthulhu complex}\label{ss: cthulhu}
In this subsection we recall, and slightly generalise, the definition of Floer cohomology for
Lagrangian cobordisms we defined in \cite{Floer_cob}.
\begin{dfn}\label{dfn: lagrangian cobordisms}
Given cylindrical exact
Lagrangian immersions $L_+$ and $L_-$ in $(W, \theta)$ which coincide outside of a
compact set, an exact Lagrangian cobordism $\Sigma$ from $L_-^+$ to $L_+^+$ is a
properly embedded submanifold
$$\Sigma \subset (\R \times M,d(e^t\beta))=(\R \times W \times \R,d(e^t(dz+\theta)))$$
such that, for $C$
and $R$ sufficiently large,
\begin{enumerate}
\item $\Sigma \cap (- \infty, -C] \times W \times \R =  (- \infty, -C] \times L_-^+$,
\item $\Sigma \cap [C, + \infty) \times W \times \R =  [C, + \infty) \times L_+^+$,
\item $\Sigma \cap (\R \times W_R^e \times \R)$ is tangent to both $\partial_s$ and the lift of
the Liouville vector field $\mathcal{L}$ of $(W, \theta)$, and  \label{eq: lateral end}
\item $e^s \alpha |_{\Sigma} = dh$ for a function $h \colon \Sigma \to \R$ which is constant on
$\Sigma \cap (- \infty, -C] \times W \times \R$.
\end{enumerate}
The intersection
$$\Sigma \cap (\R \times W_R^e \times \R)=\R \times (L^+_\pm \cap (W_R^e \times \R))$$
of \eqref{eq: lateral end} is called the {\em lateral end} of $\Sigma$.
\end{dfn}
The surgery cobordism $\Sigma(a_1, \ldots, a_k)$
defined in the previous subsection clearly satisfies all these properties when
$\mathbb L(a_1, \ldots, a_k)$ is connected.

Given two exact Lagrangian cobordisms $\Sigma^0$ and $\Sigma^1$ from $(L_-^0)^+$ to
$(L_+^0)^+$ and from $(L_-^1)^+$ to $(L_+^1)^+$ with augmentations $\varepsilon_-^0$
and $\varepsilon_-^1$ of $(L_-^0)^+$ and $(L_-^1)^+$ respectively, we define the {\em
Cthulhu complex} $\mathrm{Cth}_{\varepsilon_-^0, \varepsilon_-^1}(\Sigma^0, \Sigma^1)$ which, as an
$\F$-module, splits as a direct sum
\begin{align*}
  & \mathrm{Cth}_{\varepsilon_-^0, \varepsilon_-^1}(\Sigma^0, \Sigma^1) = \\
& \mathrm{LCC}_{\varepsilon_+^0, \varepsilon_+^1}((L_+^0)^+, (L_+^1)^+) \oplus
\mathrm{CF}_{\varepsilon_-^0, \varepsilon_-^1} (\Sigma^0, \Sigma^1) \oplus \mathrm{LCC}_{\varepsilon_-^0,
\varepsilon_-^1}((L_-^0)^+, (L_-^1)^+),
\end{align*}
where $\varepsilon_+^i$ is the augmentations of $\mathfrak{A}((L_+^i)^+)$ induced by
$\varepsilon_-^i$ and $\Sigma^i$, and $\mathrm{CF}_{\varepsilon_-^0, \varepsilon_-^1}(\Sigma^0,
\Sigma^1)$ is the $\F$-module freely generated by the intersection points $\Sigma^0 \cap
\Sigma^1$, which we assume to be transverse. Furthermore, we assume that $L_\pm^0 \cap L_\pm^1 \cap W_R^e=\emptyset,$ which is not a restriction since the ends are cylinders over Legendrian submanifolds.

The differential on the Cthulhu complex can be written as a matrix
$$\mathfrak{d}_{\varepsilon_-^0, \varepsilon_-^1}= \left ( \begin{matrix}
d_{++} & d_{+0} & d_{+-} \\
0 & d_{00} & d_{0-} \\
0 & d_{-0} & d_{--}
\end{matrix} \right )$$
where $d_{++}$ and $d_{--}$ are the differentials of
$\mathrm{LCC}_{\varepsilon_+^0, \varepsilon_+^1}((L_+^0)^+, (L_+^1)^+)$ and
$\mathrm{LCC}_{\varepsilon_-^0, \varepsilon_-^1}((L_-^0)^+, (L_-^1)^+)$ respectively, and the other
maps are defined by counting $J$-holomorphic discs in $\R \times M$ with boundary on
$\Sigma^0 \cup \Sigma^1$ and boundary punctures asymptotic to Reeb chords from
$(L_1^\pm)^+$ to $(L_1^\pm)^+$ and intersection points between $\Sigma^0$ and
$\Sigma^1$. See \cite[Section~6]{Floer_cob} for the detailed definition. The cobordisms
considered in \cite{Floer_cob} have the property that $\Sigma^i \cap [-C, C] \times M$ is
compact for every $C >0$, while here we consider cobordisms with a lateral end. The
theory developed in \cite{Floer_cob} can be extended to the present situation thank to
the following maximum principle.
\begin{lemma} \label{maximum principle}
Let $J$ and $\widetilde{J}$ be almost complex structures on $\R \times W \times \R$ and $W$, respectively, each cylindrical inside the respective symplectisation $\R \times W_R^e$ and half-symplectisation $W_R^e$ for some $R >0$. We moreover require that the the canonical projection
$$ (\R \times W_R^e \times \R,J) \to (W_R^e,\widetilde{J}) $$
is holomorphic. Then
every $J$-holomorphic map $u \colon \Delta \to \R
\times W \times \R$ with
\begin{itemize}
\item $\Delta = D^2 \setminus \{ \zeta_0, \ldots, \zeta_d \}$ where $(\zeta_0, \ldots, \zeta_d)
\in Conf^{d+1}(\partial D^2)$,
\item $u(\partial \Delta) \subset \Sigma^0 \cup \Sigma^1,$ and
\item $u$ maps some neighbourhood of the punctures $\{ \zeta_0, \ldots, \zeta_d \}$ into $\R \times W_R \times \R,$
\end{itemize}
has its entire image contained inside $\R \times W_R \times \R$.
\end{lemma}
\begin{proof}
By the assumptions the image of the curve $u|_{u^{-1}(\R \times W_R^e \times \R)}$ under the canonical projection
$$(\R \times W_R^e \times \R,J) \to (W_R^e,\widetilde{J})$$
is compact with boundary on $\R \times \partial W_R^e \times \R$. The statement is now a consequence of the maximum principle for
pseudoholomorphic curves inside $W_R^e \cong [R, + \infty) \times V$ which
\begin{itemize}
\item satisfy a cylindrical boundary condition, and
\item are pseudoholomorphic for a cylindrical almost complex structure.
\end{itemize}
Namely, by e.g.~\cite[Lemma 5.5]{Khovanov}, the symplectisation coordinate
$\mathfrak{r} \colon W_R^e \to [R, + \infty)$ restricted to such a curve cannot have a local
maximum.
\end{proof}
With Lemma \ref{maximum principle} at hand, the arguments of \cite{Floer_cob} go
through, and therefore we have the following result.
\begin{thm}[\cite{Floer_cob}]\label{great cthulhu}
The map $\mathfrak{d}_{\varepsilon_-^0, \varepsilon_-^1}$ is a differential and the Cthulhu
complex
$$(\mathrm{Cth}_{\varepsilon_-^0, \varepsilon_-^1}(\Sigma^0, \Sigma^1), \mathfrak{d}_{\varepsilon_-^0, \varepsilon_-^1})$$
is acyclic.
\end{thm}
The consequence of interest for us is the following.
\begin{cor} \label{cthulhu isomorphism}
If $\Sigma^0 \cap \Sigma^1 = \emptyset$, then the map
$$d_{+-} \colon \mathrm{LCC}_{\varepsilon_-^0, \varepsilon_-^1}((L_-^0)^+, (L_-^1)^+) \to
\mathrm{LCC}_{\varepsilon_+^0, \varepsilon_+^1}((L_+^0)^+, (L_+^1)^+)$$
is a quasi-isomorphism.
\end{cor}
\begin{proof}
If $\Sigma^0 \cap \Sigma^1 = \emptyset$, the Cthulhu differential simplify as follows:
$$\mathfrak{d}_{\varepsilon_-^0, \varepsilon_-^1}= \left ( \begin{matrix}
d_{++} & 0 & d_{+-} \\
0 & 0 & 0 \\
0 & 0 & d_{--}
\end{matrix} \right )$$
and thus the Chthulhu complex becomes the cone of $d_{+-}$. Since it is acyclic, it follows
that $d_{+-}$ is a quasi-isomorphism.
\end{proof}
\subsection{The surgery cobordism}
\label{sec:surgery-cobordism}
 In this subsection we describe the Lagrangian surgery of
\cite{Lalonde_&_Lagrangiennes_exactes} and \cite{Polterovich_Surgery_Lagrange} from the
Legendrian viewpoint. In particular, we interpret it as as a Lagrangian cobordism between
the Legendrian lifts of the Lagrangian submanifolds before and after the surgery. We refer to
\cite{LegendrianAmbient} for more details.

We first describe the local model for Lagrangian surgery. Given
$\eta,\delta>0$,
we consider the open subset
$$\mathcal{V}_{\eta,\delta}:=\{|q|<\eta,|p|<2\delta,z\in\mathbb{R}\}$$
of $\mathcal{J}^1(\mathbb{R}^n)$. Given $\zeta>0$, we denote by
$\Lambda^+_{\eta,\delta,\zeta}$ the (disconnected) Legendrian submanifold of $\mathcal{V}_{\eta,\delta}$ given by the two sheets
$$\{(q,\pm df_{\eta,\delta,\zeta}(|q|),\pm f_{\eta,\delta,\zeta}(|q|)) : |q|<\eta \},$$
where $$f_{\eta,\delta,\zeta}(s)=\frac{\delta}{2\eta}s^2+\frac{\zeta}{2}.$$
This is a Legendrian submanifold with a single Reeb chord of length $\zeta$. Note that
$\Lambda^+_{\eta,\delta,\zeta}$ is described by the generating family
$F_{\eta,\delta,\zeta}^+ \colon \mathbb{R}^n\times \mathbb{R} \to \R$ given by
$$F_{\eta,\delta,\zeta}^+(q,\xi)=\frac{\xi^3}{3}-g^+(|q|)\xi,$$
where
$$g^+(s)=\left(\frac{3}{2}f_{\eta,\delta,\zeta}(s)\right)^{\frac{2}{3}}.$$

Note that $g^+$ is smooth because $g^+(s)>0$ holds for every $s$.
Let $g^-:\mathbb{R}_+\rightarrow\mathbb{R}$ be a function such that
\begin{itemize}
\item[(i)] $g^-(s)=\left(\frac{3}{2}f_{\eta,\delta,\zeta}(s)\right)^{\frac{2}{3}}$ for
$s>3\eta/4$,
\item[(ii)] $g^-(s)<0$ for $s<\eta/2$, and
\item[(iii)] $0<(g^-)'(s)<2\frac{\delta\eta}{\delta\eta+\zeta}$.
\end{itemize}

Note that Condition (iii) can be achieved if $\zeta<\frac{7\delta\eta}{16}$.
The Legendrian submanifold $\Lambda_{\eta,\delta,\zeta}^-$ of $\mathcal{V}_{\delta,\eta}$ generated
by $$F^-_{\eta,\delta,\zeta}(q,\xi)=\frac{\xi^3}{3}-g^-(|q|)\xi$$ coincides with
$\Lambda_{\eta,\delta,\zeta}^+$ near $|q|=\epsilon$ and has no Reeb chords (see Figure \ref{fig:legsurg}). Note that indeed $\Lambda_{\eta,\delta,\zeta}^- \subset\mathcal{V}_{\delta,\eta}$ because Condition (iii) ensures that the p coordinates of $\Lambda_{\eta,\delta,\zeta}^-$, given by $\frac{\partial F^-_{\eta,\delta,\zeta}}{\partial p_i}$ along critical values of $F(q,\cdot)$, are smaller than $2\delta$.

\begin{figure}[ht!]
 \labellist
\pinlabel $\Lambda_{\eta,\delta,\zeta}^-\subset \mathcal{V}_{\eta,\delta}$ at 660 185
\pinlabel $\Lambda_{\eta,\delta,\zeta}^+\subset \mathcal{V}_{\eta,\delta}$ at 200 185
\endlabellist
  \centering
  \includegraphics[width=\textwidth]{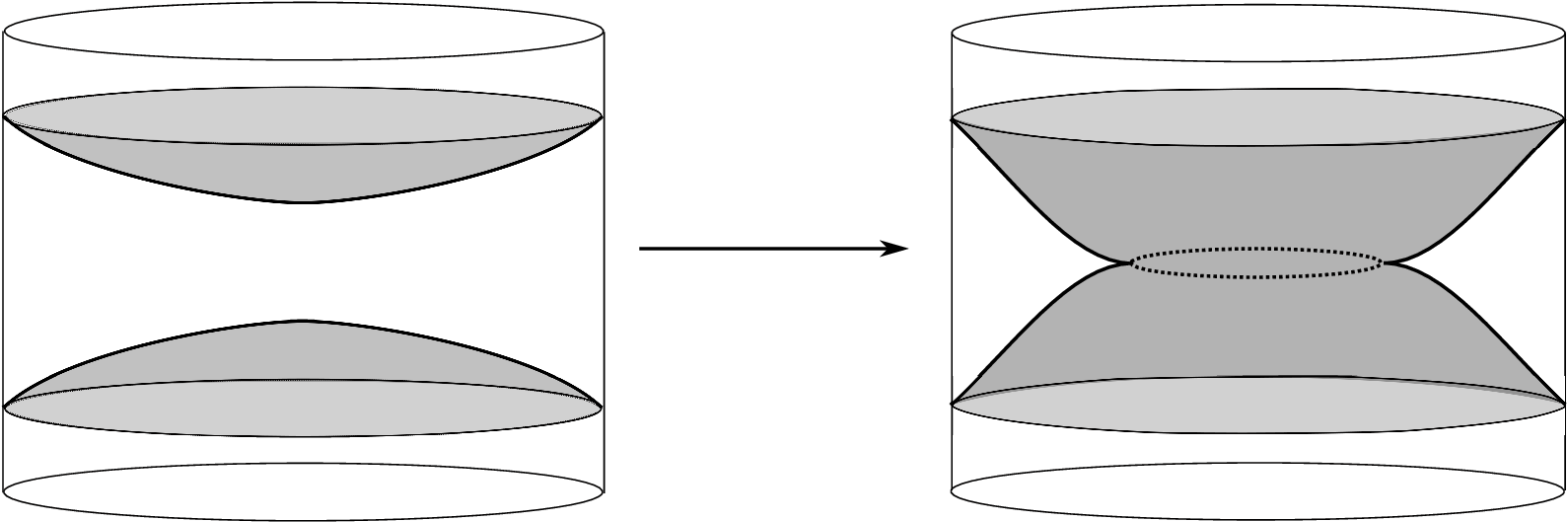}
\vspace{3mm}
  \caption{The front projections of $\Lambda^+$ and $\Lambda^-$.}
  \label{fig:legsurg}
\end{figure}

On Figure \ref{fig:frontlagsur} we see the front and Lagrangian projections of the one dimensionnal version of $\Lambda^+$ and $\Lambda^-$.

\begin{figure}[ht!]
  \centering
\labellist
\pinlabel $\Lambda^+$ at 190 60
\pinlabel $\Lambda^-$ at 660 60
\endlabellist
  \includegraphics[height=5cm]{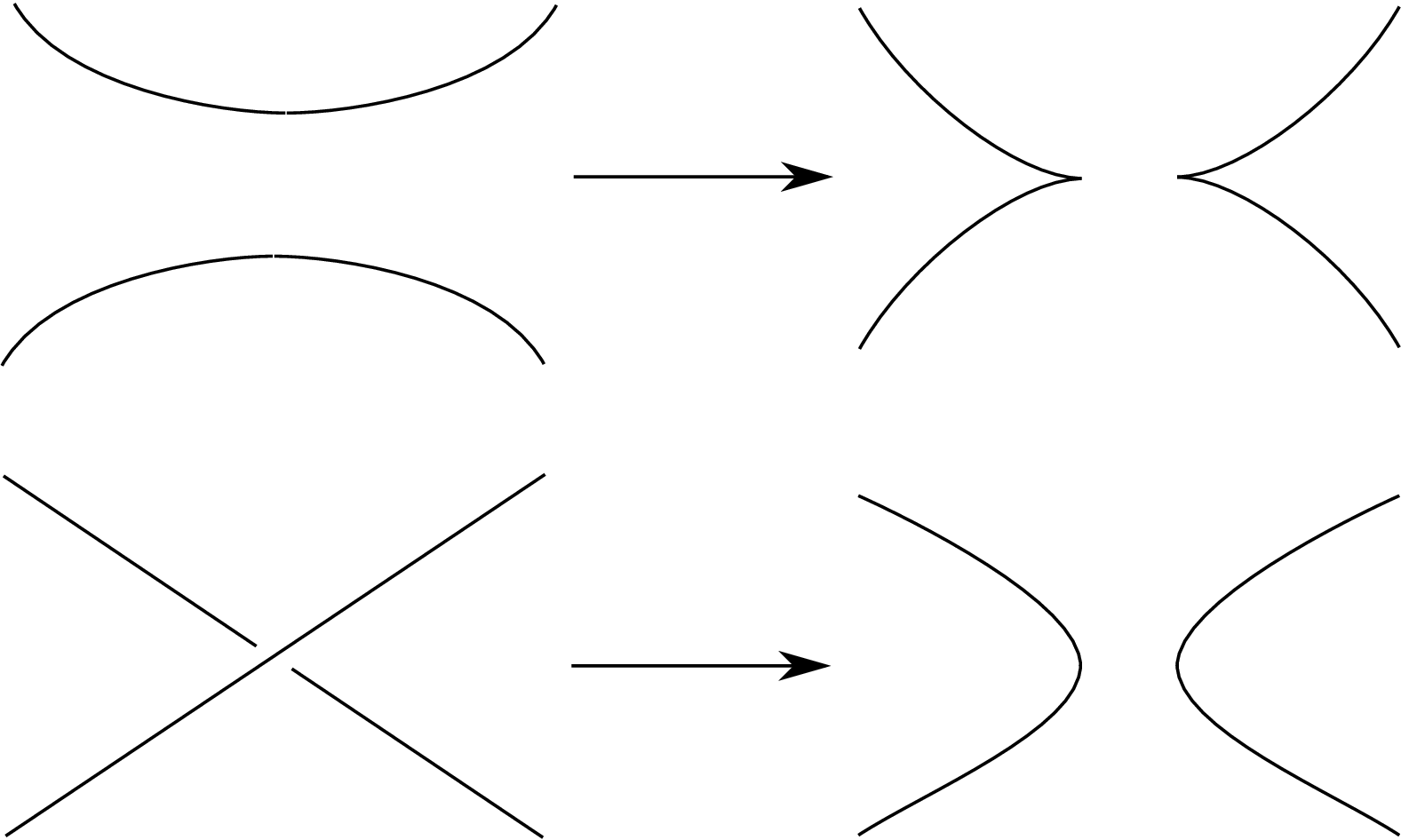}
  \caption{Front (top) and Lagrangian (bottom) projections of the Lagrangian surgery}
  \label{fig:frontlagsur}
\end{figure}

Let $\mathbb L$ be an exact Lagrangian immersion in $(W, \theta)$ with double points $a_1, \ldots,
a_k$, and let  $\mathbb L^+$ be a Legendrian lift of $\mathbb L$. The double points of $\mathbb L$ lift to Reeb chords of
$\mathbb L^+$  which we will denote with the same name by an abuse of notation.

\begin{dfn}
\label{dfn: contractible}
A set of Reeb chords $\{a_1,\ldots,a_k\}$ on $\mathbb {L}^+$ is called {\em contractible} if, for all $i=1, \ldots, k$, there is a neighbourhood $\mathcal{U}_i$ of the Reeb
chord $a_i$ in the contactisation $(M, \beta)$ of $(W, \theta)$ and
a strict contactomorphism $(\mathcal{U}_i,\mathcal{U}_i\cap {\mathbb L^+} )\cong
(\mathcal{V}_{\eta_i,\delta_i},\Lambda^1_{\eta_i,\delta_i,\zeta_i})$ for numbers
$\eta_i,\delta_i,\zeta_i$ satisfying $\zeta_i<\frac{7\delta_i\eta_i}{16}$.
\end{dfn}

\begin{rem}
This is a restrictive assumption because, in general, the lengths of the chords $a_1,\ldots,a_k$ cannot
be modified independently. An example when this is possible, and which will be the case in
our main theorem, is when $\mathbb{L}^+$ is a link with $k+1$ components, all $a_i$ are
mixed chords, and each component contains either the starting point or the end point of at
least one of the $a_i$. In this situation we can indeed modify the Legendrian link by Legendrian isotopies of each his components so that its Lagrangian projection is unchanged and all the previous conditions on the neighbourhoods are satisfied. (Note that this might not be an isotopy of the Legendrian \emph{link}.)
\end{rem}

In the following we assume that $\{a_1,\ldots,a_k\}$ is a set of contractible Reeb chords on $\mathbb {L}^+$. We denote by ${\mathbb  L^+(a_1,\ldots,a_k)}$ the Legendrian submanifold of $(M, \beta)$ obtained by
replacing each of the $\Lambda^+_{\eta_i,\delta_i,\zeta_i}$ by the corresponding
$\Lambda^-_{\eta_i,\delta_i,\zeta_i}$ and by ${\mathbb  L(a_1,\ldots,a_k)}$ the Lagrangian projection
of ${\mathbb  L^+(a_1,\ldots,a_k)}$. Observe that we here need to make use of the identifications with the standard model, which exist by the contractibility condition.

Then ${\mathbb  L(a_1,\ldots,a_k)}$ is an exact Lagrangian immersion in $(W,
\theta)$ which is the result of Lagrangian
surgery on $\mathbb L$ along the self-intersection points $a_1, \ldots, a_k$. It is evident from the
construction that $\mathbb L(a_1,\ldots,a_k)$ coincides with $\mathbb L$ outside of a neighbourhood of the
$a_i$'s and has $k$ self-intersection points removed. The latter fact follows
from the fact that since $\zeta_i$ can be chosen arbitrarily small, no Reeb chords are created when going from $\Lambda^+_{\eta_i,\delta_i,\zeta_i}$ to  $\Lambda^-_{\eta_i,\delta_i,\zeta_i}$.

Next we construct an exact Lagrangian cobordism $\Sigma(a_1, \ldots, a_k)$ in the
symplectisation of $(M, \beta)$ with $\mathbb L$ at the
positive end and $\mathbb L(a_1, \ldots, a_k)$ at the negative end. Fix $T>0$
and choose a function $G \colon (0,\epsilon)\times\mathbb{R}_+\rightarrow\mathbb{R}$
such that:
\begin{itemize}
\item $G(t,s)=g^-(s)$ for $t<1/T$,
 \item $G(t,s)=g^+(s)$ for $t>T$,
\item $\frac{\partial G}{\partial t}(t,0)> 0$, and
\item  $G(t,s)=g^+(s)=g^-(s)$ for $s>3\eta/4$.
\end{itemize}

We consider the Lagrangian submanifold of $T^*(\mathbb{R}^+\times B^n(\eta))$
described by the generating family
$$F(t,q,\xi)=t\cdot\left(\frac{\xi^3}{3}+G(t,|q|)\xi\right),$$
which is mapped by the symplectomorphism $T^*(\mathbb{R}^+\times B^n(\eta))
\cong \mathbb{R} \times\mathcal{J}^1(B^n(\eta))$ to  a Lagrangian
cobordisms $\Sigma_{\eta,\delta,\zeta}$ in the symplectisation of $(M, \beta)$ from
$\Lambda^-_{\eta_k,\delta_k,\zeta_k}$ at the negative end to
$\Lambda^+_{\eta_k,\delta_k,\zeta_k}$ at the positive end. Self-intersections of
$\Sigma_{\eta,\delta,\zeta}$ are given by the critical points of the function
$$\Delta_F(t,q,\xi_1,\xi_2)=F(t,q,\xi_1)-F(t,q,\xi_2)$$
with non $0$ critical value, and such points do not exist because of the third condition on
$G$. Thus this cobordism is embedded.

In the trivial cobordism $\mathbb{R} \times \mathbb L^+$ we replace  $\mathbb{R} \times
(\mathcal{U}_i \cap \mathbb L^+)$ with $\Sigma_{\eta_i,\delta_i,\zeta_i}$, for all $i=1, \ldots, k$,
to get a a cobordism $\Sigma(a_1,\ldots, a_k)$ from $\mathbb L^+(a_1,\ldots,a_k)$ at the
negative end to $\mathbb L^+$ at the positive end.
\subsection{Effect of surgery on Floer cohomology}
\label{sec:cthulhu-compl-surg}
In this subsection we use $\Sigma(a_1, \ldots, a_k)$ and our
Floer theory for Lagrangian cobordisms to relate the Floer cohomology of $\mathbb L$ with the Floer
homology of $\mathbb L(a_1, \ldots, a_k)$.
The Lagrangian cobordism $\Sigma(a_1, \ldots, a_k)$ induces a dga morphism
$$\Phi_{\Sigma} \colon \mathfrak{A}(\mathbb L^+) \to \mathfrak{A}(\mathbb L^+(a_1, \ldots, a_k)).$$
If follows from  \cite[Theorem 1.1]{LegendrianAmbient} that, for a suitable almost complex structure on the cobordism that has been obtained by perturbing an arbitrary cylindrical almost complex structure, we have
\begin{align}\label{dga map}
\begin{array}{ll}
\Phi_{\Sigma}(a_i)  =1& \mbox{for }\ i=1, \ldots, k, \\
\Phi_{\Sigma}(c)  = c+ \mathbf{w}& \mbox{if }\ c \ne a_i, \\
\end{array}
\end{align}
where $\mathbf{w}$ is a linear combination of products $c_1 \ldots c_m$ with
$$\mathfrak{a}(c_1) + \ldots + \mathfrak{a}(c_m) < \mathfrak{a}(c).$$

\begin{lemma} \label{lemma: push-forward}
If $\varepsilon \colon \mathfrak{A}(\mathbb L^+) \to \F$ is an augmentation such that
$\varepsilon(a_i)=1$ for $i=1, \ldots, k$, then there is an augmentation
$\overline{\varepsilon} \colon \mathfrak{A}(\mathbb L^+(a_1, \ldots, a_k)) \to \F$
such that $\varepsilon= \overline{\varepsilon} \circ \Phi_{\Sigma}$.
\end{lemma}
\begin{proof}
Let  $\mathfrak{I}$ be the bilateral ideal generated by $a_i-1, \ldots, a_k-1$: then
$\varepsilon$ induces an augmentation $$\overline{\varepsilon} \colon \mathfrak{A}(\mathbb L^+)/
\mathfrak{I} \to \F.$$ By Equation~\eqref{dga map} $\Phi_{\Sigma}$ is surjective and its
kernel is $\mathfrak{I}$.  Surjectivity is proved by a sort of
Gauss elimination using the action filtration. Then there is an isomorphism
between $\mathfrak{A}(\mathbb L^+)/\mathfrak{I}$ and $\mathfrak{A}(\mathbb L^+(a_1, \ldots, a_k))$, and
therefore the augmentation $\overline{\varepsilon} \colon \mathfrak{A}(\mathbb L^+)/
\mathfrak{I} \to \F$ induces an augmentation on $\mathfrak{A}(\mathbb L^+(a_1, \ldots, a_k))$,
which we still denote by $\overline{\varepsilon}$.
\end{proof}
The construction of $\overline{\varepsilon}$ is not explicit because the isomorphism
$\mathfrak{A}(\mathbb L^+)/ \mathfrak{I} \cong \mathfrak{A}(\mathbb L^+(a_1, \ldots, a_k))$ is not explicit.

\begin{prop}\label{prop_coneiso} For any immersed cylindrical exact Lagrangian
submanifold $T \subset W$ with augmentation $\varepsilon'$
there is a quasi-isomorphism
$$\mathrm{LCC}_{\varepsilon', \overline{\varepsilon}}(T^+, \mathbb L^+(a_1, \ldots, a_k)) \stackrel{\simeq} \longrightarrow \mathrm{LCC}_{\varepsilon', \varepsilon}(T^+, \mathbb L^+)
,$$
under the assumption that the augmentations $\varepsilon$ and $\overline{\varepsilon}$
are as in Lemma \ref{lemma: push-forward}.
\end{prop}
\begin{proof}
We denote by $\Sigma_T$ the trivial cobordism $\Sigma_T = \R \times T^+ \subset \R
\times M$. Recall that the surgery cobordism goes from $\mathbb L^+(a_1, \ldots, a_k)$ to $\mathbb L^+$. Since the surgery is localised in a neighbourhood of the intersection points $a_1,
\ldots, a_k$, by a Hamiltonian isotopy we can assume that $$\Sigma_T \cap \Sigma(a_1, \ldots,
a_k) = \emptyset.$$
Then Corollary \ref{cthulhu isomorphism} implies that the map $d_{+-}$ in the Cthulhu
differential for the cobordisms $\Sigma_T$ and $\Sigma(a_1, \ldots, a_k)$ is a
quasi-isomorphism.
\end{proof}

\begin{lemma}
\label{lem:PushForwardCycle}
Let $\mathbb{L}'$ be an immersed exact Lagrangian submanifold with an augmentation
$\varepsilon'$ and let $\mathbb{L}^+, (\mathbb{L}')^+$ be Legendrian lifts such that
$\mathbb{L}^+$ is above $(\mathbb{L}')^+$.
When Lemma \ref{lemma: push-forward} is applied to an augmentation
$$\varepsilon_c \colon \mathfrak{A}(\mathbb L^+ \cup (\mathbb L')^+) \to \F$$
induced by the cycle
$$c \in \mathrm{CF}((\mathbb L,\varepsilon),(\mathbb L',\varepsilon'))$$
as in Lemma \ref{lem:CyclesAugmentations}, then the push-forward of the augmentation under the DGA morphism
$$\overline{\varepsilon_c}=\overline{\varepsilon}_{\overline{c}} \colon \mathfrak{A}(\mathbb L^+(a_1,\ldots,a_k) \cup (\mathbb L')^+) \to \F$$
is induced by a cycle
$$\overline{c} \in \mathrm{CF}((\mathbb L(a_1,\ldots,a_k),\overline{\varepsilon}),(\mathbb L',\overline{\varepsilon'}))$$
which moreover is mapped to $c$ under the quasi-isomorphism from Proposition \ref{prop_coneiso} .
\end{lemma}
\begin{proof}
There is no Reeb chord starting on either $\mathbb L^+$ or $\mathbb L^+(a_1,\ldots,a_k)$ and ending on $(\mathbb L')^+$, so the pushed-forward augmentation is automatically of the form $\overline{\varepsilon}_{\overline{c}}$. Lemma \ref{lem:CyclesAugmentations} then implies that $\overline{c}$ is a cycle.

The last statement is an algebraic consequence of the fact that the discs counted by the DGA morphism $\Phi_\Sigma$ induced by the surgery cobordism can be identified with the discs counted by the quasi-isomorphism from Proposition \ref{prop_coneiso}.
\end{proof}

Now assume that $\mathbb L'$ is a push off of $\mathbb{L}$ as constructed in Lemma \ref{continuation element}, and let $e \in \mathrm{CF}(\mathbb L, \mathbb L')$ be the ``unit'' defined by the sum of the local minima $e_i$ of the Morse function on the connected components of $\mathbb L$; i.e.\ $e = \sum e_i$.
\begin{cor}
\label{unit after surgery}
The cycle
$$\overline{e} \in \mathrm{LCC}_{\overline{\varepsilon},\varepsilon'}(\mathbb L^+(a_1,\ldots,a_k),(\mathbb L')^+)$$
provided by Lemma \ref{lem:PushForwardCycle} (which is mapped to $e$ under the quasi-isomorphism by Proposition \ref{prop_coneiso}) satisfies the property that
$$\mu^2(\overline{e},\cdot) \colon \mathrm{CF}(T, (\mathbb L(a_1,\ldots,a_k),\overline{\varepsilon})) \to \mathrm{CF}(T, (\mathbb L',\varepsilon'))$$
is a quasi-isomorphism for any exact Lagrangian submanifold $T$ with cylindrical end.
\end{cor}

\begin{proof}
Consider the Legendrian lift $\mathbf{L}^+=\mathbb L^+ \cup (\mathbb L')^+$ such that
$\mathbb{L}^+$ is above $(\mathbb{L}')^+$. Then let the lift $\mathbf{L}^+(a_1, \ldots, a_k)=\mathbb L^+(a_1, \ldots, a_k) \cup (\mathbb L')^+,$ is specified uniquely by the
requirement that it coincides with the first lift outside of a compact subset.

Recall that $e$ is closed by Lemma \ref{unit} and by Lemma \ref{lem:CyclesAugmentations} there is thus an induced augmentation $\varepsilon_e$ of $\mathfrak{A}(\mathbf{L}^+)$. Recall that this augmentation coincides with $\varepsilon$ and $\varepsilon'$ when restricted to the generators on the components $\mathbb L^+$ and $(\mathbb L')^+,$ respectively, while $\varepsilon_e(e_i)=1$ holds for any chord  corresponding to a local minimum and $\varepsilon_e(c)=0$ for every other chord $c$ between $\mathbb{L}^+$ and $(\mathbb{L}')^+$.

Applying Proposition \ref{prop_coneiso} to the Legendrian $\mathbf{L}^+(a_1,\ldots,a_k)$ obtained by surgery on $\mathbf{L}$, yields a quasi-isomorphism
$$ \mathrm{LCC}_{\overline{\varepsilon}_{\overline{e}}}(T^+, \mathbf{L}^+(a_1, \ldots, a_k)) \stackrel{\cong} \longrightarrow \mathrm{LCC}_{\varepsilon_e}(T^+, \mathbf{L}^+).$$
(Here we use that $\overline{\varepsilon}_{\overline{e}}=\overline{\varepsilon_e}$ by Lemma \ref{lem:PushForwardCycle}.) The complex on the right-hand side is acyclic by Lemma \ref{continuation element}, and hence so is the complex on the left-hand side. The sought statement is now a consequence of the straight-forward algebraic fact that the complex
$$\mathrm{LCC}_{\overline{\varepsilon}_{\overline{e}}}(T^+, \mathbf{L}^+(a_1, \ldots, a_k))$$
is equal to the mapping cone of
$$\mu^2(\overline{e},\cdot) \colon \mathrm{CF}(T, (\mathbb L(a_1,\ldots,a_k),\overline{\varepsilon})) \to \mathrm{CF}(T, (\mathbb L',\varepsilon')).$$
\end{proof}

\subsection{Twisted complexes}
\label{sec:twisted-complexes}

The aim of this section is to relate the geometric notion of Lagrangian surgery to the algebraic
notion of twisted complex in the wrapped Fukaya category. We first recall the definition of a
twisted complex in an $A_\infty$-category.

Given a unital $A_\infty$-category $\mathcal{A}$, we describe the category
$\operatorname{Tw}\mathcal{A}$ of twisted complexes over $\mathcal{A}$ and recall its
basic properties. We introduce the following notation: given a number $d$ of matrices $A_i$
with coefficients in the morphism spaces of an  $A_\infty$-algebra, we denote by
$\mu^d_{\mathcal A}(A_d,\ldots,A_1)$  the matrix whose entries are obtained by
applying $\mu^d_{\mathcal A}$ to the entries of the formal product of the $A_i$'s.

\begin{dfn}
  A \textit{twisted complex} over $\mathcal{A}$ is given by the following data:
  \begin{itemize}
  \item a finite collection of objects $L_0, \ldots, L_k$ of $\mathcal{A}$ for some $k$,
\item integers $\kappa_i$ for  $i=0,\ldots, k$, and
\item a matrix $X= (x_{ij})_{0 \le i, j \le k}$ such that
$x_{ij}\in \hom_{\mathcal A}(L_i,L_j)$
and $x_{ij}=0$ if $i \ge j$, which satisfies the {\em Maurer-Cartan equation}
$$\sum_{d=1}^k \mu^d_{\mathcal A}(\underbrace{X,\ldots,X}_\text{$d$ times})=0.$$
  \end{itemize}
\end{dfn}
The integers $\kappa_i$ are degree shifts and are part of the definition only if the morphism
spaces $\hom_{\mathcal A}(L_i, L_j)$ are graded, and otherwise are suppressed.

Given two twisted complexes $\mathfrak{L}=(\{L_i\},\{\kappa_i\}, X)$ and $\mathfrak{L}'=(\{L'_i\},
\{\kappa_i'\}, X')$ we define $$\hom_{\operatorname{Tw}{\mathcal A}}(\mathfrak{L},\mathfrak{L}'):=
\bigoplus_{i,j}\hom_{\mathcal A}(L_i,L_j)[\kappa_i- \kappa_j']$$
and, given $d+1$ twisted complexes $\mathcal{L}_0, \ldots, \mathcal{L}_d$, we define
$A_{\infty}$ operations
$$\mu^d_{\operatorname{Tw}{\mathcal A}} \colon \hom_{\operatorname{Tw}{\mathcal A}}(\mathfrak{L}_{d-1},\mathfrak{L}_d)
\otimes \ldots \otimes \hom_{\operatorname{Tw}{\mathcal A}}(\mathfrak{L}_0,\mathfrak{L}_1) \to
\hom_{\operatorname{Tw}{\mathcal A}} (\mathfrak{L}_0, \mathfrak{L}_d)$$
by
\begin{align}\label{eq:mutwisted}
& \mu^d_{\operatorname{Tw}{\mathcal A}}(q_d,\ldots,q_1)=\nonumber\\
& \sum_{k_1,\ldots, k_d \ge 0}\mu_{\mathcal A}^{k_1+\ldots +k_d+d}(\underbrace{X_d, \ldots,
X_d}_{k_d}, q_d, X_{d-1}, \ldots, X_{1},q_1,\underbrace{X_0\ldots,X_0}_{k_0}).
\end{align}

It is shown in \cite[Section 3.k]{Seidel_Fukaya} that the set
of twisted complexes with operations $\mu^d_{\operatorname{Tw}{\mathcal A}}$ constitutes an $A_\infty$-category
$\operatorname{Tw}\mathcal{A}$ which contains $\mathcal{A}$ as a full subcategory. Furthermore it is
shown in \cite[Lemma 3.32 and Lemma 3.33]{Seidel_Fukaya} that  $\operatorname{Tw}\mathcal{A}$ is
the triangulated envelope of $\mathcal{A}$ and thus $H^0\operatorname{Tw}(\mathcal{A})$ is the derived
category of $\mathcal{A}$.
\begin{dfn} We say that a collection of objects $L_1,\ldots, L_k$ of $\mathcal{A}$ generates
$\mathcal{A}$ if and only if any object $L$ of $\mathcal{A}$ is quasi-isomorphic in
$\operatorname{Tw}\mathcal{A}$ to a twisted complex built from the object $L_i$'s.
\end{dfn}

\begin{lemma}\label{lem:twistedvanishing}
If there is a twisted complex $\mathfrak{L}$ built from $L_0, \ldots, L_k$ such that, for
every object $T$ of $\mathcal{A}$ we have $H \hom_{\operatorname{Tw}{\mathcal A}}(T, \mathfrak{L})=0$,
then $L_0$ is quasi-isomorphic in $\operatorname{Tw}\mathcal{A}$ to a twisted complex built from $L_1, \ldots,
L_k$.
\end{lemma}
\begin{proof}
  This follows from the iterated cone description of twisted complexes
  from \cite[Lemma 3.32]{Seidel_Fukaya}. More precisely, from the
  definition of twisted complexes, for any object $T$ we have that
  $\hom_{\mathcal A}(T,L_0)$ is a quotient complex of $hom_{\operatorname{Tw}{\mathcal A}}(T,\mathfrak{L})$
by the twisted complex $\mathfrak{L}'$ built from $\mathfrak{L}$ starting
  at $L_1$ (i.e. ``chopping'' out $L_0$ from the twisted complex $\mathfrak{L}$), and thus
those three objects fit in  an exact triangle. The vanishing of $H \hom_{\operatorname{Tw}{\mathcal A}}(T,
\mathfrak{L})$  implies then that
$$H\hom_{\mathcal A}(T,L_0)\cong H\hom_{\operatorname{Tw}{\mathcal A}}(T,\mathfrak{L}').$$
The result follows now because the map from $L_0$ to $\mathfrak{L}'$, which is given by
the maps $(x_{0j})$, is a map of twisted complexes.
\end{proof}

We now relate  twisted complexes in the wrapped Fukaya category with certain
augmentations of the Chekanov-Eliashberg algebra of the Legendrian lift of the
involved Lagrangian submanifolds.
\begin{rem}
\label{rem: pre-twisted}
In the following lemma we will make a slight
abuse of notation by building twisted complexes from immersed exact Lagrangian submanifolds: to our knowledge, the wrapped Fukaya category has not yet been extended to include also exact \emph{immersed} Lagrangian submanifolds. However, since the statements and proofs only concern transversely intersecting Lagrangian submanifolds, there are no additional subtleties arising when considering the $A_\infty$ operations. In other words, we only consider morphisms between \emph{different} objects in the category. We can thus think of twisted
complexes in the ``Fukaya pre-category''. Of course if all Lagrangian submanifolds $L_i$ involved are embedded, the statements make sense also in the ordinary wrapped Fukaya category.
\end{rem}

\begin{lemma}\label{augmentations and twisted complexes}
Let $(L_i, \epsilon_i)$,  for $i=0, \ldots, k$, be unobstructed exact immersed Lagrangian submanifolds   which are assumed to be equipped with fixed potentials $f_i$. We denote $\mathbb{L}= L_1 \cup \ldots \cup L_k$ and $\mathbb{L}^+$ its Legendrian lift determined by the given potentials. We assume that $\mathbb{L}^+$ is embedded.

If $\boldsymbol{\varepsilon} \colon \mathfrak{A}(\mathbb{L}^+) \to \F$ is an
augmentation such that:
\begin{enumerate}
\item $\boldsymbol{\varepsilon}(p) = \varepsilon_i(p)$ for every pure chord $p$ of $L_i^+$,
and
\item $\boldsymbol{\varepsilon}(a) =0$ for every mixed chord $a$ from $L_i^+$ to
$L_j^+$ such that $i>j$,
\end{enumerate}
we define
$$x_{ij} := \begin{cases}
\sum \limits_{a \in L_i \cap L_j} \boldsymbol{\varepsilon}(a) a & \text{if } i<j, \\
0 & \text{if } i \ge j,
\end{cases}$$
and $X=(x_{ij})_{0 \le i,j \le k},$ where the double point $a$ is considered as an element in the summand with wrapping parameter $w=0$ (see Section \ref{chain level}).  Then (ignoring the degrees for simplicity) the pair $\mathfrak{L}=
(\{(L_i, \varepsilon_i) \}, X)$ is a twisted complex in the wrapped Fukaya category. Moreover, for any test
Lagrangian submanifold $T$,
$$H\hom_{\operatorname{Tw}{\mathcal WF}}(T, \mathfrak{L})= \mathrm{HW}(T, (\mathbb{L},
\boldsymbol{\varepsilon})).$$
\end{lemma}
\begin{proof}
Denote by $\bs{\varepsilon}_0$ the augmentation of $\mathfrak{A}(\mathbb{L}^i)$ which vanishes on the mixed chords, while taking the value $\varepsilon_i$ on the generators living on the component $L_i$. Recall the chain model for wrapped Floer complex described in Subsection \ref{chain level}, where the homotopy direct limit $\mathrm{CW}((\mathbb {L},\bs{\varepsilon}_0), (\mathbb {L}, \bs{\varepsilon}_0); J_\bullet)$ is an infinite direct sum starting with the term with wrapping parameter $w=0,$ i.e.~the complex
$$ \mathrm{CF}((\mathbb {L}, \bs{\varepsilon}_0), (\mathbb {L}, \bs{\varepsilon}_0)); \boldsymbol{0}, J_\bullet) \oplus \mathrm{CF}((\mathbb {L}, \bs{\varepsilon}_0),(\mathbb {L}, \bs{\varepsilon}_0); \bs{0}, J_\bullet)q.$$
The bounding cochain $X$ can be identified to a sum of elements in the leftmost summand by definition.

Note that $\mathbb {L}$, of course, is  only immersed. However, in the case when it consists of a union of embeddings, it still represents an object in the twisted complexes of the ordinary wrapped Fukaya category; namely, it is the ``direct sum'' of the Lagrangian submanifolds $L_i,$ $i=1,\ldots,m$.

First we prove that $X$ satisfies the Maurer-Cartan equation. The Maurer-Cartan equation
involves a count of holomorphic polygons in moduli spaces
$\mathfrak{M}_{L_0, \ldots, L_d}^0(\mathbf{p}^d, a_0, \mathbf{p}^1, a_1, \ldots,
\mathbf{p}^{d-1}, a_d; J)$ as in Section \ref{CF products}.
On the other hand, the equation $\boldsymbol{\varepsilon} \circ \mathfrak{d}=0$ counts
holomorphic polygons in the moduli spaces $\mathfrak{N}_{\mathbb{L}}(a_0;  \mathbf{p}^1,
a_1, \ldots, a_d, \mathbf{p}^d)$, which are the subset of the previous moduli spaces
consisting of those holomorphic polygons which
satisfy the extra requirement that the intersection
points $a_1, \ldots, a_d$ should be negative punctures (in the sense of Definition
\ref{positive negative}). Condition (2) in the definition of
$\boldsymbol{\varepsilon}$ however implies that $\# \mathfrak{M}_{L_0, \ldots, L_d}^0
(\mathbf{p}^d, a_0, \mathbf{p}^1, a_1, \ldots, \mathbf{p}^{d-1}, a_d; J)$ is multiplied
by a nonzero coefficient only if $a_1, \ldots a_d$ are negative punctures. This proves that $X$
satisfies the Maurer-Cartan equation.

For the second part, note that the differential in $\hom_{\operatorname{Tw}{\mathcal WF}}(T, \mathfrak{L})$
counts the same holomorphic polygons (with Hamiltonian perturbations) as the differential in
$\mathrm{CW}(T, (\mathbb{L}, \boldsymbol{\varepsilon}))$ because the Maurer-Cartan element $X$
involves only elements in the Floer complexes defined with wrapping parameter $w=0,$ and hence vanishing Hamiltonian term.
\end{proof}

The previous lemma together with Proposition \ref{prop_coneiso} implies the following result, which is the main result of this section:

\begin{prop}\label{prop:surgeryTW}
Let $(L_1, \varepsilon_1),\ldots,(L_m, \varepsilon_m)$ be unobstructed immersed exact Lagrangian submanifolds with preferred choices of potentials $f_i$, and let $a_1,\ldots,a_k$ be a set of intersection points lifting to contractible Reeb chords on the induced Legendrian lift $\mathbb {L}^+,$ where $\mathbb {L}:=L_1\cup\ldots\cup L_m$. Assume that there is an augmentation $\boldsymbol{\varepsilon}$ of the Chekanov-Eliashberg algebra of $\mathbb {L}^+$ such that:
\begin{itemize}
\item[(1)] $\boldsymbol{\varepsilon}(c)= \varepsilon_i(c)$ if $c$ is a double point of $L_i$,
\item[(2)] $\boldsymbol{\varepsilon}(a_i)=1$ for $i=1,\ldots,k$, and
\item[(3)] $\boldsymbol{\varepsilon}(q)=0$ if $q \in L_i \cap L_j$ is an intersection point, with $i>j,$ at which $f_i(q)>f_j(q)$ (i.e.~$q$ corresponds to a Reeb chord from $L_i^+$ to $L_j^+$).
\end{itemize}
Then for any other exact Lagrangian submanifold $T$ there is a quasi-isomorphism
$$\mathrm{CW}(T, (\mathbb {L}(a_1,\ldots,a_k),\overline{\boldsymbol{\varepsilon}})) \cong \mathrm{hom}(T,\mathfrak{L}),$$
with $\overline{\boldsymbol{\varepsilon}}$ induced by $\boldsymbol{\varepsilon}$ as in Lemma \ref{lemma: push-forward}, and where $\mathfrak{L}$ is a twisted complex built from the $L_i$ with $i=1,\ldots, m$.

\end{prop}

\begin{rem} Conditions (2) and (3) of
Proposition \ref{prop:surgeryTW} imply that if $a_k$ is an
intersection point between different Lagrangians $L_i$ and $L_j$ for
$i<j$, then $f_i(a_k)<f_j(a_k)$. Conditions (1) and (2) of
Proposition \ref{prop:surgeryTW} imply that if $a_k$ is a
self-intersection point of $L_i$, then augmentation $
\varepsilon_i$ evaluates to $1$ on $a_k$.
\end{rem}
\begin{proof}[Proof of Proposition \ref{prop:surgeryTW}]

We consider the twisted complex $\mathfrak{L}$ built from $L_i$, $i=1,\ldots,m$, that is constructed by an application of Lemma \ref{augmentations and twisted complexes} with the augmentation $\boldsymbol{\varepsilon}$. In other words, the twisted complex is defined using the Maurer-Cartan element
$$X:=a_1+\ldots+a_k \in \mathrm{CF}((\mathbb{L}, \bs{\varepsilon}_0), (\mathbb{L}, \bs{\varepsilon}_0); \bs{0}, J_\bullet) \subset \mathrm{CW}((\mathbb {L}, \bs{\varepsilon}_0), (\mathbb{L}, \bs{\varepsilon}_0);J_\bullet)$$
living in the summand with wrapping parameter $w=0.$ The quasi-isomorphism
$$ \mathrm{hom}(T,\mathfrak{L}) \cong \mathrm{CW}(T,({\mathbb L},\boldsymbol{\varepsilon}))$$
is then a consequence of the same lemma.

What remains is constructing a quasi-isomorphism
$$\mathrm{CW}(T, (\mathbb {L}(a_1,\ldots,a_k),\overline{\boldsymbol{\varepsilon}})) \cong \mathrm{hom}(T,\mathfrak{L})$$
for all test Lagrangian submanifolds $T$. This is done by considering the twisted complex corresponding to the cone of the ``unit'' $\overline{e}$ from Corollary \ref{unit after surgery}. We proceed to give the details.

Let $\mathbb{L}'$ be the push-off of $L_1\cup\ldots \cup L_m$ as considered in Lemma \ref{continuation element}.

Consider the cycle $\overline{e}\in \mathrm{CW}((\mathbb{L}(a_1,\ldots,a_k),\overline{\boldsymbol{\varepsilon}}),(\mathbb{L}',\boldsymbol{\varepsilon}'))$ supplied by Corollary \ref{unit after surgery}. As above, $\overline{e}$ is  an element in the summand
\begin{eqnarray*}
\lefteqn{\mathrm{CF}((\mathbb{L}(a_1,\ldots,a_k),\overline{\boldsymbol{\varepsilon}}),(\mathbb {L}',\boldsymbol{\varepsilon}'); \bs{0}, J_\bullet) } \\
& \subset & \mathrm{CW}((\mathbb{L}(a_1,\ldots,a_k),\overline{\boldsymbol{\varepsilon}}),(\mathbb {L}',\boldsymbol{\varepsilon}')); J_\bullet)
\end{eqnarray*}
with wrapping parameter $w=0$. Then
$$\left(\{(\mathbb{L}(a_1,\ldots,a_k),\overline{\boldsymbol{\varepsilon}}),(\mathbb{L}',\boldsymbol{\varepsilon}')\},\overline{e}\right)$$
is a twisted complex $\mathfrak{L}'$ corresponding to the cone of $\mu^2(\overline{e},\cdot).$ The last part of Corollary \ref{unit after surgery} combined with Lemma \ref{lem:twistedvanishing} then establishes the sought quasi-isomorphism. Indeed, every summand
$$ \mathrm{CF}(T,(\mathbf{L},\boldsymbol{\varepsilon}_{\overline{e}}); w\cdot H,J_\bullet) \subset \mathrm{CW}(T,(\mathbf{L},\boldsymbol{\varepsilon}_{\overline{e}});J_\bullet)$$
in the homotopy direct limit which computes the homology of the cone is acyclic by Corollary \ref{unit after surgery}. (Here $\mathbf{L}=\mathbb{L}(a_1,\ldots,a_k) \cup \mathbb{L}'$ as in the proof of the latter corollary.)

Here we remind the reader that the components of $\mathbb{L}(a_1,\ldots,a_k)$ typically are only immersed, as opposed to embedded. The statement that $$\mathrm{CW}(T, (\mathbb {L}(a_1,\ldots,a_k),\overline{\boldsymbol{\varepsilon}})) \cong \mathrm{hom}(T,\mathfrak{L}),$$
 is hence established on the level of twisted complexes on the pre-category level; c.f.~Remark \ref{rem: pre-twisted}.
\end{proof}

\section{Generating the Wrapped Fukaya category}
\label{sec: generation}
In this section we prove Theorem~\ref{main}.

\subsection{Geometric preparations}
\label{ss: preparation}
Before proving the main theorems we need some geometric preparations which will be used in the technical work of Section \ref{ss: augmentation}.
Recall that the Liouville
form $\theta$ has been modified in order to make
$(\mathcal{H}_1\cup \ldots \cup \mathcal{H}_l,\theta,\mathfrak{f})$
into a union of standard critical Weinstein handles. After adding the
differential of a function supported in a small neighbourhood of
$\mathcal{H}_1 \cup \ldots \cup \mathcal{H}_l,$ we change
the Liouville form once again so that the symplectomorphism between
$(\mathcal{H}_i,d \theta)$ and
$(D_\delta T^* C_i, d \mathbf{p} \wedge d \mathbf{q})$ maps the new
Liouville form $\theta_c$ to $\mathbf{p} d \mathbf{q}$. We make the
modification so that the new Liouville vector field ${\mathcal L}_c$
is still positively transverse to $\partial W_0$, has no zeros outside
$W_0$, and so that the new and old Lagrangian skeleta coincide. (On the other hand ${\mathcal L}_c$ it is no longer a
pseudo-gradient vector field for $\mathfrak{f}$, but this will not
impair the proof of Theorem~\ref{main}.) Note that the above
identification maps the core of a handle to the zero section and the
cocore into a cotangent fibre. Further, we perform the construction of the new Liouville form so that the corresponding Liouville vector field is still everywhere tangent to $D_i.$ The reason for changing $\theta$ to $\theta_c$ is to simplify the arguments of Subsection~\ref{ss: augmentation}.

The set of cylindrical exact Lagrangian submanifolds of $(W, \theta)$ coincide with that of
$(W, \theta_c)$ and the wrapped Floer cohomology between any two such Lagrangian submanifolds is
unaffected by the modification of $\theta$ by the invariance properties of wrapped Floer
homology; see \cite[Section 5]{OpenString}. This means that $\mathcal{WF}(W, \theta)$ is quasi-equivalent to $\mathcal{WF}(W, \theta_c)$.

With a new Liouville vector field we will choose a new function $\mathfrak{r} \colon W \to
[R_0,  + \infty)$ satisfying Conditions (i) and (ii) of Section~\ref{ss: liouville} for $R_0 \ll
0$ such that, on $\partial W_0$, the old and new $\mathfrak{r}$ coincide. From now on,
$\mathfrak{r}$ will always be defined using the new Liouville vector field ${\mathcal L}_c$.
 Later in the proof of Proposition~\ref{generating an object}, we will modify $\mathfrak{r}$ so that the new $R_0 \ll 0$ becomes sufficiently small, while keeping $\mathfrak{r}$ fixed outside of a compact subset.

Let $\psi_t$ be the Liouville flow of $(W, \theta_c)$ and let
$$\widehat{\mathcal H}_i := \bigcup \limits_{t \ge 0} \psi_t({\mathcal H}_i).$$
It follows that $\widehat{\mathcal H}_i \subset W$ are pairwise disjoint, embedded codimension zero manifolds. Moreover, there are exact symplectomorphisms
$$(\widehat{\mathcal H}_i,\theta_c) \cong (T^*C_i,\mathbf{p} d \mathbf{q})$$
with the standard symplectic cotangent bundles.

Recall Conditions (i) and (ii) from Subsection~\ref{ss: liouville}. In particular, $\mathfrak{r}^{-1}(R_0)=W^{sc} \cup \mathcal{H}_1 \cup \ldots \cup \mathcal{H}_l,$ while $\mathfrak{r}|_{\mathfrak{r}^{-1}[R_0+1,+\infty)}$ is a symplectisation coordinate induced by the hypersurface $\mathfrak{r}^{-1}(R_0+1)$ of contact type. In the following we make the further assumption that
\begin{equation}\label{weinstein neighbourhood}
\mathfrak{r}^{-1}(R_0+1) \cap \widehat{\mathcal H}_i=S^*_{r_0}T^*C_i
\end{equation}
for some $r_0>0,$ where the latter radius-$r$ spherical cotangent bundle is induced by the flat metric on $C_i$. This means that
\begin{equation} \label{strange condition}
\mathfrak{r}(\mathbf{p},\mathbf{q})=\log{\|\mathbf{p}\|}- \log{r_0}+R_0+1,\quad
\|\mathbf{p}\| \ge r_0,
\end{equation}
holds in the above canonical coordinates.

Given a point $a \in C_i$ (for some $i=1, \ldots, l$),
we denote by $D_a$ the Lagrangian
plane which satisfies $D_a \cap C_i= \{a \}$ while being everywhere tangent to the Liouville vector field. In particular, $D_a \cap \mathcal{H}_i$ corresponds to the cotangent fibre $D_\delta T^*_aC_i \subset D_\delta T^*C_i$ under the identification ${\mathcal H}_i \cong D_\delta T^*C_i$.
\begin{lemma}\label{moving a bit}
For every $i= 1, \ldots, l$ and $a \in C_i$, the Lagrangian plane $D_a$ is isotopic
to $D_i$ by a cylindrical Hamiltonian isotopy.
\end{lemma}
\begin{proof}
Recall that $(\widehat{\mathcal H}_i, \theta_c)$ is isomorphic to $(T^*C_i, \mathbf{p}
d \mathbf{q})$ as a Liouville manifold and $D_a$ and
$D_i$ correspond to two cotangent fibres. Therefore they are clearly isotopic by a cylindrical
Hamiltonian isotopy.
\end{proof}
In particular, $D_a$ and $D_i$ are isomorphic objects in the wrapped Fukaya category when
$a \in C_i$.

The next lemma is immediate.
\begin{lemma}\label{first preparation}
Let $L \subset W$ be a cylindrical exact Lagrangian submanifold. Then, up to a (compactly
supported) Hamiltonian isotopy, we can assume that $L \cap (C_1 \cup \ldots \cup C_l)=
\{ a_1, \ldots, a_k \}$, the intersections are transverse and $L \cap W^{sc} = \emptyset$.
\end{lemma}
Now we are going to normalise the intersections between $L$ and the planes $D_{a_i}$.
For every $a_i$ we choose the natural symplectomorphism between a neighbourhood
$${\mathcal D}_{a_i} \subset (\widehat{\mathcal H}_i,\theta_c) \cong (T^*C_{i},\mathbf{p}
d\mathbf{q})$$
of $D_{a_i} \cong T^*_{a_i}C_i$ and $(D_\eta T^*D_{a_i},-d\tilde{\mathbf{p}} \wedge d \tilde{\mathbf{q}})$
for some $\eta >0$ small, where $(\tilde{\mathbf{p}}, \tilde{\mathbf{q}})$ are the canonical coordinates on $T^*D_{a_i}$.
It is clearly possible to make this
identification so that
\begin{equation}\label{stranger condition}
\mathfrak{r}(\tilde{\mathbf{p}},\tilde{\mathbf{q}})=\log{\|\tilde{\mathbf{q}}\|}- \log{r_0}+R_0+1, \quad
\|\tilde{\mathbf{q}}\| \ge r_0,
\end{equation}
is satisfied.

We redefine $\mathfrak{r}$ as in Remark \ref{rem: change of r}, without deforming it outside of a compact subset. After making $R_0 \ll 0$ sufficiently small in this manner, we may assume that:
\begin{itemize}
\item $R_0 + k+3 \le 0$,
\item $L \cap W_{R_0 +k+3}$ is the union of $k$ disjoint discs with centres at $a_1, \ldots, a_k$,
and
\item the connected component of $L \cap W_{R_0 +k+3}$ containing $a_i$ is identified inside
${\mathcal D}_{a_i} \cong D_\eta T^*D_{a_i}$ with the graph of the differential of a function
$g_{a_i} \colon D_{a_i} \to \R$ for $i=1, \ldots, k$.
\end{itemize}
Then we modify $L$ by a compactly supported Hamiltonian isotopy so that it satisfies the following
properties:
\begin{itemize}
\item[(L1)] The connected component of $L \cap W_{R_0+k+3}$ containing $a_i$ is contained
inside the Weinstein neighbourhood
$$\mathcal{D}_{a_i} \cap W_{R_0+k+3} \cong D_\eta T^*(D_{a_i} \cap \{\|\tilde{\mathbf{q}}\|
\le e^{k+2}r_0\}),$$
where it is described by the graph of the differential of a function $g_{a_i} \colon D_{a_i} \cap
\{\|\tilde{\mathbf{q}}\| \le e^{k+2}r_0\} \to \R$ with a nondegenerate minimum at $a_i$ and no
other critical points,
\item[(L2)] the connected components of $L \cap W_{R_0+k+3} \setminus W_{R_0+k+2}$ are
cylinders which are disjoint from all the cocores $D_{a_i}$; moreover, these cylinders are
tangent to the Liouville vector field ${\mathcal L}_c$ in the same subset; and
\item[(L3)] $\|g_{a_i} \|_{C^2} \le \epsilon'$ for $i=1, \ldots, k$ and $\epsilon'>0$ small which will be specified in Lemma~\ref{wrapping the discs}.
\end{itemize}
Conditions (L1)---(L3) provides sufficient control of the intersections of $L$ and the Lagrangian skeleton. Later in Lemma~\ref{wrapping the discs} we will use this in order to perform a deformation of the immersed Lagrangian submanifold
$$L \cup D_{a_1} \cup \ldots \cup D_{a_k}$$
by Hamiltonian isotopies applied to the different components $D_{a_i}$. The goal is to obtain an exact Lagrangian immersion admitting a suitable augmentation; the corresponding bounding cochain (see Lemma \ref{augmentations and twisted complexes}) will then give us the twisted complex which exhibits $L$ as an object built out of the different $D_{a_i}$.

\subsection{Proof of Theorem \ref{main}}
In this section we prove Theorem \ref{main} assuming the results of Section~\ref{ss: augmentation}.
The result is a corollary of the following proposition.
\begin{prop}\label{generating an object}
Let $L \subset W$ be an exact Lagrangian submanifold with cylindrical end. If $L \cap W^{\mathrm{sk}} = L \cap (C_1 \cup \ldots \cup C_l)= \{ a_1, \ldots, a_k \}$ and the intersections
are transverse, then $L$ is isomorphic in $\operatorname{Tw} \mathcal{WF}(W, \theta)$ to a twisted complex built from the objects $D_{a_1}, \ldots, D_{a_k}$.
\end{prop}
\begin{proof}
We assume that $L$ satisfies Conditions (L1), (L2) and (L3) from the previous section. Then by Lemma~\ref{wrapping the discs} combined with Lemma~\ref{existence of the
augmentation} there exist Lagrangian planes $D_{a_1}^w, \ldots, D_{a_k}^w$ satisfying the
following properties. First $D_{a_i}^w$ is Hamiltonian isotopic to $D_{a_i}$ (possibly after re-indexing the $a_1,\ldots,a_k$) by a cylindrical
Hamiltonian isotopy supported in $W \setminus W_{R_0+i}$. Second, for an appropriate Legendrian lift $\mathbb{L}^+$ of $\mathbb{L}= L \cup D_{a_1}^w \cup \ldots \cup D_{a_k}^w$ to $(W \times \R, \theta + dz)$ such that the intersection point $a_i$ lifts to a Reeb chord from $(D_{a_i}^w)^+$ to $L^+$ of
length $\epsilon >0$ for $i=1, \ldots, k$ --- see Lemma~\ref{wrapping the discs} for more details --- there exists an augmentation $\boldsymbol{\varepsilon} \colon \mathfrak{A}(\mathbb{L}^+)
\to \F$ for which
\begin{enumerate}
\item $\boldsymbol{\varepsilon}(a_i)=1$ for $i=1, \ldots, k$, and
\item $\boldsymbol{\varepsilon}(d)=0$ if $d$ is a chord from $L^+$ to $(D_{a_i}^w)^+$ for $i=1,
\ldots, k$, or a chord from $(D_{a_i}^w)^+$ to $(D_{a_j}^w)^+$ with $i>j$.
\end{enumerate}
Moreover, using Property (L3) above for $\epsilon'>0$ sufficiently small, it follows that the Reeb chords $a_i$ all are contractible (c.f.~Definition \ref{dfn: contractible}).

By Proposition \ref{prop:surgeryTW} the augmentation $\boldsymbol{\varepsilon}$ induces a twisted complex $\mathfrak{L}$ built from $L_1=D_{a_1}^w, \ldots, L_k=D_{a_k}^w,L_{k+1}=L,$ for which
$$\mathrm{hom}(T,\mathfrak{L}) \cong \mathrm{CW}(T,(\mathbb{L}(a_1,\ldots,a_k),\overline{\boldsymbol{\varepsilon}})).$$
The right-hand side is an acyclic complex by Proposition \ref{prop: trivial triviality}. Using this acyclicity, Proposition \ref{prop:surgeryTW} implies that $L$ is quasi-isomorphic to a twisted complex built from the different $D_{a_i}^w \cong D_{a_i}$ (this last isomorphism follows from the invariance properties for wrapped Floer cohomology under cylindrical Hamiltonian isotopy; see e.g.~\cite[Section 5]{OpenString}).
\end{proof}

We can therefore complete the proof of Theorem \ref{main}:

{\begin{proof}[Proof of Theorem \ref{main}]
Lemma~\ref{first preparation} and Proposition~\ref{generating an object} imply that $L$ is isomorphic to a twisted complex built out of the Lagrangia planes $D_{a_i}$. Lemma~\ref{moving a bit} and the fact that Hamiltonian isotopies generated by cylindrical Hamiltonians induce isomorphisms in the wrapped Fukaya category (see e.g.~\cite[Section 5]{OpenString}) imply that each $D_{a_i}$ is isomorphic to one of the cocore of $W$.
  \end{proof}
}

\subsection{Constructing the augmentation} \label{ss: augmentation}
We start by assuming that the modifications from Section \ref{ss: preparation} have been performed, so that in particular (L1)--(L3) are satisfied. When considering potentials in this subsection, recall that we have modified the Liouville form from $\theta$ to $\theta_c$.

Let $f \colon L \to \R$ be a potential function for $L$. We order the intersection points
$a_1, \ldots, a_k$ such that
$$f(a_k) \le \ldots \le f(a_1).$$
The Morse function $g_{a_i} \colon D_{a_i} \cap \{ \tilde{\mathbf{q}} \le e^{k+2}r_0 \} \to \R$
from (L1) can be assumed to be sufficiently small by (L3), so that $df$ is almost zero inside
$L \cap W_{R_0+k+1}.$

\begin{figure}[h!]
\labellist
\pinlabel $\rho$ at 182 103
\pinlabel $R_0+i$ at 62 115
\pinlabel $R_0+i+1$ at 144 115
\endlabellist
\centering
\includegraphics[height=5cm]{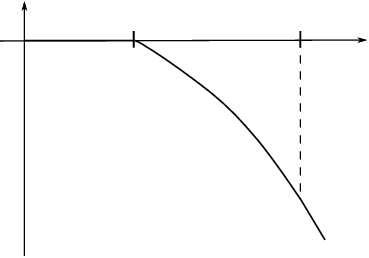}
\caption{The graph of $\mathfrak{h}_i$.}
\label{fig:graphhi}
\end{figure}

We fix functions $\mathfrak{h}_i \colon \R^+ \to \R$ such that
$$\mathfrak{h}_i(\rho) = \begin{cases}
0, & \mbox{if } \rho \le R_0+i,\\
- \rho +  R_0 +i+ \frac 12, & \mbox{if } \rho \ge  R_0+i+1,
\end{cases}$$
and $\mathfrak{h}_i''(\rho) \le 0$ for all $\rho \in \R^+$.
Then we define the cylindrical
Hamiltonians  $H^i \colon W \to \R$, $i=1, \ldots, k$, by
$$H^i(w)= \mathfrak{h}_i(e^{\mathfrak{r}(w)}).$$
The graph of $\mathfrak{h}_i$ appears in Figure \ref{fig:graphhi}.

We will denote by $\phi^i_t$ the flow of the Hamiltonian vector field of $H^i$.
Given $T_i \in \R$, we denote $D_{a_i}^w = \phi^i_{T_i}(D_{a_i})$.

We fix $\epsilon >0$, and on each Lagrangian plane $D_{a_i}$ we choose the
potential function $f_i \colon D_{a_i} \to \R$ such that
$$f_i = f(a_i)+ \epsilon.$$
Note that the functions $f_i$ indeed are constant, since the Liouville vector field is tangent to the planes $D_{a_i}$. Let  $f_i^w \colon D_{a_i}^w \to \R$ be the potential function on
$D_{a_i}^w$ induced by $f_i$ using Equation~\eqref{eq: change of potential}.

We denote by $\mathbb{L}= L \cup D_{a_1}^w \cup \ldots \cup D_{a_k}^w$, which we
regard as an exact Lagrangian immersion, and by $\mathbb{L}^+$ the Legendrian lift of
$\mathbb{L}$ to $(W \times \R, \theta + dz)$ defined using the potential functions
$f, f_1^w, \ldots ,  f_k^w$. Note that an intersection point $d \in D_{a_i}^w
\cap D_{a_j}^w$ lifts to a chord \emph{starting} on $D_{a_i}^w$ and \emph{ending} on $D_{a_j}^w$ if and only if $f_i^w(d)
> f_j^w(d)$, and similarly if one of the two discs is replaced by $L$ and its potential is
replaced by $f$.

 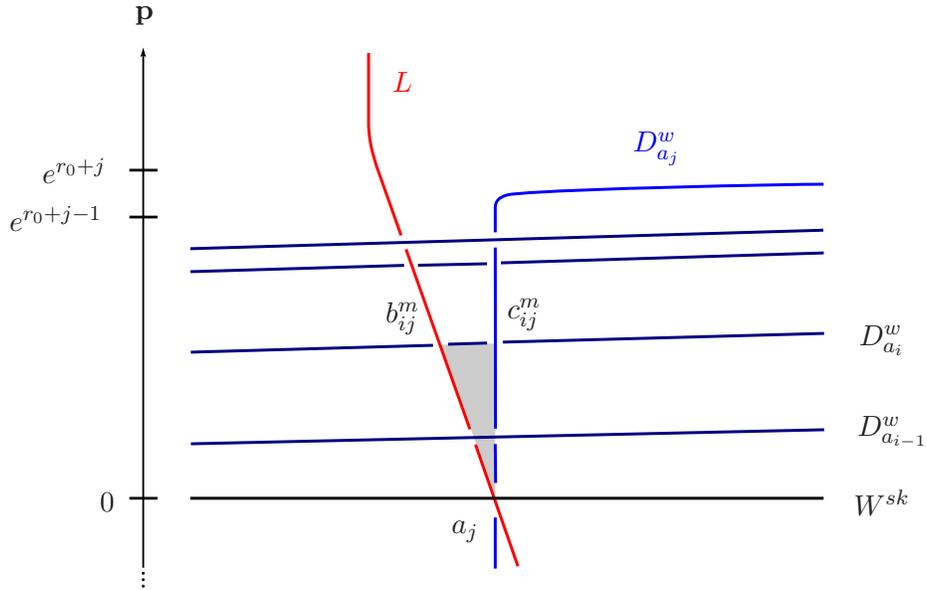
\begin{figure}
\vspace{5mm}
\begin{center}
\labellist
\pinlabel $\mathbf{p}$ at 4 168
\pinlabel $\color{red}L$ at 74 150
\pinlabel $\color{blue}D^w_{a_j}$ at 143 132
 \pinlabel $W^{\mathrm{sk}}$ at 204 36
 \pinlabel $D^w_{a_i}$ at 204 80
 \pinlabel $D^w_{a_{i-1}}$ at 207 55
\pinlabel $b_{ij}^m$ at 74 86
 \pinlabel $c_{ij}^m$ at 107 87
\pinlabel $a_j$ at 91 28
\pinlabel $e^{r_0+j}$ at -15 126
\pinlabel $e^{r_0+j-1}$ at -20 113
\pinlabel $0$ at -6 36
\endlabellist
\includegraphics[scale=1.4]{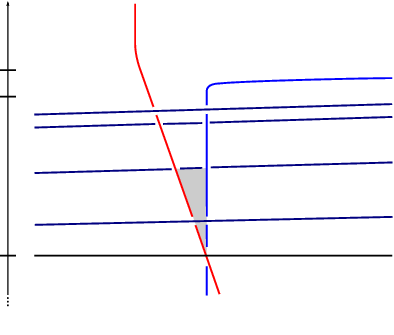}
\caption{A schematic picture of the wrapping and of the small triangle with $i<j$. (Also c.f.~Equation (\ref{strange condition}) combined with Figure \ref{fig:graphhi}.)}
\label{fig: wrapping}
\end{center}
\end{figure}

\begin{lemma}\label{wrapping the discs}
There exist real numbers $0 < T_k < \ldots < T_1$ and $\epsilon, \epsilon' >0$ such that,
if $L$ satisfies (L1)--(L3), then each chord of $\mathbb{L}^+$
is of one of the following types:
\begin{enumerate}
\item type $a$: the chords $a_i$, going from $(D_{a_i}^w)^+$ to $L^+$ for $i=1, \ldots, k$, of length
$\epsilon$,
\item type $b$: chords $b_{ij}^m$ consisting of all other chords from $(D_{a_i}^w)^+$ to $L^+$ for $1 \le i < j \le k$ and
$1 \le m \le m_0(i,j)$ for some $m_0(i,j)$,
\item type $c$: chords $c_{ij}^m$ from $(D_{a_i}^w)^+$ to $(D_{a_j}^w)^+$ for  $1 \le i < j
\le k$ and $1 \le r \le m_0(i,j)$, and
\item ``order-reversing'' type: chords from $L^+$ to $(D_{a_i}^w)^+$ for $i=1, \ldots, k$ or
chords  from $(D_{a_i}^w)^+$ to $(D_{a_j}^w)^+$ for $i=1, \ldots, k$ and $i>j$.
\end{enumerate}
(see Figure~\ref{fig: wrapping}).
Moreover, for every $i<j$ and $m$, there exists a unique rigid and transversely cut out pseudoholomorphic triangle in $W$ having boundary on $L \cup D_{a_i}^w \cup D_{a_j}^w,$ a positive puncture at $b^m_{ij},$ and negative punctures at $a_j$ and $c^m_{ij}$, in the order following the boundary orientation. (Positivity and negativity is determined by our choice of Legendrian lift.)
\end{lemma}
Note that the set $\{c_{ij}^m\}$ could be empty for some $i,j$. In that case, we say that $m_0(i,j)=0$.

\begin{proof}
Recall that Properties (L1)--(L3) from Subsection \ref{ss: preparation} have been made to hold; in particular $L \cap W_{R_0+k+2}$ consists of a $k$ number of discs which may be assumed to be close to the discs $D_{a_i}$, $i=1,\ldots,k$.

The proof of the lemma at hand is easier to see if one starts by Hamiltonian isotoping $L$ to make it \emph{coincide} with $D_{a_1} \cup \ldots \cup D_{a_k}$ inside $W_{R_0+k+2}$. (Thus, we can argue about the intersection points of the deformations $D_{a_i}^w$ and $D_{a_j}$, as opposed to the intersection points of $D_{a_i}^w$ and the different parts of $L$.) By Property (L2) it suffices to deform $L$ in such a way that it becomes the graph $dg_{a_i}$ for a function satisfying $g_{a_i} \equiv 0$ inside the subsets $D_{a_i} \cap W_{R_0+k+2}$. Note that, in the case when $L$ and $D_{a_j} \cap W_{R_0+k+2}$ \emph{coincide}, the intersection points $b^m_{ij}$ and $c^m_{ij}$ coincide as well.

First, we observe that $L, D_{a_1}^w, \ldots, D_{a_k}^w$ are embedded exact Lagrangian
submanifolds, and therefore there is no Reeb chord either from $L^+$ to $L^+$ or from
$(D_{a_i}^w)^+$ to $(D_{a_i}^w)^+$ for any $i=1, \ldots, k$.

From Equation~\eqref{eq: change of potential}, the potential of $D_{a_i}^w$ is
$$f_i^w = f(a_i)+ \epsilon + T_i(\mathfrak{h}_i'(e^{\mathfrak{r}}) e^{\mathfrak{r}} -
\mathfrak{h}_i(e^{\mathfrak{r}})).$$
Note that the quantity $T_i(\mathfrak{h}_i'(e^r)e^r - \mathfrak{h}_i(e^r))$ is nonincreasing
in $r$ because $\mathfrak{h}_i'' \le 0$. Therefore $f_i^w$ satisfies
\begin{align*}
\begin{array}{ll}
f_i^w(w)  = f(a_i) + \epsilon & \text{if } w \in D_{a_i}^w \cap W_{R_0+i}, \\
f_i^w(w)  \in [f(a_i) + \epsilon - T_i(R_0+i+ \frac 12), f(a_i) + \epsilon] & \text{if } w \in  D_{a_i}^w \cap ( W_{R_0+i+1}\setminus  W_{R_0+i}), \\
f_i^w(w)   = f(a_i) + \epsilon - T_i(R_0+i+ \frac 12) & \text{if } w \in D_{a_i}^w \cap
 W_{R_0+i+1}^e.\\
\end{array}
\end{align*}

Note that $D_{a_i}^w \cap W_{R_0+i} = D_{a_i}\cap W_{R_0+i}$ and that $D_{a_i}^w \cap
W_{R_0+i+1}^e$ is a cylinder over a Legendrian submanifold.

We choose positive numbers $ 0<T_k< \ldots< T_1$ such that
\begin{enumerate}
\item $f(a_1)+ \epsilon -T_1(R_0+1+ \frac 12) < \ldots < f(a_k) + \epsilon - T_k (R_0+k+
\frac 12)$,
\item $f(a_i) + \epsilon - T_i (R_0+i+ \frac 12) < \min \limits_{L} f$ for all $i=1, \ldots, k$,
\label{troppo giu}
\item there are no intersection points between $L, D_{a_1}, \ldots, D_{a_k}, D_{a_1}^w, \ldots
D_{a_k}^w$ in their cylindrical parts, and
\item at every intersection point between $L, D_{a_1}, \ldots D_{a_k}, D_{a_1}^w, \ldots
D_{a_k}^w$ the respective potential functions are different, except for intersection points
$p \in L \cap D^w_{a_i}$ where $H^i(p)=\mathfrak{h}_i(e^{\mathfrak{r}})=0$.
\end{enumerate}
The last two conditions are achieved by choosing $T_1, \ldots, T_k$ generically.

We observe that, for any point $c \in D_{a_i}^w \cap D_{a_j}$ and any $i,j = 1, \ldots, k$,
the quantity $\mathfrak{a}(c)= |f_i^w(c)-f_j(c)|$ is independent of $\epsilon$. Then we choose $\epsilon >0$ sufficiently small so that
$$\epsilon < \min \{ \mathfrak{a}(c) : c \in D_{a_i}^w \cap D_{a_j} \text{ and }
\mathfrak{a}(c) \ne 0 \text{ for } i,j = 1, \ldots, k \}.$$
This implies that, for all $c \in D_{a_i}^w \cap D_{a_j}$ such that $\mathfrak{a}(c) \ne 0$,
the signs of $f_i^w(c)- f_j(c)$ and of $f_i^w(c)- f(c)$ are equal. (Recall that $f=f_j-\epsilon$ holds there by construction.)

Consider the set of points $c_{ij}^m \in D_{a_i}^w \cap D_{a_j}$ with positive action difference
$$ 0<f_i^w(c_{ij}^m) - f_j(c_{ij}^m) =f_i^w(c_{ij}^m) - (f(a_j)+\epsilon).$$
(Here $m$ is an index distinguishing the various points with the
required property.) Then $i <j$ and $c_{ij}^m \in W_{R_0+i+1} \setminus W_{R_0+i}$; in
particular $c_{ij}^m \in D^w_{a_i} \cap D^w_{a_j}$. See Figure~\ref{fig: wrapping the discs}.

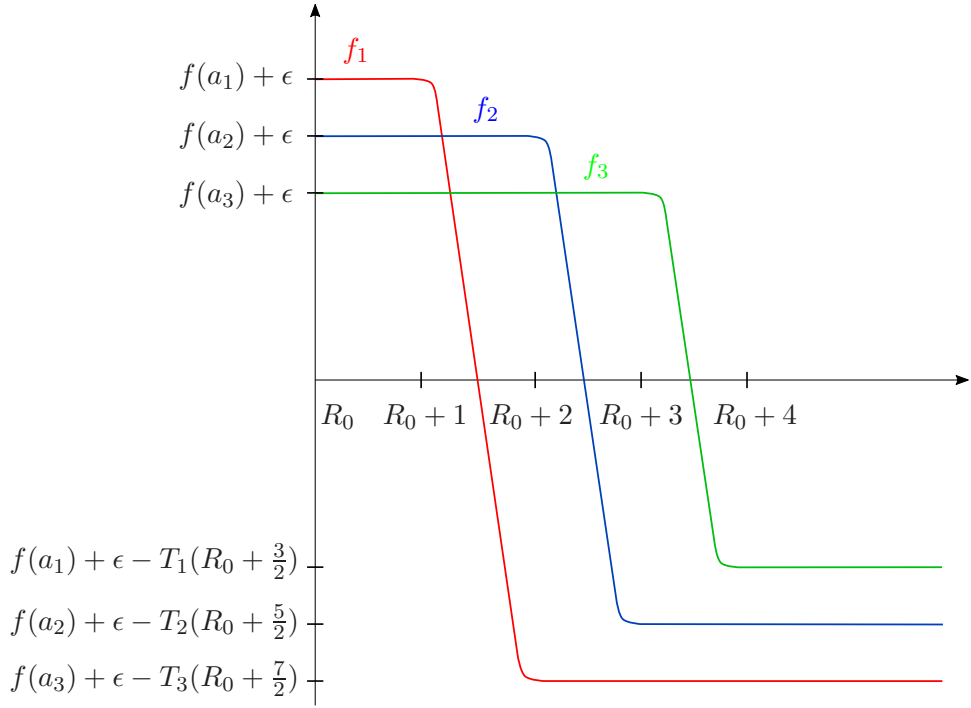
\begin{figure}
\labellist \pinlabel $f(a_1)+\epsilon$  at -35 309 \pinlabel
$f(a_2)+\epsilon$  at -35 280 \pinlabel $f(a_3)+\epsilon$
at -35 251 \pinlabel $f(a_1)+\epsilon-T_1(R_0+\frac{3}{2})$ at
-75 71 \pinlabel $f(a_2)+\epsilon-T_2(R_0+\frac{5}{2})$ at -75
41 \pinlabel $f(a_3)+\epsilon-T_3(R_0+\frac{7}{2})$  at -75 13
\pinlabel $R_0$  at 15 142 \pinlabel $R_0+1$ at 58 142 \pinlabel
$R_0+2$ at 110 142 \pinlabel $R_0+3$ at 164 142 \pinlabel $R_0+4$ at
220 142 \pinlabel $\color{red}f_1$ at 24 322 \pinlabel
$\color{blue}f_2$ at 88 293 \pinlabel $\color{green}f_3$ at 142 265
\endlabellist
\hspace{20mm}
 \includegraphics[width=7.5cm]{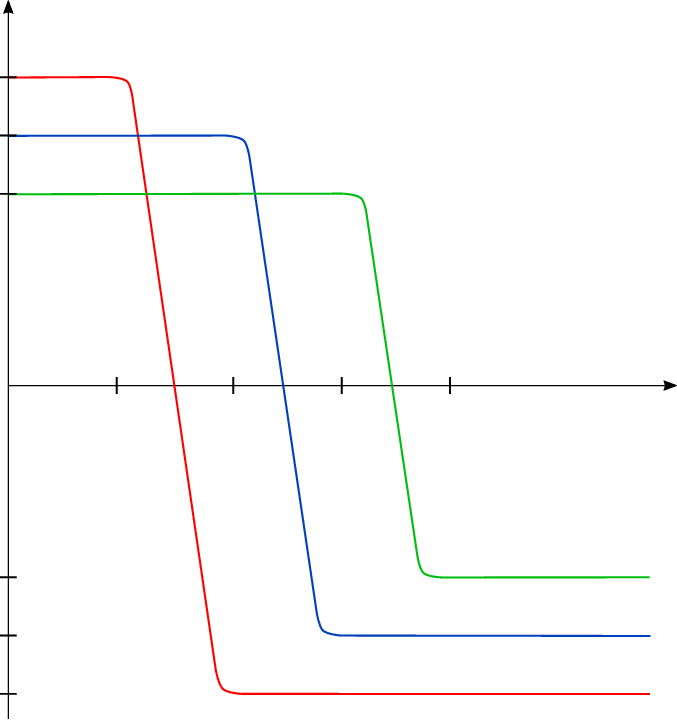}
\caption{The profiles of  $f_i^w$.}
 \label{fig: wrapping the discs}
\end{figure}

The intersection points $b^{m}_{ij}$ now coincide with $c^{m}_{ij}$, but seen as intersections of $L = D_{a_j}$ and $D_{a_i}^w$. We now perturb $L$ back to make it coincide with the graph of $dg_{a_i}$ of a sufficiently small Morse function $g_{a_i}$ near each $D_{a_i}$ having a unique critical point consisting of a global minimum. Recall that this global minimum corresponds to the intersection point $a_i \in L \cap D_{a_i}$.

We can make the Morse function satisfy $\|g_{a_i}\|_{C^2} \le \epsilon'$ for $\epsilon'>0$ sufficiently small. In particular, this means that each intersection point
$c_{ij}^m$ still corresponds to a unique intersection point $b_{ij}^m \in D_{a_i}^w \cap L$, such that moreover
$f_i^w(b_{ij}^m) - f(b_{ij}^m) >0$ is satisfied. Conversely, any intersection point $d \in D_{a_i}^w \cap L$ with $f_i^w(d) - f(d) >0$ is either $a_i$ or one of the $b_{ij}^m$. This is the case
because the only intersection point in $D_{a_i}^w \cap L \cap W_{R_0+i}$ is $a_i$ and for any intersection point $d \in D_{a_i}^w \cap L \cap W_{R_{0}+i+1}^e$ we must have $f_i^w(d) - f(d) < 0$.

The existence of the triangle  follows now by applying Corollary \ref{cor: triangles} to $L \cup D_{a_i}^w \cup D_{a_j}^w$ intersected with the subset $W_{R_0+j} \subset W$. Note that, inside this Liouville subdomain, our deformed Lagrangian $L$ is given as the graph of the differential of a small Morse function $g_{a_j}$ on $D_{a_j}$ (using a Weinstein neighbourhood of the latter); hence the lemma indeed applies. Here the monotonicity property for the symplectic area of a pseudoholomorphic disc can be used in order to deduce that the triangles of interest can be a priori confined to the same Liouville subdomain.
\end{proof}

The triangle provided by the previous lemma is the stepping stone in the inductive construction of an augmentation for $\mathfrak{A}(\mathbb{L}^+)$.
\begin{lemma}\label{existence of the augmentation}
The Chekanov-Eliashberg algebra $(\mathfrak{A}(\mathbb{L}^+), \mathfrak{d})$ of
$\mathbb{L}^+$ admits an augmentation $$\varepsilon \colon \mathfrak{A}(\mathbb{L}^+)
\to \F$$ such that $\varepsilon(a_i)=1$ for all $i=1, \ldots, k$. Moreover, this
augmentation vanishes on the order reversing chords.
\end{lemma}
\begin{proof}
Set $L_i:=D_{a_i}^w$ and $L_{k+1}:=L$. recall that each of the $L_i$ is embedded, and therefore there is no Reeb chord from $L_i^+$ to itself for any $i$. Thus all Reeb chords go between different connected components of $\mathbb{L}^+$ and are as described in Lemma~\ref{wrapping the discs}.

 The bilateral ideal of $\mathfrak{A}(\mathbb{L})$
generated by the order reversing chords is preserved by the
differential, and therefore the quotient algebra, which we will denote by
$\mathfrak{A}^{\to}$, inherits a differential $\mathfrak{d}^{\to}$. We can identify
$\mathfrak{A}^{\to}$ with the subalgebra of $\mathfrak{A}$ \color{black} generated by the chords
of type $a$, $b$ and $c$, and $\mathfrak{d}^{\to}$ to the portion of the differential of
${\mathfrak A}(\mathbb{L})$ involving only generators of $\mathfrak{A}^{\to}$.

On $\mathfrak{A}^\to$ we define a filtration of  algebras
\begin{equation}\label{filtration of algebras}
 \Z = \mathfrak{A}^\to_{k+1} \subset  \mathfrak{A}^\to_k \subset \ldots \subset
\mathfrak{A}^\to_0 = \mathfrak{A}^\to,
\end{equation}
where $\mathfrak{A}^\to_i$ is generated by all chords
$a_s, b_{sj}^m, c_{sj}^m$ with $s \ge i$.

Given a chord $c$ of $\mathbb{L}^+$, we denote its action by
$\mathfrak{a}(c)$. The differential $\mathfrak{d}^\to$ preserves the
action filtration on $\mathfrak{A}^\to$ (and on all its
subalgebras). We assume that
\begin{itemize}
\item[(i)] the actions of all chords $b_{ij}^m$ and
$c_{ij}^m$ are pairwise distinct and,
\item[(ii)] for all $i,j,m$, the actions
$\mathfrak{a}(b_{ij}^m)$ and $\mathfrak{a}(c_{ij}^m)$ are close enough that,
whenever $\mathfrak{a}(c_{i_- j_-}^{m_-}) <  \mathfrak{a}(c_{ij}^{m}) <
\mathfrak{a}(c_{i_+ j_+}^{m_+})$, we also have  $\mathfrak{a}(c_{i_- j_-}^{m_-}) <
\mathfrak{a}(b_{i j}^{m}) < \mathfrak{a}(c_{i_+ j_+}^{m_+})$.
\end{itemize}
The first is a generic assumption, and the second is achieved by choosing
$\epsilon'>0$ sufficiently small in Lemma~\ref{wrapping the discs}.

For each fixed $i$ we define a total order on the pairs $(j,m)$ by declaring that $(h,l) \prec_i (j,m)$ if
$\mathfrak{a}(c_{ih}^l) < \mathfrak{a}(c_{ij}^m)$. When the index $i$ is clear from the context, we will simply write $\prec$.

We know that $\mathfrak{d}^\to a_i=0$ for action reasons and $\langle \mathfrak{d}^\to b_{ij}^m,
a_jc_{ij}^m \rangle =1$ by the last part of Lemma~\ref{wrapping the discs}.
Combining this partial information on the differential $\mathfrak{d}^{\to}$ and the assumptions
(i) and (ii) above with the action filtration, we obtain the following structure for the
differential:
\begin{align*}
\mathfrak{d}^\to a_i & =0, \\
\mathfrak{d}^\to b_{ij}^m & = \alpha_j^m a_i + \sum \limits_{(h,l) \prec_i (j,m)}  \beta_{jl}^{mh} b_{ih}^l
+ a_j c_{ij}^m + \sum \limits_{(l,h) \prec_i (j,m)} w_{lj}^{hm}c_{ih}^l , \\
\mathfrak{d}^\to c_{ij}^m & = \sum \limits_{(l,h) \prec_i (j,m)}  \widetilde{w}_{lj}^{hm}c_{ih}^l
\end{align*}
with $\alpha_j^m, \beta_{jl}^{mh} \in \Z$ and $w_{lj}^{hm}, \widetilde{w}_{lj}^{hm} \in
\mathfrak{A}^\to_{i+1}$.

Then the filtration \eqref{filtration of algebras} is preserved by $\mathfrak{d}^{\to}$.
We want to define an
augmentation $\varepsilon \colon \mathfrak{A}^\to \to \Z$ such that
$\varepsilon(a_i)=1$ for all $i=1, \ldots, k$ working by induction on
$i$.

For $i=k+1$, there is nothing to prove since
$\mathfrak{A}^\to_{k+1} = \Z$.

Suppose now we have defined an augmentation
$\varepsilon \colon \mathfrak{A}^\to_{i+1} \to \Z$. We will extend it
to an augmentation $\varepsilon \colon \mathfrak{A}^\to_i \to \Z$ by
an inductive argument over the action of the chords $c_{ij}^m$. For this reason in the following discussion $i$ will be fixed.

We define $\varepsilon(a_i)=1$ and $\varepsilon(b_{ij}^m)=0$ for all
$j$ and $m$. To define $\varepsilon$ on $c_{ij}^m$ we work inductively
with respect to the order $\prec$ induced by the action.  Suppose that we
have defined $\varepsilon(c_{ih}^l)$ for all $c_{il}^h$ such that
$(h,l) \prec (j,m)$. Then we can achieve
$\varepsilon(\mathfrak{d}^\to b_{ij}^m)=0$ by prescribing an appropriate value to
$$\varepsilon(c_{ij}^m)=\varepsilon(a_jc_{ij}^m),$$
since the values of $\varepsilon$ on all other chords appearing in the expression of $\mathfrak{d}^\to b_{ij}^m$ already have been determined.

Now we have defined $\varepsilon$ on all generators of
$\mathfrak{A}_i^\to$ and, by construction,
$\varepsilon(\mathfrak{d}^\to d)=0$ for every chord $d$ in
$\mathfrak{A}^\to_i$ except possibly for the chords $c_{ij}^m$. We will prove
that in fact $\varepsilon(\mathfrak{d}^\to c_{ij}^m)=0$ holds as well, and thus show that $\varepsilon$ is an augmentation on $\mathfrak{A}^\to_i$. Once again
we will argue by induction on the action of the chords $c_{ij}^m$.

If $(j,m)$ is the minimal element for the order $\prec$, then
$\mathfrak{d}^\to c_{ij}^m=0$ and therefore
$\varepsilon(\mathfrak{d}^\to c_{ij}^k)=0$. Suppose now that we have
verified that $\varepsilon(\mathfrak{d}^\to c_{ih}^l)=0$ for all
$(h,l) \prec (j,m)$. From $\mathfrak{d}^\to (\mathfrak{d}^\to b_{ij}^m)=0$ and
$\varepsilon (\mathfrak{d}^\to ( a_jc_{ij}^m))= \varepsilon(\mathfrak{d}^\to
c_{ij}^m)$ we obtain
$$\varepsilon(\mathfrak{d}^\to c_{ij}^m) + \sum \limits_{(h,l) \prec (j,m)}  \beta_{jl}^{mh}
\varepsilon(\mathfrak{d}^\to b_{ih}^l) + \sum \limits_{(l,h) \prec (j,m)} \varepsilon(\mathfrak{d}^\to(w_{lj}^{hm}c_{ih}^l))=0.$$
We have $\varepsilon(\mathfrak{d}^\to b_{ih}^l)=0$ by construction and
$\varepsilon(\mathfrak{d}^\to(w_{lj}^{hm}c_{ih}^l))=0$ by the induction
hypothesis.
From this we conclude that $\varepsilon(\mathfrak{d}^\to c_{ij}^m)=0$.

Finally we simply precompose $\varepsilon$ with the projection $\mathfrak{A}(\mathbb{L}) \to \mathfrak{A}^{\to}$ and
obtain an augmentation of $\mathfrak{A}(\mathbb{L})$ satisfying the required conditions.
\end{proof}

\section{Generation of the Wrapped Fukaya category of Weinstein sectors.}
\label{sec:wrapp-fukaya-categ}
In this section we prove Theorem \ref{main-sectors}. We recall that the ``linear setup'', introduced by Abouza\"{\i}d and Seidel in \cite{OpenString} and used in the proof of Theorem \ref{main}, is not available for sectors; instead, Ganatra, Pardon and Shende in \cite{GanatraPardonShende} define the wrapped Fukaya category of a Liouville sector by a localisation procedure. However, the strategy of the proof of Theorem \ref{main} applies to
the ``localisation setup'' as well, with only minor modifications of some technical details.
The goal of this section is to explain those modifications, which in most cases will be simplications. Before proceeding, we recall that the proof of Theorem \ref{main} had four main steps:
\begin{enumerate}
\item an extension of the construction of wrapped Floer cohomology to certain exact immersed Lagrangian submanifolds,
\item triviality of wrapped Floer cohomology for immersed exact Lagrangian submanifolds which are disjoint from the skeleton (``trivial triviality''),
\item identification of certain twisted complexes in the wrapped Fukaya category with Lagrangian surgeries, and
\item construction of the bounding cochain after a suitable modification of the Lagrangian cocores.
\end{enumerate}

\subsection{The wrapped Fukaya category for sectors}
In this subsection we recall briefly the definition of the wrapped Floer cohomology and the wrapped Fukaya category for sectors following \cite{GanatraPardonShende} and show that
our construction of the wrapped Floer cohomology of an exact immersed Lagrangian submanifold
can be carried over to this setting.

Given $\epsilon >0$ we, denote $\C_{0 \le \Re < \epsilon}= \{ x+iy \in \C : 0 \le x < \epsilon \}$. If $(S, \theta, I)$ is a Liouville sector, by \cite[Proposition 2.24]{GanatraPardonShende} there is an identification
\begin{equation}\label{near the boundary}
(\mathrm{Nbd}(\partial S), \theta) \cong \left(F \times \C_{0 \le \Re < \epsilon}, \theta_F + \frac 12 (x dy -y dx) + df\right)
\end{equation}
where $f \colon F \times \C_{0 \le \Re < \epsilon} \to \R$ satisfies the following properties:
\begin{itemize}
\item the support of $f$ is contained in $F_0 \times  \C_{0 \le \Re < \epsilon}$ for some Liouville
domain $F_0 \subset F$, and
\item $f$ coincides with $f_{\pm \infty} \colon F \to \R$ for $|y| \gg 0$.
\end{itemize}
We denote $\pi \colon \mathrm{Nbd}(\partial S) \to \C_{0 \le \Re < \epsilon}$ the projection induced by
the identification \eqref{near the boundary}.

We will consider almost complex structures $J$ on $S$ which are cylindrical with respect to the Liouville vector field ${\mathcal L}$ of $\theta$ and make the projection $\pi$ holomorphic (where, of course, we endow $\C$ with its standard complex structure). It is easy to see that this choice of almost complex structures constrains the holomorphic curves with boundary in $\mathrm{int}(S)$ so that they stay away from the boundary $\partial S$; see \cite[Lemma 2.41]{GanatraPardonShende}. Thus, if $L_0, L_1$ are two transversely intersecting exact Lagrangian submanifolds with cylindrical ends, the Floer chain complex {\em with zero Hamiltonian} $\mathrm{CF}(L_0, L_1)$ is defined.

Let $L_\bullet = \{ L_t \}_{t \in I}$ be an isotopy of exact Lagrangian submanifolds which are cylindrical at infinity over Legendrian submanifolds $\Lambda_t$ in the contact manifold $(V, \alpha)$ which is the boundary at infinity of $(S, \theta, I)$. Let $X_t$ be a vector field along $L_t$ directing the isotopy. We can choose this vector field so that, where the isotopy is cylindrical, it is the lift of a vector field along $\Lambda_t$ which we denote by $X_t^\infty$.  We say that the Lagrangian isotopy is {\em positive} if $\alpha(X_t^\infty) \ge 0$ everywhere. We say that the  isotopy is {\em short} if its trace
$\bigcup \limits_{t \in I} \Lambda_t$ is embedded. This implies that Legendrian links $\Lambda_0 \cup \Lambda_t$ are embedded and thus Legendrian isotopic to each other for all $t \in I \setminus \{0\}$.

Following \cite[Subsection 3.3]{GanatraPardonShende}, to any positive isotopy $L_\bullet$ of exact Lagrangian submanifolds with cylindrical ends, we associate a {\em continuation element} $c(L_\bullet) \in \mathrm{HF}(L_1, L_0)$ as follows. If the isotopy is small, there is a map $H^*(L_0) \to \mathrm{HF}(L_1, L_0)$, and we define $c(L_\bullet)$ as the image of the unit in $H^*(L_0)$ under this map. If $L_\bullet$ is not small, then we decompose it into a concatenation of small isotopies and define $c(L_\bullet)$ as the composition (by the triangle product) of the continuation elements of the small isotopies. Then, for any Lagrangian submanifold $K$ which is transverse with both $L_0$ and $L_1$, we define the continuation map
\begin{equation}\label{continuation map for sectors}
\mathrm{HF}(L_0, K) \xrightarrow{[\mu_2(\cdot, c(L_\bullet))]} \mathrm{HF}(L_1, K).
\end{equation}
See \cite[Lemma 3.26]{GanatraPardonShende} for the properties of the continuation element.

Given a Lagrangian submanifold $L$ with cylindrical end, following \cite[Subsection 3.4]{GanatraPardonShende} we consider its {\em wrapping category}
$(L \rightarrow -)^+$, which is the category whose objects are isotopies of Lagrangian submanifolds $\phi \colon L \to L^w$ and morphisms from $(\phi \colon L \to L^w)$ to $(\phi' \colon L \to L^{w'})$ are {\em homotopy classes} of {\em positive} isotopies $\psi \colon L^w \to L^{w'}$ such that $\phi \# \psi= \phi'$.

With all this in place, wrapped Floer cohomology is defined as
\begin{equation}\label{HW for sectors}
\mathrm{HW}(L, K) = \varinjlim \limits_{(L \to L^w)^+} \mathrm{HF}(L^w, K)
\end{equation}
where the maps in the direct system are the continuation maps defined above.

Now suppose that $K$ is immersed and $\varepsilon$ is an augmentation of its obstruction algebra. Then $\mathrm{CF}(L, (K, \epsilon))$ is defined as in Section \ref{sec: Floer homology for immersions}, as long as we use the trivial Hamiltonian $H= \boldsymbol{0}$ in the definition --- being in a Liouville sector makes no difference in any other aspect of the construction.
The definition of $\mathrm{HW}(L, (W, \varepsilon))$ is then the same as in Equation \eqref{HW for sectors} using the product $\mu_2$ defined in Section \ref{CF products}.

\begin{rem}\label{rem: immersions too}
This definition is sufficient for our needs because in the proof of Theorem \ref{main-sectors} we only need wrapped Floer cohomology with immersed Lagrangian submanifolds in the right entry. However, it is possible to extend the definition to the case of immersed Lagrangian submanifolds on the left by identifying augmentations of $L$ with augmentations of $L^w$ and defining the continuation element for small isotopies using Lemmas \ref{continuation element} and \ref{unit}. Note that these lemmas are stated for Floer cohomology with trivial Hamiltonian, and therefore they extend immediately to Liouville sectors.
\end{rem}

Now we sketch the construction of the wrapped Fukaya category following \cite[Subsection 3.5]{GanatraPardonShende}. We recall that we do not need to extend the definition so that it includes immersed Lagrangian submanifolds, even if it would probably not be too difficult. We fix a countable set $I$ of exact Lagrangian submanifolds with cylindrical ends so that any cylindrical Hamiltonian isotopy class has at least one representative and, for every $L \in I$, we fix a a cofinal sequence $L=L^{(0)} \to L^{(1)} \to \ldots$ in $(L \rightarrow -)^+$. We denote by ${\mathcal O}$ the set of all these
Lagrangian submanifolds.  We assume that we have chosen the elements in ${\mathcal O}$ so that all finite subsets $\{ L_1^{i_1}, \ldots, L_k^{i_k} \}$ with $i_1 < \ldots < i_k$ consist of mutually transverse Lagrangian submanifolds.

We make ${\mathcal O}$ into a strictly unital $A_\infty$-category by defining
$$\mathrm{hom}_{\mathcal O} (L^{(i)}, K^{(j)})=
\begin{cases} \mathrm{CF}(L^{(i)}, K^{(j)}) & \text{if } i > j, \\
 \Z & \text{if } L^{(i)}= K^{(j)}, \\  0 & \text{otherwise}.
\end{cases}$$
If $i>j$, the continuation element of the positive isotopy $L^{(j)} \to L^{(i)}$ belongs to $H\hom_{\mathcal O}(L^{(i)}, L^{(j)})= \mathrm{HF}(L^{(i)}, L^{(j)})$. We will write $L^{(i)} > K^{(j)}$ if $i>j$.

We denote by $C$ the set of all closed morphisms of ${\mathcal O}$ which represent a continuation element. Thus we define the wrapped Fukaya category of $(S, \theta, I)$ as ${\mathcal WF}(S, \theta)={\mathcal O}_{C^{-1}}$, where ${\mathcal O}_{C^{-1}}$ is the $A_\infty$-category obtained by dividing $\mathcal{O}$ by all cones of morphisms in $C$: i.e.\ ${\mathcal O}_{C^{-1}}$ has the same objects as ${\mathcal O}$ and its morphisms are defined as the morphisms of the image of ${\mathcal O}$ in the quotient of the triangulated closure of ${\mathcal O}$ by its full subcategory of cones of elements of $C$. This construction has the effect of turning all elements of $C$ into quasi-isomorphisms.
See \cite[Subsection 3.1]{GanatraPardonShende}, and in particular Definition 3.1 therein, for a precise definition of the localisation of an $A_\infty$-category. In the following lemma we summarise the properties of the localisation that we will need.
\begin{lemma}\label{localisation}
The categories ${\mathcal O}$ and ${\mathcal WF}(S, \theta)$ are related as follows:
\begin{enumerate}
\item ${\mathcal WF}(S, \theta)$ and ${\mathcal O}$ have the same objects,
\item $H(\mathrm{hom}_{\mathcal WF}(L, K)) \cong \mathrm{HW}(L, K)$,
\item the category ${\mathcal WF}(S, \theta)$ is independent of all choices up to quasi-equivalence, and
\item The localisation functor ${\mathcal O} \to {\mathcal WF}(S, \theta)$ is the identity on objects and has trivial  higher order terms (i.e.\ it matches $A_\infty$ operations on the nose). Moreover, when $\mathrm{hom}_{\mathcal O}(L, K)= \mathrm{CF}(L, K)$, it induces the natural map $\mathrm{HF}(L, K) \to \mathrm{HW}(L, K)$.
\end{enumerate}
\end{lemma}
\begin{proof}
(1) follows from the definition of localisation. (2) is the statement of \cite[Lemma 3.37]{GanatraPardonShende}. (3) is the statement of \cite[Proposition 3.39]{GanatraPardonShende}. (4) follows from the definition of $\mathcal{A}_\infty$-products in $\mathcal{A}_\infty$-quotients; see \cite[Corollary 2.4]{quotientAinfin}.
\end{proof}
\subsection{Trivial triviality.}
\label{sec:trivial-triviality}
In this section we prove Proposition~\ref{prop: trivial triviality} for sectors. The proof is the same as in Section \ref{ss: pushup} except for few details which must be adjusted because we need to use geometric wrapping of the Lagrangian submanifolds instead of Hamiltonian perturbations of the Floer equation and the continuation maps from Equation \eqref{continuation map for sectors} instead of those from Subsection \ref{sss: changing hamiltonian}.

For Liouville sectors we  need to modify the definition of a cylindrical Hamiltonian in order to have a complete flow. To that aim, the crucial  notion is that of a {\em coconvex set}.
\begin{dfn} \label{decent dynamics}
Let $X$ be a vector field on a manifold $V$. We say that a subset ${\mathcal N} \subset V$ is
{\em coconvex} (for X) if every finite time trajectory of the flow of $X$ with initial
and final point in $V \setminus {\mathcal N}$ is contained in $V \setminus {\mathcal N}$.
\end{dfn}
The following lemma is a rewording of \cite[Proposition 2.34]{GanatraPardonShende}.
\begin{lemma}\label{convexification}
Given a contact manifold $(V, \alpha)$ with convex boundary, it is possible to find a function $g \colon V \to \R_{\ge 0}$
such that
\begin{enumerate}
\item $g > 0$ outside the boundary $\partial V$ and $g \equiv 1$ outside a collar neighbourhood $\partial V \times [0, \delta)$ on which $\alpha = dt+ \beta$ where $\beta$ is a one-form on $\partial V$,
\item $g=t^2G$ on $\partial V \times [0, \delta)$, where $G>0$ and
$t$ is the coordinate of $[0, \delta)$, and
\item there is a collar neighbourhood ${\mathcal N}$ of $\partial V$, contained in $\partial V \times [0, \delta)$, which is coconvex for the contact Hamiltonian $X_g$ of $g$.
\end{enumerate}
\end{lemma}
Note that $X_g$ vanishes along $\partial V$. It is called a {\em cut off Reeb
vector field} in \cite{GanatraPardonShende} because it is the Reeb vector field of the
contact form $g^{-1} \alpha$ on $\mathrm{int}(V)$. From now on we will assume that $g$ and ${\mathcal N}$ are fixed once and for all for the contact manifold $(V, \alpha)$ arising as
boundary at infinity of $(S, \theta, I)$. We will also extend $g$ to the complement of a compact set of $S$ so that it is invariant under the Liouville flow.

\begin{dfn}\label{cylindrical Hamiltonian in sector}
Let $S$ be a Liouville sector. A Hamiltonian function $H \colon [0,1] \times S \to \R$ is
{\em cylindrical} if there is a function $h \colon \R^+ \to \R$ such that
$H(t, w)=g(w)h(e^{\fr(w)})$ outside a compact set of $S$.
\end{dfn}
The definition of cylindrical Hamiltonian compatible with two immersed exact Lagrangian submanifolds in the case of sectors is the same as Definition \ref{compatible hamiltonian}.
Condition (iv) in the latter definition becomes equivalent to asking that $\lambda$ should not be the length of a chord of the cut off Reeb vector field. In this section, cylindrical Hamiltonian will be used to define positive Hamiltonian isotopies of Lagrangian submanifolds and not to deform the Floer equation.

We say that an exact Lagrangian submanifold of $S$ (possibly immersed) is {\em safe} if it  is
cylindrical over a Legendrian submanifold contained in $V \setminus {\mathcal N}$. Since
${\mathcal N}$ is strictly contained in an invariant neighbourhood of $\partial V$,
every cylindrical exact Lagrangian submanifold of $S$ is Hamiltonian isotopic to one which is safe by a cylindrical Hamiltonian isotopy. We will assume that all Lagrangian submanifolds with cylindrical end are safe unless we explicitly state the contrary.

Fix an exact Lagrangian submanifold $L$ in $S$ with cylindrical end. Given $\lambda, \Lambda, R \in \R$ such that $0<\lambda\leq \Lambda$ and $0<R$, let $h_{\lambda, \Lambda, R}$ be the function defined in Equation \eqref{two steps hamiltonian} and consider the cylindrical Hamiltonian $H_{\lambda, \Lambda, R}$ induced by $h_{\lambda, \Lambda, R}$ as in definition \ref{cylindrical Hamiltonian in sector}. Note that, when $\Lambda = \lambda$, we obtain the Hamiltonian function $H_{\lambda}$ induced by the function $h_\lambda$ of Equation \eqref{H_lambda}  independently of $R$. We denote by $L_\bullet^{\lambda,\Lambda,R}= \{ L_t^{\lambda,\Lambda,R} \}_{t \in \R}$ the positive Hamiltonian isotopy generated by $H_{\lambda, \Lambda, R}$ such that $L_0^{\lambda,\Lambda,R}=L$. When $\Lambda= \lambda$ we write $L_\bullet^\lambda$ instead.

\begin{lemma}\label{lambda cofinal}
Let $\lambda_n \to + \infty$ be an increasing sequence. Then the Lagrangian sumbanifolds
$L^{\lambda_n}_1$ form a cofinal collection in $(L \rightarrow -)^+$.
\end{lemma}
\begin{proof}
Since the Hamiltonian $H_\lambda$ is autonomous, for every $\kappa \in \R$ we have $L_1^{\kappa \lambda} = L_{\kappa}^\lambda$. The family $\{ L^\lambda_t \}_{t  \ge 0}$ is cofinal by \cite[Lemma 3.30]{GanatraPardonShende} because the Hamiltonian vector field of $H_\lambda$ in the cylindrical end of $S$ is the lift of a (strictly) positive multiple of the Reeb vector field of the contact form $g^{-1} \alpha$.
\end{proof}

Let $K$ be an immersed exact Lagrangian submanifold with cylindrical end, and let $\varepsilon$ be an augmentation of the obstruction algebra of $K$. Often we will drop $\varepsilon$ from the notation: even if the Floer complex depends on it, the arguments in this subsection do not. Given $\lambda<\Lambda$, we call the intersection points in $L_1^{\lambda, \Lambda, R} \cap K \cap \mathfrak{r}^{-1}((- \infty, R/2))$ {\em intersection points of type I} and the intersection points in $L_1^{\lambda, \Lambda, R} \cap K \cap \mathfrak{r}^{-1}((R/2, + \infty)$ {\em intersection points of type II}. They correspond to the Hamiltonian chords of type I or II in Subsection \ref{ss: pushup}. We denote by $\mathrm{CF}^I(L_1^{\lambda, \Lambda, R}, K)$ the subcomplex of  $\mathrm{CF}(L_1^{\lambda, \Lambda, R}, K)$ generated by the intersection points of type
I. The following lemma is the equivalent of Lemma \ref{too much energy} in this context.
\begin{lemma}\label{too much energy 2}
If the Liouville flow of $(S, \theta)$ displaces $K$ from every compact set, then the inclusion\\
$\mathrm{CF}^I(L_1^{\lambda, \Lambda, R}, K) \hookrightarrow \mathrm{CF}(L_1^{\lambda, \Lambda, R}, K)$ is trivial in homology when $\Lambda$ and $R$ are sufficiently large.
\end{lemma}
\begin{proof}[Sketch of proof]
The proof is the same as that of Lemma \ref{too much energy}, whose main ingredients are Equation \eqref{eq: action of chords}, which computes the action of the generators of the Floer complex, and Lemma \ref{frignolo} which estimates the action shift of the  continuation maps for compactly supported safe isotopies from Subsection \ref{sss: compactly supported isotopies}. Both ingredients are still available for Liouville sectors: in fact intersection points between $L_1^{\lambda, \Lambda, R}$ and $K$ are in bijection with Hamiltonian chords of $H_{\lambda, \Lambda, R}$ and the action of an intersection point is the same as the action of the corresponding chord by Equation \eqref{eq: change of potential} and and Equation \eqref{action of a chord}. Thus \eqref{eq: action of chords} still gives bounds on the action of the generators of  $\mathrm{CF}(L_1^{\lambda, \Lambda, R}, K)$, after taking into account the fact that the extra factor involving $g$ coming from Definition  \ref{cylindrical Hamiltonian in sector} is uniformly bounded because $L_1^{\lambda, \Lambda, R}$ and $K$ are safe.

Moreover, the definition of the continuation maps for compactly supported safe isotopies from Subsection \ref{sss: compactly supported isotopies} and the proof of Lemma \ref{frignolo} do not depend on the Hamiltonian deformation in the Floer equation and therefore, setting $H \equiv 0$ in the Floer equations, they hold also for sectors.
\end{proof}

Nowe we can finish the proof of the equivalent of Proposition \ref{prop: trivial triviality} for sectors.
\begin{prop}\label{prop: sector trivial triviality}
Let $(S, \theta, I)$ be a Liouville sector and let $K$ and $L$ be exact Lagrangian submanifolds of $S$ with cylindrical ends. We allow $K$ to be immersed, and in that case we assume its obstruction algebra admits an augmentation $\varepsilon$. If the Liouville flow of $(S, \theta)$ displaces $K$ from every compact set of $S$, then $\mathrm{HW}(L, (K, \varepsilon))=0$.
\end{prop}
\begin{proof}
For any fixed $\lambda<\Lambda$ there is a natural homotopy class of positive isotopies $L_1^\bullet$ from $L_1^\lambda$ to $L_1^\Lambda$. We need to show that, for $\Lambda$ large enough with respect to $\lambda$, the continuation map associated to this class is trivial.

We represent this class by a concatenation of positive isotopies $L_1^{\lambda, \bullet, R}$ from $L_1^{\lambda}$ to $L_1^{\lambda,\Lambda,R}$ and $L_1^{\bullet, \Lambda, R}$ from $L_1^{\lambda,\Lambda,R}$ to $L_1^\Lambda$ which lead to continuation maps
\begin{align*}
[\mu^2(\cdot,c^{\lambda,\bullet, R})] &  \colon \mathrm{HF}(L_1^{\lambda},K)\rightarrow \mathrm{HF}(L_1^{\lambda,\Lambda,R},K) \\
[\mu^2(\cdot,c^{\bullet, \Lambda, R})] &  \colon \mathrm{HF}(L_1^{\lambda,\Lambda,R},K)\rightarrow \mathrm{HF}(L_1^{\Lambda},K)
\end{align*}
To prove the proposition, it is sufficient to prove that, for any fixed $\lambda$. $[\mu^2(\cdot,c^{\lambda,\bullet, R})]$ is trivial if $\Lambda$ and $R$ are large enough.

It follows from \cite[Lemma 3.27]{GanatraPardonShende} and the definition of $H_{\lambda, \Lambda, R}$ that the map $\mu^2(\cdot,c(L_1^{\lambda, \bullet, R})$ is the natural inclusion $\mathrm{CF}(L_1^\lambda ,K)\subset \mathrm{CF}(L_1^{\lambda,\Lambda,R},K)$ whose image is the subcomplex  $\mathrm{CF}^I(L_1^{\lambda,\Lambda,R},K)$. Thus, the triviality of $[\mu^2(\cdot,c(L_1^{\lambda, \bullet, R})]$ follows from Lemma \ref{too much energy 2}.
\end{proof}

\begin{rem}
In view of Remark \ref{rem: immersions too}, we expect Proposition \ref{prop: sector trivial triviality} to hold also when $L$ is immersed, as long as its obstruction algebra admits an augmentation. However, we haven't checked the details.
\end{rem}

\subsection{Twisted complexes.}
\label{sec:twisted-complexes-1}
In this subsection we extend to sectors the results of Section \ref{sec: surgeries and cones} identifying certain Lagrangian surgeries with twisted complexes.

The first step is to observe that the constructions in Subsecton \ref{ss: cthulhu} can be extended to Lagrangian cobrdisms in the symplectisation of the contactisation of a Liouville sector. In fact, we can work with almost complex structures on $\R \times S \times \R$ satisfying the hypothesis of Lemma \ref{maximum principle} and such that the projection $\R \times S \times \R \to S$ is holomorphic near $\R \times \partial S \times \R$. Then, by \cite[Lemma 2.14]{GanatraPardonShende}, the holomorphic curves appearing in the definition of the Cthulhu complex do not approach $\R \times \partial S \times \R$.

In Subsection \ref{sec:cthulhu-compl-surg} we work only with Floer complexes with trivial Hamiltonian, and therefore the results of that subsection extend to Liouville sectors without effort. Moreover, some contortions which were needed to apply those results in the linear setup are no longer necessary in the localisation setup. Let $L_1, \ldots L_m$ be exact Lagrangian submanifolds with cylindrical ends and denote $\mathbb{L} = L_1 \cup \ldots \cup L_m$. Unlike in Section \ref{sec: surgeries and cones}, here we do not need to consider the case of immersed $L_i$.

We recall some notation from Section \ref{sec: surgeries and cones}. Given a set of intersection points $\{ a_1, \ldots, a_k \}$ which corresponds to a set of contractible chords (see Definition \ref{dfn: contractible}) for a Legendrian lift $\mathbb{L}^+$ of $\mathbb{L}$ to the contactisation of $S$, we denote by $\mathbb{L}(a_1, \ldots, a_k)$ the result of Lagrangian surgery performed on $a_1, \ldots, a_k$ as explained  in \ref{sec:surgery-cobordism}.  If $\mathbb{L}^+$ admits an augmentation $\boldsymbol{\varepsilon}$ such that $\boldsymbol{\varepsilon}(a_i)=1$ for $i=1, \ldots, k$, then by Lemmas \ref{lemma: push-forward} and \ref{prop_coneiso} there exists an augmentation $\overline{\boldsymbol{\varepsilon}}$ of $\mathbb{L}(a_1, \ldots, a_k)^+$ such that, for any exact Lagrangian submanifold $T$ with cylindrical ends, there is an isomorphism
\begin{equation}\label{cthulhu isomorphisms in sectors}
\Phi_* \colon \mathrm{HF}(T, (\mathbb{L}(a_1, \ldots, a_k), \overline{\boldsymbol{\varepsilon}})) \xrightarrow{\cong} \mathrm{HF}(T, (\mathbb{L}, \boldsymbol{\varepsilon})).
\end{equation}
Moreover, this isomorphism preserves the triangle products in the sense that, given two exact Lagrangian submanifolds $T_0$ and $T_1$, the diagram
$$\xymatrix{\mathrm{HF}(T_0, (\mathbb{L}(a_1, \ldots, a_k), \overline{\boldsymbol{\varepsilon}}))\otimes \mathrm{HF}(T_1, T_0) \ar[r]^-{[\mu_2]} \ar[d]_{\Phi \otimes Id} & \mathrm{HF}(T_1, (\mathbb{L}(a_1, \ldots, a_k), \overline{\boldsymbol{\varepsilon}})) \ar[d]^{\Phi} \\
\mathrm{HF}(T_0, (\mathbb{L}, \boldsymbol{\varepsilon}))\otimes \mathrm{HF}(T_1, T_0) \ar[r]^-{[\mu_2]}  & \mathrm{HF}(T_1, (\mathbb{L}, \boldsymbol{\varepsilon}))
}$$
commutes. This is a particular case of \cite[Theorem 2]{legout}. Then the isomorphisms \eqref{cthulhu isomorphisms in sectors} induce isomomorphisms
\begin{equation}\label{surgery isomorphism for HW in sectors}
\mathrm{HW}(T, (\mathbb{L}(a_1, \ldots, a_k), \overline{\boldsymbol{\varepsilon}})) \cong \mathrm{HW}(T, (\mathbb{L}, \boldsymbol{\varepsilon}))
\end{equation}
for every exact Lagrangian submanifold $T$ with cylindrical ends.

Assume that $\boldsymbol{\varepsilon}(q)=0$ for all Reeb chord from $L^+_i$ to $L^+_j$ with $i>j$. We add enough objects to ${\mathcal O}$ so that $L_1, \ldots, L_m$ are objects of ${\mathcal O}$ and $L_m > \ldots > L_1$. By Lemma \ref{localisation}(3) this operation does not change ${\mathcal WF}(S, \theta)$ up to quasi-isomorphism. Then $X=(x_{ij})_{0 \le i,j \le m} \in \bigoplus \limits_{0 \le i,j \le m} \hom_{\mathcal O}(L_j, L_i)$ defined as
$$x_{ij}= \begin{cases} \sum \limits_{a \in L_i \cap L_j} \boldsymbol{\varepsilon}(a)a & \text{if } i<j \\
0  & \text{if } i \ge j
\end{cases}$$
satisfies the Maurer-Cartan equation in ${\mathcal O}$ by the same argument as in the proof of Lemma \ref{augmentations and twisted complexes} and our choice of ordering of the objects $L_1, \ldots, L_m$. Thus $X$ satisfies the Maurer-Cartan equation also in ${\mathcal WF}(S, \theta)$ by Lemma \ref{localisation}(4). We denote by $\mathfrak{L}=(\{ L_i \}, X)$ the corresponding twisted complex both in ${\mathcal O}$ and in ${\mathcal WF}(S, \theta)$. For any object $T$ of ${\mathcal O}$ with $T>L_m$ there is a tautological identification of chain complexes $\hom_{\operatorname{Tw}{\mathcal O}}(T, \mathfrak{L}) = \mathrm{CF}(T, (\mathbb{L}, \boldsymbol{\varepsilon}))$ which, moreover, respects the triangle products; this follows from a direct comparison of the holomorphic  polygons counted by the differentials on the left and on the right as in the proof of Lemma \ref{augmentations and twisted complexes}. Thus the
isomorphisms are compatible with the multiplication by continuation elements and in the limit we obtain a map
$$\mathrm{HW}(T, (\mathbb{L}, \boldsymbol{\varepsilon})) \to H\hom_{\operatorname{Tw}{\mathcal WF}}(T, \mathfrak{L})$$
which is an isomorphism by \cite[Lemma 3.37]{GanatraPardonShende} and a simple spectral sequence argument.
We can summarise these results in the following lemma, which is the analogue of Proposition \ref{prop:surgeryTW} in the context of sectors.

\begin{lemma}\label{thm:twisted_complex_in_sectors}
Let $L_1,\ldots,L_m,$ be embedded exact Lagrangian submanifolds with cylindrical ends. If there exist a Legendrian lift $\mathbb{L}^+$ of $\mathbb{L}=L_1 \cup \ldots \cup L_m$,
an augmentation $\boldsymbol{\varepsilon}$ of the Chekanov-Eliashberg algebra of $\mathbb {L}^+$ and a set of contractible Reeb chord $\{ a_1, \ldots, a_k \}$ such that:
\begin{itemize}
\item[(1)] $\boldsymbol{\varepsilon}(a_i)=1$ for $i=1,\ldots,k$, and
\item[(2)] $\boldsymbol{\varepsilon}(q)=0$ if $q$ is a Reeb chord from $L_i^+$ to $L_j^+$,
\end{itemize}
then there exist a twisted complex $\mathfrak{L}$ built from $L_1, \ldots, L_m$ and an augmentation $\overline{\boldsymbol{\varepsilon}}$ of the Chekanov--Eliashberg algebra of $\mathbb{L}(a_1, \ldots, a_k)^+$ such that, for any other exact Lagrangian submanifold with cylindrical end $T$ there is an isomorphism
$$\mathrm{HW}(T, (\mathbb {L}(a_1,\ldots,a_k),\overline{\boldsymbol{\varepsilon}})) \cong H\hom_{\mathcal WF}(T,\mathfrak{L}).$$
\end{lemma}

\subsection{Construction of the augmentation}
\label{what to do with sectors?}
In this section we finish the proof of Theorem \ref{main-sectors}.
Let $(S, \theta, I, \mathfrak{f})$ be a Weinstein sector. We modify $\theta$ so that it coincides with the standard Liouville form on the handles ${\mathcal H}_i$ and half-handles ${\mathcal H}_i'$.

We denote $\partial_+ {\mathcal H}_i' \cong S_\delta T^*H_i$.
The Reeb vector field on $\partial_+ {\mathcal H}_i$ induced by the canonical Liouville
form is the cogeodesic flow of $H_i$ for the flat metric, and therefore $\partial_+ {\mathcal H}_i' \cap \partial W$ has a coconvex collar. This is an important observation, because it allows us to choose, once and for all, $g$ and a corresponding coconvex collar ${\mathcal N}$ for $\partial V$ as in Definition~\ref{decent dynamics} on the contact manifold $(V, \alpha)$ which is the boundary at infinity of $(S, \theta, I, \mathfrak{f})$ such that $g \equiv 1$ on ${\mathcal H}_i' \setminus {\mathcal N}$ for all Weinstein half-handles  ${\mathcal H}_i'$.

As in Section \ref{sec: generation}, we isotope $L$ so that it is disjoint from the subcritical part of $S^{\mathrm{sk}}$ and  intersects   the cores of the critical Weinstein handles and half-handles
transversely  in a finite number of points $a_i, \ldots, a_k$, and for each point $a_i$ we
consider the cocore plane $D_{a_i}$ passing through it. The wrapping of the cocore planes taking place in the proof of Lemma~\ref{wrapping the discs} is the point where the proof requires a little more work than the case of a Weinstein manifold.

We define the Hamiltonian functions $H^i$ of Lemma~\ref{wrapping the discs} to be
$H^i(w)= g(w)\mathfrak{h}_i(e^{\mathfrak{r}(w)})$ and denote by $D_{a_i}^w$ the image of
$D_{a_i}$ under the Hamiltonian flow of $H^i$ for a sufficiently large time. The flow can push
$D_{a_i}^w$ close to $\partial W$, and therefore we can no longer assume that the wrapped
planes $D_{a_i}^w$ are safe. However, $L$ and the planes $D_{a_i}$ were safe, and therefore all intersection points of type $a$, $b$, and $c$ (defined in Lemma \ref{wrapping the discs}) between $L$  and the planes $D_{a_1}^w, \ldots, D_{a_k}^w$ correspond to Hamiltonian chords contained in the complement of ${\mathcal N}$ because the Hamiltonian vector field of $H^i$ is a negative multiple of the cut off Reeb vector field outside of a compact set. Since $g=1$ outside ${\mathcal N}$, the
energy estimates of Subsection \ref{ss: augmentation} still hold in the case of sectors, and that allows us to construct an augmentation $\boldsymbol{\varepsilon}$ for a suitable Legendrian lift $\mathbb{L}^+$ of $\mathbb{L}=L \cup D_{a_i}^w \cup \ldots \cup D_{a_k}^w$ as in Lemma \ref{existence of the augmentation}. At this point the proof of
Theorem~\ref{main-sectors} proceeds in the same way as the proof of Theorem~\ref{main}. \color{black}

\section{Hochschild homology and symplectic cohomology}
\label{hochschild and symplectic}
In this section we use the work of Ganatra \cite{Ganatra} and Gao \cite{Gao-product} to derive Theorem \ref{thm: HH and SH} from Theorem \ref{main}. Since Ganatra and Gao work in the quadratic setup of the wrapped Fukaya category, we must extend the proof of Theorem \ref{main} to that setup first.

\subsection{Wrapped Floer cohomology in the quadratic setup}
\label{sec:quadratic}
On the level of complexes, wra\-pped Floer cohomology in the quadratic setup is in some sense the simplest one to define; in this case the wrapped Floer complex $\mathrm{CW}(L_0,L_1)$ is the Floer complex $\mathrm{CF}(L_0,L_1;H)$ for a \emph{quadratic} Hamiltonian $H \colon W \to \R,$ by which we mean that $H=C\cdot e^{2\mathfrak{r}}$ is satisfied outside of some compact subset of $W$ for some constant $C > 0.$ This construction of the wrapped Floer complexes can be generalised to the case when $L_1$ has transverse double points in the same manner as in the linear case, by using the obstruction algebra.

In the following we assume that all Lagrangians are cylindrical inside the noncompact cylindrical end $W^e_{R-1} \subset W$ for some $R \gg R_0$. Denote by $\psi_t \colon (W,\theta) \to (W,e^{-t}\theta)$ the Liouville flow of $(W,\theta)$ and recall that $\mathfrak{r}\circ\psi_t=\mathfrak{r}+t$ in $W^e_{R-1}$ and hence $\psi_t(W^e_{r})=W^e_{r+t}$ for all $r \ge R-1$.

\subsection{Trivial triviality}
In this subsection we prove Proposition~\ref{prop: trivial triviality}, i.e.\ ``trivial triviality'', in the quadratic setting. In fact, since the quadratic wrapped Floer cohomology complex does not involve continuation maps the, the proof becomes even simpler here.

Assume that $L_1 \subset W$ is disjoint from the skeleton, and thus that $\psi_t(L_1)$ is a safe exact isotopy that displaces $L_1$ from any given compact subset. The Liouville flow $\psi_t$ is conformally symplectic with the conformal factor $e^t,$ i.e.~$(\psi_t)^*d\theta=e^t\cdot d\theta$. Since $L_1$ is cylindrical outside of a compact subset, the aforementioned safe exact isotopy is generated by a locally defined Hamiltonian function $G_t$ which satisfies the following action estimate.

\begin{lemma}
The locally defined generating Hamiltonian $G_t \colon \psi_t(L_1) \to \R$ from Section \ref{sec: trivial triviality} is of the form $G_t=e^t\cdot G_0\circ \psi^{-t}$ for a function $G_0 \colon W \to \R$ which vanishes outside of a compact set, and thus satisifes the bound $\| G_t\|_\infty \le C_{L_1}e^t$ for some constant $C_{L_1}$ that only depends on $L_1.$ In particular, the Hamiltonian isotopy whose image of $L_1$ at time $t=1$ is equal to $\psi_t(L_1)$ can be generated by a compactly supported $\widetilde{G}_t$ that satisfies $\|\widetilde{G}_t\|_\infty \le tC_{L_1}e^t$.
\end{lemma}
The conformal symplectic property implies that the primitive of $\theta$ satisfies $\theta|_{T\psi_t(L_1)}=e^td(f_1\circ \psi_t)$ for the primitive $f_1\colon L_1\to \R$ of $\theta|_{TL_1}$ that vanishes in the cylindrical end. The formula for the action of the Reeb chords \eqref{action of a chord}, in particular c.f.~\eqref{action} in the example thereafter applied with $h(x)=x^2$, now readily implies that:
\begin{lemma}
The generators of $\mathrm{CF}(L_0,\psi_t(L_1);H)$ have action bounded from below by $C'_{L_1}e^{2t}$ for some constant $C'_{L_1} >0$ whenever $t \gg 0$ is sufficiently large.
\end{lemma}

A construction using moving boundary conditions (as in ~Section \ref{sss: compactly supported isotopies}) yields quasi-isomorphisms
$$\Psi_{G_t} \colon \mathrm{CF}(L_0,L_1;H) \to \mathrm{CF}(L_0,\psi_t(L_1);H).$$
By an action estimate (c.f.~Lemma \ref{frignolo}) in conjunction with the above two lemmas,
 any subcomplex of $\colon \mathrm{CF}(L_0,L_1;H)$ that is spanned by the generators below some given action level is contained in the kernel of $\Psi_{G_t}$ whenever $t \gg 0$ is taken sufficiently large.
In conclusion, we obtain our sought result:
\begin{prop}
If $\psi_t(L_1)$ displaces $L_1$ from any given compact subset, then $\mathrm{HW}(L_0,L_1)=0$.
\end{prop}

\subsection{Twisted complexes and surgery formula }

In order to obtain the surgery formula in Proposition \ref{prop:surgeryTW} in the quadratic setting one needs to take additional care. The reason is that, since the products and higher $A_\infty$-operations in this case are defined using a trick that involves rescaling by the Liouville flow, it is a priori not so clear how to relate these operations to operations defined by counts of ordinary pseudoholomorphic polygons, as in the differential of the Chekanov--Eliashberg algebra. (Recall that our surgery formula is proven by identifying bounding cochains with augmentations for the Chekanov--Eliashberg algebra of an exact Lagrangian immersion.) Our solution to this problem is to amend the construction of the $A_\infty$-structure to yield a quasi-isomorphic version, for which the compact part of the Weinstein manifold is left invariant by the rescaling (while in the cylindrical end we still apply the Liouville flow, as is necessary for compactness issues). The upshot is that the new $A_\infty$-structure is given by counts of ordinary pseudoholomorphic discs (together with small Hamiltonian perturbations) inside the compact part.

Here we show how to modify the definition of the $A_\infty$-operations of the wrapped Fukaya category in a way that allows us to apply our strategy for proving the surgery formula. It will follow that the cones constructed can be quasi-isomorphically identified with cones in the original formulation of the quadratic wrapped Fukaya category, which is sufficient for establishing the generation result.

In order to modify the definition of the $A_\infty$-operations we begin by constructing a family $\psi_{s,t} \colon W \to W$ of diffeomorphisms parametrised by $t \ge 0$ and $s \in [0,1]$ that satisfies the following:
\begin{itemize}
\item $\psi_{s,t}|_{W_{R-1}}=\psi_{st}$ inside $W_{R-1}$;
\item $\psi_{s,t}|_{\partial W_r}=\psi_{(s+(1-s)\beta(r))t}$ for any $r \in [R-1,R]$; and
\item $\psi_{s,t}|_{W^e_{R}}=\psi_t$ for any $r \ge R$,
\end{itemize}
where $\beta(r)$ is a smooth  cut off function that satifies
\begin{itemize}
\item $\beta(r)=0$ near $r=R-1$;
\item $\beta(r)=1$ near $r=R$; and
\item $\partial_r\beta(r) \ge 0.$
\end{itemize}
The above flow $\psi_{s,t}$ is a conformal symplectomorphism only when $s=1$, in which case $\psi_{1,t}=\psi_t$ is the Liouville flow. For general values of $s$ it is the case that $\psi_{s,t}$ is a conformal symplectomorphism only outside of a compact subset. Notwithstanding, it is the case that:
\begin{lemma}
  \label{lem:almostconformal}
\begin{enumerate}
\item If $L$ is an exact Lagrangian immersion that is cylindrical inside $W^e_{R-1},$ it follows that $L_\tau := \psi_{s(\tau),t(\tau)}^{-1}(L)$ is a safe exact isotopy, which moreover is fixed setwise inside $W^e_{R-1}$;
\item Any compatible almost complex structure $J_t$ which is cylindrical inside $W^e_{R-1+st}$ satisfies the property that $\psi_{s,t}^*J_t$ is a compatible almost complex structure which is cylindrical inside $W^e_R$ and, moreover, equal to $J_t$ inside $W^e_{R-1+st}$; and
\item Conjugation $\psi_{s_0,t_0}^{-1} \circ \varphi_t^H \circ \psi_{s_0,t_0}$ with the diffeomorphism $\psi_{s_0,t_0}$ induces a bijective correspondance between Hamiltonian isotopies of $W$ generated by Hamiltonians which depend only on $\mathfrak{r}$ inside $W^e_{R-1+s_0t_0}$ and Hamiltonian isotopies of $W$ generated by Hamiltonians which only depend on $\mathfrak{r}$ inside $W^e_{R-1}.$ More precisely, if the former Hamiltonian is given by $H \colon W \to \R$ then the latter is given by $(f_{s_0,t_0} \cdot H) \circ \psi_{s_0,t_0}$ for a smooth function $f_{s_0,t_0} \colon W \to \R_{>0}$ which only depends on $\mathfrak{r}$ inside $W^e_{R-1+s_0t_0},$ while $f_{s_0,t_0}|_{W_{R-1+s_0t_0}}\equiv e^{-s_0t_0}.$
\end{enumerate}
\end{lemma}

\begin{rem}
The function $f_{s_0,t_0} \colon W \to \R_{>0}$ in Part (3) of the previous lemma can be determined as follows. First note that $f_{1,t} \equiv e^{-t}$, while for general $s \in [0,1]$ the equality $f_{s,t} \equiv e^{-t}$ still holds outside of a compact subset. Inside $W^e_{R-1+s_0t_0}$ the simple ordinary differential equation
$$e^{-\mathfrak{r}}\partial_{\mathfrak{r}}((f_{s_0,t_0}(\mathfrak{r}) \cdot H)\circ \psi_{s_0,t_0})=(e^{-\mathfrak{r}}\partial_{\mathfrak{r}}H(\mathfrak{r}))\circ \psi_{s_0,t_0}$$
then determines $f_{s,t}.$
\end{rem}

By the properties described in Lemma \ref{lem:almostconformal}, it follows that we can use $\psi_{s_0,t}$ for any fixed $s_0$ instead of the Liouville flow $\psi_t$ in Abouzaid's construction of the wrapped Fukaya category \cite{OnWrapped}. We illustrate this in the case of the product $\mu^2$. One first defines a morphism
\begin{align}\label{we're almost done}
  & \mathrm{CF}(L_1,L_2;H,J_t) \otimes \mathrm{CF}(L_0,L_1;H,J_t) \to \\
  \nonumber \mathrm{CF} & \left((\psi_{s_0,\log 2})^{-1}(L_0), (\psi_{s_0,\log 2})^{-1}(L_2);(f_{s_0,\log 2}\cdot H)\circ\psi_{s_0,\log 2},\psi_{s_0,\log 2}^*J_t \right)
\end{align}
for any fixed $s_0 \in [0,1]$ that is defined by a count of three-punctured discs with a suitable moving boundary condition, and where
$$f_{s_0,\log 2} \colon W \to \R_{>0}$$
is the function from Part (2) of Lemma \ref{lem:almostconformal}. (In particular, $$f_{s_0,\log 2}\equiv e^{-\log 2}$$ holds outside of a compact subset.)
\begin{rem}
The fact that $(f_{s_0,\log w}(\mathfrak{r})\cdot H)\circ \psi_{s_0,\log w}$ coincides with $\frac{1}{w}e^{2(t+\log w)}=we^{2t}$ outside of a compact subset is crucial for the maximum principle (and thus compactness properties) of the Floer curves involved in the definition of the morphism of Equation \eqref{we're almost done}.
\end{rem}
It is immediate that an analogous version of \cite[Lemma 3.4]{OnWrapped} now also holds in the present setting, giving rise to a canonical identification between the Floer complexes
$$\mathrm{CF}\left((\psi_{s_0,\log 2})^{-1}(L_0),(\psi_{s_0,\log 2})^{-1}(L_2);(f_{s_0,\log 2}\cdot H)\circ\psi_{s_0,\log 2},\psi_{s_0,\log 2}^*J_t \right)$$
and $$\mathrm{CF}(L_1,L_2;H,J_t).$$
In this manner we obtain the operation $\mu^2.$ The general case follows similarly, by an adaptation of the construction \cite{OnWrapped} based upon $\psi_{s_0,t}.$

Finally, to compare the $A_\infty$ structures defined by different values of the paramter $s_0 \in [0,1]$, one can use \cite{Seidel_Fukaya}. Here Part (1) of Lemma \ref{lem:almostconformal} is crucial. Also, note that the family $(f_{s_0,\log w}\cdot H)\circ\psi_{s_0,\log w}$ of Hamiltonians is independent of the paramter $s_0$ outside of a compact subset for any fixed value of $w$.

Finally, since the intersection points of type $a, b,c$ of Lemma \ref{wrapping the discs} belong to the region where we have turned off the rescaling by the Liouville flow, the construction of the augmentation in Lemma \ref{existence of the augmentation} remains unchanged. From this point, the proof of Theorem \ref{main} in the therefore quadratic setup procedes as in the linear setup.

\subsection{Proof of Theorem \ref{thm: HH and SH}}
The proof of Theorem \ref{thm: HH and SH} is based on Theorem \ref{main} and the following trivial observation. If $(W, \theta, \mathfrak{f})$ is a Weinstein manifold and $\pi_i \colon W \times W \to W$, for $i=1,2$, are the projections to the factors, we consider the Weinstein manifold $(W \times W,
\pi_1^* \theta - \pi_2^* \theta, \mathfrak{f} \circ \pi_1 + \mathfrak{f} \circ \pi_2)$.
Note that the sign in the Liouville form has been chosen so that the diagonal $\Delta \subset W \times W$ is an exact Lagrangian submanifold with cylindrical end. The following lemma is a direct consequence of the definition.
\begin{lemma}
\label{lma:product}
The Weinstein manifold $(W \times W, \pi_1^* \theta - \pi_2^* \theta, \mathfrak{f} \circ \pi_1 + \mathfrak{f} \circ \pi_2)$ has a Weinstein handle decomposition for which the Lagrangian cocores are precisely the products of the cocores $D_i \times D_j$, where $D_i$ denotes a Lagrangian cocore in the Weinstein decomposition of $(W,\theta,\mathfrak{f}).$
\end{lemma}

 Let $\mathcal{W}^2$ be the version of the wrapped Fukaya category for the product Liouville manifold $(W \times W, \pi_1^*\theta - \pi_2^*\theta)$ defined in \cite{Ganatra},  where the wrapping is performed by a split Hamiltonian. If $\mathcal{B}$ is a full subcategory of $\mathcal{WF}(W,\theta)$, we denote by $\mathcal{B}^2$ the full subcategory of $\mathcal{W}^2$ whose objects are products of objects in $\mathcal{B}$. Then \cite[Proposition 14.1]{Ganatra} implies the following.
\begin{prop}\label{prop:diaggenerate}
  If $\Delta$ is generated by $\mathcal{B}^2$ in $\mathcal{W}^2$, then the map
\begin{equation}\label{open-closed map}
[\mathcal{OC}] \colon HH_{n-*}(\mathcal{B},\mathcal{B}) \to SH^*(W)
\end{equation}
has the unit in the symplectic cohomology in its image.
\end{prop}

\begin{proof}[Proof of Theorem \ref{thm: HH and SH}]
Let $(W, \theta, \mathfrak{f})$ be a Weinstein manifold and let ${\mathcal D}$ be the colection of the Lagrangian cocore discs of $W$. Then, by Lemma \ref{lma:product} and Theorem \ref{main} the collection ${\mathcal D}^2$ of products of cocore discs of $W$ generates ${\mathcal WF}(W \times W, \pi_1^*\theta - \pi_2^*\theta)$ and so, in particular,
generates the diagonal $\Delta$. By \cite[Theorem 1.1]{Gao-product} the category ${\mathcal WF}(W \times W, \pi_1^*\theta - \pi_2^*\theta)$ is equivalent to the category ${\mathcal W}^2$ defined above, and therefore the collection ${\mathcal D}^2$ generates the diagonal also in ${\mathcal W}^2$. Thus Proposition\ref{prop:diaggenerate} implies that the image of the open-closed map \eqref{open-closed map} contains the unit and therefore Theorem \ref{thm: HH and SH} follows from \cite[Theorem 1.1]{Ganatra}.
\end{proof}
\bibliographystyle{plain}
\bibliography{Bibliographie_en}
\end{document}